\documentclass[fleqn,preprint,12pt]{elsarticle}
\usepackage{float}
\usepackage{graphicx,color}
\usepackage{epsfig,amssymb,amsthm}
\usepackage{amsmath}
\numberwithin{equation}{section}  
\usepackage[normalem]{ulem}
\graphicspath{{figures_pdf/}}
\usepackage{booktabs}
\usepackage{hyperref}
\journal{Elsevier}
\usepackage{empheq}
\AtBeginDocument{}
\usepackage{amsmath}
\usepackage{listings}

\usepackage{enumitem}
\usepackage{xcolor}
\usepackage{multirow}
\setenumerate[1]{itemsep=0pt,partopsep=0pt,parsep=\parskip,topsep=6pt}
\setitemize[1]{itemsep=0pt,partopsep=0pt,parsep=\parskip,topsep=6pt}
\setdescription{itemsep=0pt,partopsep=0pt,parsep=\parskip,topsep=6pt}

\newtheorem{thm}{Theorem}[section]  
\newtheorem{corollary}[thm]{Corollary}      
\newtheorem{lemma}[thm]{Lemma}              
\theoremstyle{remark}
\newtheorem{remark}{Remark}
\usepackage{siunitx}
\usepackage{multirow,bigstrut}
\usepackage{geometry}
\biboptions{numbers,sort&compress}
\usepackage{diagbox}
\usepackage[aboveskip=1pt]{subcaption}
\newlength\OneImW
\begin{document}
	
	\begin{frontmatter}
		
		\title{Complete solution to the weakly persistent center and isochronous center problems of cubic polynomial systems}

\author{Feng Li$^a$, Huimei Liu$^a$, Chen Cao$^b$, Pei Yu$^{c,}$\footnote{Corresponding author.
				\hspace*{0.15in}E-mail: lfeng@lyu.edu.cn (F. Li), 240701001304@lyu.edu.cn (H. Liu), aceaccce@sjtu.edu.cn(C. Chen),pyu@uwo.ca (P. Yu)}}
		
		\address{
			\small\it
			\vspace{0.2cm}\small\it  $^a$School of Mathematics and Statistics, 
			Linyi University, \\[-0.7ex]
			\vspace{0.2cm}\small\it   Linyi, 276005,  China\\
			\small\it
			\vspace{0.2cm}\small\it  $^b$School of Mathematical Sciences, CMA-Shanghai,
			Shanghai Jiao Tong University, \\[-0.7ex]
			\vspace{0.2cm}\small\it   Shanghai, 200240, China \\
			\small\it
			\vspace{0.2cm}\small\it  $^c$Department of Mathematics,
			Western University, \\[-0.7ex]
			\vspace{0.2cm}\small\it   London, N6A 5B7, Canada \\ }

\begin{abstract}
In this paper, we investigate weakly persistent centers and isochronous centers for several
classes of planar cubic differential systems, including systems with a
nondegenerate linear center and systems with a nilpotent center.
Necessary parameter relations are obtained by computing singular point
quantities or quasi-Lyapunov constants, while their sufficiency is
established by suitable analytic constructions. For the general complex cubic system, we obtain necessary and sufficient conditions for the origin to be a weakly persistent center, and the corresponding result for the associated general real cubic system follows by imposing the conjugacy relations between the complex coefficients. We further characterize the weakly persistent isochronous centers of the general complex cubic system by combining the weakly persistent center conditions with the vanishing of the linearization quantities and constructing time-preserving analytic linearizations. At last, we  obtain necessary and sufficient conditions for the nilpotent origin to be a weakly persistent center by computing  quasi-Lyapunov constants.

\end{abstract}
	\begin{keyword}
		weakly persistent center; weakly persistent isochronous center; complex cubic system; nilpotent center; analytic linearization.
	\end{keyword}
	
\end{frontmatter}

\section{Introduction}\label{sec_intro}

The center--focus problem for planar analytic systems is one of the classical topics in the qualitative theory of differential equations, see \cite{Aa2010,LB2017,Gine2017}. Its study can be traced back to the pioneering works of Poincar\'e and Lyapunov \cite{Liapunov1966} on the local phase portrait of planar autonomous systems. 

Consider the planar real system
\begin{equation} \label{sys_pq}
\dot x=P(x,y),\quad
\dot y=Q(x,y),
\end{equation}
where \(P\) and \(Q\) are analytic in a neighborhood of the origin, and the origin is an isolated singular point. If the eigenvalues of the linearized system at the origin are purely imaginary, namely \(\pm i\), then the origin is called a nondegenerate center-type singularity. In this case, the local trajectories near the origin exhibit two possible behaviors: if all nearby trajectories are closed orbits, then the origin is called a center; otherwise, if the trajectories spiral toward or away from the origin, then the origin is called a focus \cite{Zhang1992}.

The distinction between centers and foci is usually characterized by the so-called focal values, also known as Lyapunov constants. Using normal form theory, the Poincaré return map \cite{Zhang1992}, or polar-coordinate expansions, one can derive a sequence of focal values describing the local behavior of trajectories near the singular point. If all focal values vanish, the system formally exhibits a center structure. However, the vanishing of finitely many focal values does not automatically guarantee the existence of an analytic center. Therefore, additional structural tools, such as first integrals, integrating factors \cite{LW1995,LYR2014}, inverse integrating factors, or symmetry properties \cite{Ye1986}, are often required to establish the center property.

The center problem for quadratic polynomial differential systems has been completely solved, and the corresponding center variety is known to consist of four irreducible components \cite{Bautin1954,Dulac1908}. For general cubic systems, however, the ordinary center problem remains open, and a complete classification of cubic centers is still unavailable. After normalization, the general cubic family contains fourteen independent coefficients, and the associated focal values rapidly become large multivariate polynomials. The difficulty is further illustrated by the result of Bothmer and Kröker, who showed that the vanishing of the first eleven focal values is not sufficient to guarantee a center \cite{Bothmer2010}. Thus, even the number of focal values required to characterize cubic centers remains unknown.

Owing to this difficulty, most results for cubic systems have been obtained for families satisfying additional structural restrictions. Examples include Kukles systems \cite{Li2019,JM2010}, quasi-homogeneous systems \cite{AZIZ2014}, $Z_2$-equivariant systems \cite{LF2020}, and systems possessing prescribed invariant algebraic curves \cite{LHM2027,Cozma2022}. Although center or isochronicity conditions have been determined for a number of such subclasses, these results do not yield a classification of centers for the general cubic system.

Within this framework, a natural question is whether the center property is preserved throughout a parameterized family of differential systems. This leads to the notions of persistent centers and weakly persistent centers introduced by Cima, Gasull, and Medrado \cite{CIMa2009}. For a real planar analytic differential system written in the complex form
\[
\dot{z}=iz+F(z,\bar z),
\]
the origin is called a persistent center if it is a center of
\[
\dot{z}=iz+\lambda F(z,\bar z)
\]
for every \(\lambda\in\mathbb C\), whereas it is called a weakly persistent center if it is a center of
\[
\dot{z}=iz+uF(z,\bar z)
\]
for every \(u\in\mathbb R\). Clearly, every persistent center is a weakly persistent center. Cima et al. concentrated on persistent centers and obtained complete classifications for several families, including cubic and rigid systems. The problem of weakly persistent centers was not investigated in their work. They also pointed out that analogous notions could be introduced for degenerate centers.

Chen, Romanovski, and Zhang \cite{CHEN2014} subsequently extended these notions to complex planar differential systems in which the two variables are independent and are not required to satisfy a conjugacy relation. More precisely, they considered
\[
\dot{x}=ix+\lambda F(x,y),\qquad
\dot{y}=-iy+\mu G(x,y),
\]
where \(x,y\in\mathbb C\). In this setting, the origin is called a 
persistent center if it is a center for all independent 
\(\lambda,\mu\in\mathbb C\), and a weakly persistent center if it is a center when \(\lambda=\mu\in\mathbb C\). Thus, weakly persistent centers of complex systems correspond to the use of a common complex perturbation parameter in the two equations.

They further investigated the general complex cubic system
\[
\begin{cases}
\dot{x}=ix+a_{20}x^2+a_{11}xy+a_{02}y^2
+a_{30}x^3+a_{21}x^2y+a_{12}xy^2+a_{03}y^3,\\
\dot{y}=-iy+b_{20}x^2+b_{11}xy+b_{02}y^2
+b_{30}x^3+b_{21}x^2y+b_{12}xy^2+b_{03}y^3.
\end{cases}
\]
For this system, they obtained eight sets of necessary and sufficient conditions for persistent centers. For the cubic Lotka--Volterra subfamily, they completely characterized weakly persistent centers and obtained five sets of necessary and sufficient conditions. However, a complete classification of weakly persistent centers for the general complex cubic system was not obtained because of the substantial computational difficulties involved in decomposing the associated affine variety.

An important feature of the complex formulation is its relation to real planar systems. Chen et al.~\cite{CHEN2014} showed that, under the appropriate conjugacy relations between the two sets of complex
coefficients, the persistent and weakly persistent center conditions of the complex system coincide with those of the corresponding real
system. Thus, a complete classification for the general complex system
also provides, by restriction to the conjugate parameter subspace, the
corresponding classification for the associated real system.

Further progress was made from the viewpoint of linearizability. 
Mencinger et al. \cite{Men2018} introduced the notions of linearizable persistent centers and linearizable weakly persistent centers. They obtained necessary and sufficient conditions for the general complex cubic system above to have a linearizable persistent center at the origin. Moreover, for a semi-Kolmogorov subfamily, they obtained necessary and sufficient conditions for linearizable weakly persistent centers.

Despite these advances, weakly persistent centers of the general
complex cubic system have not yet been completely characterized.
Existing complete results for weakly persistent centers and
linearizable weakly persistent centers are mainly restricted to
particular subclasses, such as the cubic Lotka--Volterra and
semi-Kolmogorov systems. It is therefore natural to ask whether a
complete classification can be obtained for the general complex cubic
system. In view of the correspondence between complex and real systems
under the conjugacy conditions, such a classification would also yield
the weakly persistent center conditions for the associated general real cubic system.

To address this question, after the standard normalization of the linear part, we consider the following \(1:-1\) resonant form of the general complex cubic system with a common perturbation parameter
\begin{equation}\label{sys_general}
\begin{cases}
\dot z=z+\varepsilon\left(
a_{20}z^2+a_{11}zw+a_{02}w^2
+a_{30}z^3+a_{21}z^2w+a_{12}zw^2+a_{03}w^3
\right),\\
\dot w=-w-\varepsilon\left(
b_{20}w^2+b_{11}zw+b_{02}z^2
+b_{30}w^3+b_{21}zw^2+b_{12}z^2w+b_{03}z^3
\right),
\end{cases}
\end{equation}
where \(\varepsilon\in\mathbb C\), and all coefficients \(a_{ij}\) and \(b_{ij}\) are independent complex parameters, with no conjugacy relation imposed between \(z\) and \(w\).

For system~\eqref{sys_general}, our first objective is to obtain a
complete classification of weakly persistent centers directly in terms
of the original parameters. By imposing the conjugacy relations between the two sets of complex coefficients, this classification also yields the corresponding weakly persistent center conditions for the associated general real cubic system. We further investigate weakly  persistent isochronous centers of system~\eqref{sys_general}. By combining the weakly persistent center classification with the
linearization quantities, we derive the corresponding necessary and
sufficient conditions and establish their sufficiency by
time-preserving analytic linearizations.

Starting from the general complex cubic system, we further investigate weakly persistent centers for systems with degenerate center-type singularities, with particular emphasis on nilpotent singularities. 
As noted by Cima et al. \cite{CIMa2009}, the notion underlying weakly persistent centers can also be considered for degenerate centers. This motivates us to study the weakly persistent center problem when the singularity at the origin is nilpotent rather than nondegenerate.

More precisely, we consider the cubic polynomial system
\begin{equation}\label{sys_nilp_0}
\begin{cases}
\dot x=y+\varepsilon(a_{20}x^2+a_{11}xy+a_{02}y^2
+a_{30}x^3+a_{21}x^2y+a_{12}xy^2+a_{03}y^3),\\
\dot y=-2x^3+\varepsilon(b_{11}xy+b_{02}y^2
+b_{21}x^2y+b_{12}xy^2+b_{03}y^3).
\end{cases}
\end{equation}

The linearization of system~\eqref{sys_nilp_0} at the origin is nilpotent, and hence the usual focal-value approach for nondegenerate center-type singularities is no longer directly applicable. Instead, quasi-Lyapunov constants are used to derive necessary conditions for weakly persistent centers, while the sufficiency of the resulting conditions is established separately by appropriate analytic structures.

The remainder of this paper is organized as follows.
Section~\ref{sec_pre} introduces the basic definitions and analytic
tools used throughout the paper. In Section~\ref{sec_general_cubic_center}, we derive necessary and sufficient conditions for weakly persistent centers of the general complex cubic system and then restrict the resulting classification to the conjugate parameter subspace to obtain the corresponding result for the associated real cubic system.
Section~\ref{sec_cubic_iso} is devoted to weakly persistent isochronous centers of the general complex cubic system.In Section~\ref{sec_nilpotent}, we investigate weakly persistent centers for the cubic polynomial system with a nilpotent singularity. Finally, Section~\ref{sec_conclusion} presents the conclusion.


\section{Preliminaries}\label{sec_pre}

\subsection{Singular point quantities}

We first introduce the singular point quantities used throughout this paper. 
Consider a real planar analytic system whose linear part is a center,
\[
\dot x=-y+O(2),\qquad
\dot y=x+O(2).
\]
Introducing the complex variables
\[
z=x+iy,\qquad w=x-iy,
\]
and rescaling the time variable by
\[
T=it,
\]
the linear part is transformed into
\[
\frac{dz}{dT}=z,\qquad
\frac{dw}{dT}=-w.
\]
Thus, the eigenvalues $\pm i$ of the real system are transformed into
$1$ and $-1$ by the complex change of variables together with a nonzero
rescaling of time.

When the coefficients of the resulting complex system satisfy the
corresponding conjugacy relations, the complex system represents a real
planar system on the real slice $w=\overline{z}$. In this case, the
quantities obtained from the recursive procedure below correspond, up
to nonzero normalization factors, to the focal values of the associated
real system. For a complex system with eigenvalues $1$ and $-1$, these
quantities are also often called saddle quantities. Since the variables
and coefficients considered in this paper are allowed to be independent
complex quantities, we use the unified terminology
\emph{singular point quantities} throughout the paper.

Consider the complex analytic system
\begin{equation}\label{sys_saddle_general}
\dot z=z+\varepsilon P(z,w),\qquad
\dot w=-w+\varepsilon Q(z,w),
\end{equation}
where $z$ and $w$ are independent complex variables,
$\varepsilon\in\mathbb C$, and $P$ and $Q$ contain only nonlinear
terms.

To study the existence of a local first integral near the origin, we
seek a formal power series of the form
\[
H(z,w)
=
zw+\sum_{j+k\geq3}h_{jk}z^jw^k.
\]
Let
\[
X_0
=
z\frac{\partial}{\partial z}
-
w\frac{\partial}{\partial w}
\]
denote the linear part of the vector field associated with
system~\eqref{sys_saddle_general}. Since
\[
X_0(z^jw^k)
=
(j-k)z^jw^k,
\]
the coefficients corresponding to the nonresonant monomials
$z^jw^k$, with $j\neq k$, can be determined recursively.

The resonant monomials, however, cannot in general be eliminated.
Consequently, the coefficients $h_{jk}$ can be chosen recursively so
that
\begin{equation*}
\frac{dH}{dt}
=
\sum_{m=1}^{\infty}
g_m(\varepsilon)(zw)^{m+1}.
\end{equation*}
The quantities
\[
g_m(\varepsilon),\qquad m=1,2,\ldots,
\]
are called the \emph{singular point quantities} of the origin.

If system~\eqref{sys_saddle_general} admits an analytic first integral
of the form
\[
H(z,w)=zw+O(3),
\]
then necessarily
\[
g_m(\varepsilon)=0,
\qquad m=1,2,\ldots.
\]
Accordingly, a necessary condition for the origin to be a weakly
persistent center is
\begin{equation*}
g_m(\varepsilon)\equiv0,
\qquad m=1,2,\ldots,
\end{equation*}
for every value of the common perturbation parameter
$\varepsilon$.

For the polynomial systems considered in this paper, each singular
point quantity is a polynomial in $\varepsilon$. Hence, it can be
written as
\[
g_m(\varepsilon)
=
\sum_{k=0}^{N_m}
G_{m,k}\varepsilon^k,
\]
where the coefficients $G_{m,k}$ are polynomials in the system
parameters. Therefore,
\(
g_m(\varepsilon)\equiv0
\)
is equivalent to
\[
G_{m,k}=0,
\qquad
k=0,1,\ldots,N_m.
\]

Consequently, necessary parameter relations for weakly persistent
centers can be obtained by equating to zero all coefficients of the
singular point quantities with respect to $\varepsilon$. The resulting
polynomial systems are then analyzed by symbolic algebraic
computations. The sufficiency of the obtained parameter relations is
established separately by constructing analytic first integrals,
inverse integrating factors, or other suitable analytic structures.

\subsection{Weakly persistent isochronous centers and linearization quantities}

We next introduce weakly persistent isochronous centers and the
corresponding linearization quantities. Consider again
system~\eqref{sys_saddle_general}. The origin is called a weakly persistent isochronous center if, for
every $\varepsilon\in\mathbb R$, there exists a local analytic
near-identity transformation
\[
Z=z+O(2),\qquad W=w+O(2),
\]
which transforms system~\eqref{sys_saddle_general}, without any
rescaling of the time variable, into the linear system
\[
\dot Z=Z,\qquad
\dot W=-W.
\]
This notion corresponds to the linearizable weakly persistent center
considered in \cite{Men2018}.

To derive necessary conditions for the origin to be a weakly persistent
isochronous center, we seek a formal near-identity transformation of the
form
\[
Z=z+\sum_{j+k\geq2}p_{jk}z^jw^k,\qquad
W=w+\sum_{j+k\geq2}q_{jk}z^jw^k.
\]
Substituting these expansions into
system~\eqref{sys_saddle_general} and comparing the coefficients of
monomials of the same degree give a sequence of homological equations.

For the linear part
\[
X_0
=
z\frac{\partial}{\partial z}
-
w\frac{\partial}{\partial w},
\]
the monomial $z^jw^k$ has weight $j-k$. Hence, the resonant monomials
in the first equation satisfy
\[
j-k=1,
\]
whereas those in the second equation satisfy
\[
j-k=-1.
\]
The coefficients of these resonant monomials cannot, in general, be
eliminated by a near-identity transformation. They give rise to the
linearization quantities
\[
L_m^{+}(\varepsilon),\qquad
L_m^{-}(\varepsilon),\qquad
m=1,2,\ldots.
\]

After eliminating the nonresonant terms recursively, the corresponding
formal normal form can be written as
\[
\dot Z
=
Z+\sum_{m=1}^{\infty}
L_m^{+}(\varepsilon)Z^{m+1}W^m,
\]
and
\[
\dot W
=
-W+\sum_{m=1}^{\infty}
L_m^{-}(\varepsilon)Z^mW^{m+1}.
\]
Therefore, a necessary condition for the origin to be a weakly
persistent isochronous center is
\[
L_m^{+}(\varepsilon)\equiv0,\qquad
L_m^{-}(\varepsilon)\equiv0,
\qquad
m=1,2,\ldots,
\]
for every value of the common perturbation parameter $\varepsilon$.

For the polynomial systems considered in this paper, each linearization
quantity is a polynomial in $\varepsilon$. Hence, it can be written as
\[
L_m^{\pm}(\varepsilon)
=
\sum_{k=0}^{N_m}
L_{m,k}^{\pm}\varepsilon^k.
\]
Therefore, \(
L_m^{\pm}(\varepsilon)\equiv0\)
is equivalent to
\[
L_{m,k}^{\pm}=0,
\qquad
k=0,1,\ldots,N_m.
\]

Consequently, necessary parameter relations for weakly persistent
isochronous centers can be obtained by equating to zero all coefficients
of the linearization quantities with respect to $\varepsilon$. The
resulting polynomial systems are then analyzed by symbolic algebraic
computations. The sufficiency of the obtained parameter relations is
established separately by constructing time-preserving analytic
linearizations.

\subsection{Nilpotent singularities and quasi-Lyapunov constants}

We next recall the method of quasi-Lyapunov constants for nilpotent singularities, see \cite{LYR2014,YU2017,Li2019}. Consider the analytic system
\begin{equation*}
    \begin{cases}
\dot{x} = X(x,y) = y + \sum_{i+j\ge 2} a_{ij} x^i y^j, \\[4pt]
\dot{y} = Y(x,y) = \sum_{i+j\ge 2} b_{ij} x^i y^j.
\end{cases}
\end{equation*}
Since $X_y(0,0)=1$, the implicit function theorem guarantees that the equation
$$
X(x,y)=0
$$
admits a unique analytic solution
$$
y=f(x)=\sum_{j=2}^{\infty} c_j x^j, \qquad f(0)=0,
$$
in a neighborhood of the origin. Suppose that
$$
Y(x,f(x)) = \alpha_k x^k + o(x^k), \qquad \alpha_k \neq 0.
$$
Then the origin is called a nilpotent singular point of multiplicity $k$. In particular, the nilpotent singularity  of system \eqref{sys_nilp} with multiplicity three is monodromic if and only if
$$b_{20}=0, (2a_{20}+b_{11})^2+8 \alpha_3<0.$$ 
It is well known that the origin of a system with a multiplicity $3$ monodromic singular point can be written in the following form:
\begin{equation}\label{2.2}\begin{split}
\frac{dx}{dt}&=y+a_{20} x^2+\sum\limits_{i+2j=3}^\infty a_{ij}x^iy^j=X(x,y),\\
\frac{dy}{dt}&=-2x^3+2a_{20} xy+\sum\limits_{i+2j=4}^\infty b_{ij}x^iy^j=Y(x,y).
\end{split}\end{equation}
 The center-focus problem of system \eqref{2.2} can be studied by means of quasi-Lyapunov constants. More precisely, for a fixed nonnegative integer $s$, one can recursively construct a formal power series
$$
M(x,y) = x^4 + y^2 + o(r^4), \qquad r = \sqrt{x^2 + y^2},
$$
such that
\[
\frac{\partial}{\partial x}\!\left(\frac{X(x,y)}{M^{s+1}(x,y)}\right)
+ \frac{\partial}{\partial y}\!\left(\frac{Y(x,y)}{M^{s+1}(x,y)}\right)
= \frac{1}{M^{s+2}(x,y)}
\sum_{m=1}^{\infty} (2m-4s-1) V_m \bigl[x^{2m+4} + o(r^{2m+4})\bigr].
\]
The quantities $V_m$ are called the quasi-Lyapunov constants of the origin; see \cite{LYR2014,LF2012}.

The coefficients of $M(x,y)$ are determined recursively by expanding the preceding identity and comparing terms of the same weighted degree. In the recursive computations used below, if $w_{2m+4}$ denotes the coefficient obtained at the $m$-th step, then
$$
V_m = \frac{w_{2m+4}}{2m-4s-1}, \qquad m=1,2,\dots.
$$

Under the monodromy assumptions imposed in this paper, the origin is a center if and only if all its quasi-Lyapunov constants vanish:
$$
V_m = 0, \qquad m=1,2,\dots.
$$

For the perturbed nilpotent systems considered below, the quasi-Lyapunov constants are polynomials in \(\varepsilon\). Therefore, the weakly persistent center conditions are obtained by requiring all their coefficients with respect to \(\varepsilon\) to vanish.

\subsection{Criteria for proving the sufficiency of center and isochronous center conditions}
\label{subsec_sufficiency_criteria}

The vanishing of the singular point quantities or quasi-Lyapunov
constants provides necessary parameter relations for the origin to be
a weakly persistent center, while the vanishing of the linearization
quantities gives necessary relations for a weakly persistent
isochronous center. To establish the sufficiency of the resulting
parameter relations, the corresponding analytic structures must be
verified for every admissible value of the common perturbation
parameter.

For the weakly persistent center problem, we construct analytic first
integrals of the form
\[
H(z,w)=zw+O(3).
\]
Depending on the structure of the reduced system, the proofs are based
on explicit first integrals, inverse integrating factors, regular
analytic initial-value problems, or convergent analytic constructions.
For the weakly persistent isochronous center problem, we further seek
analytic coordinates
\[
U=z+O(2),
\qquad
V=w+O(2),
\]
satisfying
\[
XU=U,
\qquad
XV=-V,
\]
so that the system is analytically linearized without any
reparametrization of time. We record below several analytic tools that
will be used repeatedly in the subsequent sufficiency proofs.

We first recall a standard consequence of the theory of inverse
integrating factors; see, for example, \cite{Garcia2010,Giacomini1996}.

\begin{lemma}\label{lem_inverse_integrating_factor}
Consider the analytic system
\[
\dot z=P(z,w),
\qquad
\dot w=Q(z,w),
\]
where
\[
P(z,w)=z+O(2),
\qquad
Q(z,w)=-w+O(2).
\]
If there exists an analytic function $V(z,w)$ satisfying
\[
V(0,0)=1,
\qquad
PV_z+QV_w=(P_z+Q_w)V,
\]
then the system admits a local analytic first integral of the form
\[
H(z,w)=zw+O(3).
\]
In particular, if an analytic function $W(z,w)$ satisfies
\[
XW=KW,
\qquad
W(0,0)=1,
\]
where
\[
X
=
P\frac{\partial}{\partial z}
+
Q\frac{\partial}{\partial w}
\]
and
\[
\operatorname{div}X=\lambda K
\]
for some constant $\lambda$, then $V=W^\lambda$ is an inverse
integrating factor.
\end{lemma}

The next lemma provides another useful way of constructing analytic
first integrals after a suitable linear change of coordinates.

\begin{lemma}\label{lem_regular_IVP}
Suppose that
\[
x=\alpha z+\beta w,
\qquad
y=\alpha z-\beta w,
\qquad
\alpha\beta\neq0,
\]
transforms the system into
\[
\dot x=y\,d(x),
\qquad
\dot y=x\,q(x)+bxy+A(x)y^2,
\]
where
\[
d(0)=q(0)=1.
\]
Let
\[
k(x)=\frac{d'(x)+A(x)}{d(x)}.
\]
If
\[
\bigl(d(x)q(x)\bigr)'
-k(x)d(x)q(x)
=
\gamma x
\]
for some constant $\gamma$, then the system admits a local analytic
first integral
\[
H(z,w)=zw+O(3).
\]
\end{lemma}

\begin{proof}
Set $p=y\,d(x)$, and let $\mu(x)$ be the analytic solution of
\[
\mu'(x)=-k(x)\mu(x),
\qquad
\mu(0)=1.
\]
Define
\[
T'(x)=\mu(x)x,
\qquad
T(0)=0,
\qquad
R=\mu(x)p-bT.
\]
By the hypothesis,
\[
\mu(x)d(x)q(x)=1+\gamma T.
\]
A direct calculation gives
\[
\dot T=x(R+bT),
\qquad
\dot R=x(1+\gamma T).
\]
Hence, away from $x=0$,
\[
\frac{dT}{dR}
=
\frac{R+bT}{1+\gamma T}.
\]
Since the right-hand side is analytic near $(T,R)=(0,0)$, this
equation admits a local analytic first integral $K(T,R)$ normalized by
\[
K(T,0)=T.
\]
Its quadratic part is
\[
K(T,R)
=
\frac{x^2-y^2}{2}
+O(3).
\]
Since
\[
x^2-y^2=4\alpha\beta zw,
\]
the function
\[
H(z,w)
=
\frac{1}{2\alpha\beta}K(T,R)
\]
is an analytic first integral satisfying
\[
H(z,w)=zw+O(3).
\]
The identity extends analytically across $x=0$, and hence the result
holds in a full neighborhood of the origin.
\end{proof}

We next record two criteria that will be used in proving the
sufficiency of the weakly persistent isochronous center conditions.
The first shows that, once an analytic first integral is available, it
is sufficient to construct only one analytic linearizing coordinate.

\begin{lemma}\label{lem_first_integral_linearization}
Consider an analytic system with linear part
\[
\dot z=z,
\qquad
\dot w=-w.
\]
Suppose that the system admits an analytic first integral
\[
H(z,w)=zw+O(3).
\]
If there exists an analytic function
\[
V(z,w)=w+O(2)
\]
satisfying
\[
XV=-V,
\]
then there exists an analytic function
\[
U(z,w)=z+O(2)
\]
such that
\begin{equation}\label{eq:iso_eigen_pair}
XU=U,
\qquad
XV=-V.
\end{equation}
Consequently, $(U,V)$ gives a local time-preserving analytic
linearization. The same conclusion holds if an analytic function
$U=z+O(2)$ satisfying $XU=U$ is given instead.
\end{lemma}

\begin{proof}
Since $XV=-V$, the analytic curve $V=0$ is invariant. The restriction
of the vector field to this curve has the form
\[
\dot z=z+O(z^2).
\]
Since $H$ is constant along the orbits and $H(0,0)=0$, its restriction
to $V=0$ vanishes identically near the origin. Hence, analytic division
gives
\[
H=UV
\]
for some analytic function $U$. Since
\[
H=zw+O(3),
\qquad
V=w+O(2),
\]
we have
\[
U=z+O(2).
\]
Applying $X$ to $H=UV$ gives
\[
0
=
XH
=
(XU)V+U(XV)
=
(XU-U)V.
\]
Therefore,
\[
XU=U,
\]
and hence \eqref{eq:iso_eigen_pair} holds. Since the Jacobian of
$(U,V)$ at the origin is the identity, $(U,V)$ is a local analytic
change of coordinates. Thus, the system is transformed into
\[
\dot U=U,
\qquad
\dot V=-V
\]
without any change of the time variable. The case in which $U$ is
given first is analogous.
\end{proof}

For some algebraic parameter families, an explicit analytic
linearization is first constructed on a generic subset. The following
criterion allows the remaining degenerate parameter values to be
included.

\begin{lemma}\label{lem_linearization_extension}
Consider an irreducible algebraic family of analytic systems with
linear part
\[
\dot z=z,
\qquad
\dot w=-w.
\]
Suppose that the system is analytically linearizable on a
Zariski-dense open subset of the parameter family. Then all
linearization quantities vanish identically on the whole family.
Consequently, the system is analytically linearizable at every
parameter value of the family.
\end{lemma}

\begin{proof}
Seek a tangent-to-the-identity formal change of variables degree by degree.  At
degree $n$, the homological operator on a monomial $z^iw^je_r$ has eigenvalue
\[
 \langle(i,j),(1,-1)\rangle-\lambda_r=i-j-\lambda_r,
 \qquad \lambda_1=1,\quad\lambda_2=-1.
\]
All nonzero eigenvalues are nonzero integers and can therefore be inverted
without introducing parameter denominators.  After the lower-degree equations
have been solved, projection of the degree-$n$ right-hand side onto the kernel
of the homological operator gives finitely many resonant obstruction functions
$B_{n,\alpha}\in\mathcal O(S)$.  Their coefficients are regular functions of
$s$ because the original coefficients are regular and every division made in
the recursion is by a fixed nonzero integer.

Analytic linearizability on the dense subset implies formal linearizability
there, so every $B_{n,\alpha}$ vanishes on a Zariski-dense subset of $S$.
Since $S$ is reduced and irreducible, each obstruction is identically zero in
$\mathcal O(S)$.  The degree-by-degree recursion consequently produces a
formal linearization of $X_s$ for every $s\in S$.

It remains to justify convergence.  In the notation of Brjuno's analytic
normalization theorem, put
\[
 \omega_k=\min\bigl\{
 |i-j-\lambda_r|:\ 2\le i+j<2^k,\ r=1,2,
 \ i-j-\lambda_r\ne0\bigr\}.
\]
Here $\omega_k\ge1$ for every $k$; hence
\[
 \sum_{k\ge1}2^{-k}\log(\omega_k^{-1})=0.
\]
Thus Brjuno's arithmetic condition $\omega$ holds.  Moreover, the formal
normal form just obtained is exactly $Ax$, so Brjuno's Condition A holds with
scalar factor $f\equiv1$.  Brjuno's convergence theorem therefore turns the
formal linearization at each fixed $s$ into a holomorphic one; see
\cite{Bruno1971}.  No uniform neighborhood or parameter-holomorphic choice of
the conjugacy is needed.  Notice also that the dense-set hypothesis must give
an all-order linearization: vanishing through any fixed finite order would not
justify this argument.
\end{proof}

For several parameter conditions, one of the two equations becomes
autonomous. The following elementary formula gives the corresponding
one-dimensional analytic linearization.

\begin{lemma}\label{lem_one_dim_linearization}
Let
\[
\dot u=\sigma u f(u),
\qquad
\sigma\in\{1,-1\},
\qquad
f(0)=1,
\]
where $f$ is analytic near the origin. Then
\begin{equation}\label{eq:one_dim_linearization}
\Phi(u)
=
u\exp\left(
\int_0^u
\frac{f(s)^{-1}-1}{s}\,ds
\right)
\end{equation}
is analytic near $u=0$, satisfies
\[
\Phi(u)=u+O(u^2),
\]
and transforms the equation into
\[
\dot\Phi=\sigma\Phi.
\]
\end{lemma}

\begin{proof}
Since $f(0)=1$, the function
\[
\frac{f(u)^{-1}-1}{u}
\]
has a removable singularity at $u=0$ and is analytic there. Hence,
$\Phi$ in \eqref{eq:one_dim_linearization} is analytic near the origin
and satisfies
\[
\Phi(u)=u+O(u^2).
\]
Moreover,
\[
\frac{\Phi'(u)}{\Phi(u)}
=
\frac{1}{u}
+
\frac{f(u)^{-1}-1}{u}
=
\frac{1}{uf(u)}.
\]
Therefore,
\[
\dot\Phi
=
\Phi'(u)\dot u
=
\sigma\Phi.
\]
\end{proof}

Finally, the following observation allows us to avoid repeating the
proofs for parameter conditions obtained by exchanging the two
coefficient sets.

\begin{lemma}\label{lem_exchange}
Let
\[
\sigma(z,w)=(w,z),
\]
and let $X^{e}$ denote the vector field obtained from $X$ by exchanging
the two coefficient sets $a_{ij}$ and $b_{ij}$. Then
\begin{equation}\label{eq:exchange_vector_field}
X^{e}=-\sigma_*X.
\end{equation}
Consequently, both the weakly persistent center property and the
weakly persistent isochronous center property are preserved under the
exchange
\[
a_{ij}\longleftrightarrow b_{ij}.
\]
\end{lemma}

\begin{proof}
Suppose first that the origin is a weakly persistent center of $X$.
Then, for every admissible value of the common perturbation parameter,
there exists an analytic first integral
\[
H(z,w)=zw+O(3)
\]
satisfying $XH=0$. Define
\[
H^{e}(z,w)=H(w,z).
\]
By \eqref{eq:exchange_vector_field},
\[
X^{e}H^{e}
=
-(XH)(w,z)
=
0.
\]
Moreover,
\[
H^{e}(z,w)=zw+O(3).
\]
Hence, the exchanged system also has a weakly persistent center.

Next, suppose that the origin is a weakly persistent isochronous
center of $X$. Then, for every admissible value of the common
perturbation parameter, there exist analytic functions
\[
U=z+O(2),
\qquad
V=w+O(2),
\]
satisfying
\[
XU=U,
\qquad
XV=-V.
\]
Define
\[
U^{e}(z,w)=V(w,z),
\qquad
V^{e}(z,w)=U(w,z).
\]
Clearly,
\[
U^{e}=z+O(2),
\qquad
V^{e}=w+O(2).
\]
Using \eqref{eq:exchange_vector_field}, we obtain
\[
X^{e}U^{e}
=
-(XV)(w,z)
=
V(w,z)
=
U^{e},
\]
and
\[
X^{e}V^{e}
=
-(XU)(w,z)
=
-U(w,z)
=
-V^{e}.
\]
Therefore,
\[
X^{e}U^{e}=U^{e},
\qquad
X^{e}V^{e}=-V^{e}.
\]
Thus, the exchanged system is analytically linearizable without any
reparametrization of the time variable. Hence, the weakly persistent
isochronous center property is also preserved under the exchange.
\end{proof}
\section{Weakly persistent centers in the general complex cubic system} \label{sec_general_cubic_center}

\subsection{Weakly persistent center conditions}
\label{subsec_general_cubic_center_conditions}

We now consider system~\eqref{sys_general}. By computing the singular
point quantities and analyzing the resulting polynomial relations, the finite collection of necessary conditions is decomposed completely into candidate algebraic components, and then every surviving component is proved sufficient by an independent analytic argument. We obtain a complete characterization of its weakly persistent centers.
The main result is stated as follows.

\begin{thm}\label{thm:cubic_center}
The origin of system~\eqref{sys_general} is a weakly persistent center
if and only if
\[
a_{21}=b_{21},
\qquad
a_{11}a_{20}=b_{11}b_{20},
\]
and one of the following conditions holds:
\begin{align*}
\text{F$_{01}$}\quad
& b_{11}-2a_{20}
=a_{11}-2b_{20}
=b_{12}-3a_{30}
=a_{12}-3b_{30}
=0;
\\
\text{F$_{02}$}\quad
& \begin{aligned}[t]
&\text{there exists }h\in\mathbb C,\quad h\neq0,\quad\text{such that}\\
&b_{20}-ha_{20}=hb_{11}-a_{11}=h^3b_{02}-a_{02}=b_{30}-h^2a_{30}=h^2b_{12}-a_{12}
 =h^4b_{03}-a_{03}\\
&\quad \ =0;
\end{aligned}
\\
\text{F$_{03}$}\quad
& a_{11}=b_{11}=a_{21}=a_{03}=b_{03}=a_{12}+b_{30}=b_{12}+a_{30}=0;
\\
\text{F$_{04}$}\quad
& \begin{aligned}[t]
&\text{there exists }t\in\mathbb C\text{ such that}\\
&a_{02}=b_{02}=a_{03}=b_{03}=a_{11}-tb_{20}=b_{11}-ta_{20}=a_{12}-(2t-1)b_{30}\\
& \quad  \ =b_{12}-(2t-1)a_{30}=0;
\end{aligned}
\\ 
\text{F$_{05}$}\quad
& a_{20}=b_{20}=a_{02}=b_{02}=a_{30}=b_{30}=a_{03}=b_{03}=0;
\\ 
\text{F$_{06}$(a)}\quad
& a_{20}=b_{11}=b_{02}=a_{30}=b_{12}=b_{03}=0;
\\ 
\text{F$_{06}$(b)}\quad
& b_{20}=a_{11}=a_{02}=b_{30}=a_{12}=a_{03}=0;
\\ 
\text{F$_{07}$(a)}\quad
& a_{11}=a_{02}=a_{12}=a_{03}=a_{21}=b_{11}=b_{30}=0;
\\ 
\text{F$_{07}$(b)}\quad
& b_{11}=b_{02}=b_{12}=b_{03}=b_{21}=a_{11}=a_{30}=0;
\\ 
\text{F$_{08}$}\quad
& a_{11}=a_{02}=a_{21}=a_{12}=a_{03}=b_{11}=b_{02}=b_{12}=b_{03}=0;
\\
\text{F$_{09}$}\quad
& \begin{aligned}[t]
&a_{20}b_{20}\neq0, \enspace a_{30}=a_{21}=a_{12}=a_{03}
=b_{30}=b_{03}=b_{12}=2a_{11}+b_{20}=2b_{11}+a_{20}\\
&\quad \ =4a_{02}b_{02}-a_{20}b_{20}=0;
\end{aligned}
\\ 
\text{F$_{10}$}\quad
& \begin{aligned}[t]
&a_{20}b_{20}\neq0, \enspace 
2a_{11}+b_{20}=2b_{11}+a_{20}=2a_{02}a_{20}+b_{20}^2=2b_{02}b_{20}+a_{20}^2=a_{12}-3b_{30}\\
& \quad \ =b_{12}-3a_{30}=2a_{03}a_{20}^3
+3a_{20}^2b_{20}b_{30}-a_{30}b_{20}^3=2b_{03}b_{20}^3-a_{20}^3b_{30}
+3a_{20}a_{30}b_{20}^2\\
& \quad \ =2a_{20}a_{21}b_{20}+3a_{20}^2b_{30}
+3a_{30}b_{20}^2=0;
\end{aligned}
\\
\text{F$_{11}$}\quad
& \begin{aligned}[t]
&a_{20}b_{20}\neq0, \enspace
4a_{11}-b_{20}=4b_{11}-a_{20}
=4a_{02}a_{20}-b_{20}^2
=4b_{02}b_{20}-a_{20}^2
=a_{12}-3b_{30}\\
&\quad\ 
=b_{12}-3a_{30}
=2a_{03}a_{20}^3+3a_{20}^2b_{20}b_{30}-a_{30}b_{20}^3=2b_{03}b_{20}^3-a_{20}^3b_{30}
+3a_{20}a_{30}b_{20}^2\\
&\quad\
=2a_{20}a_{21}b_{20}
-3a_{20}^2b_{30}-3a_{30}b_{20}^2=a_{20}^4b_{30}^2
-4a_{20}^2a_{30}b_{20}^2b_{30}
+a_{30}^2b_{20}^4=0;
\end{aligned}
\\
\text{F$_{12}$}\quad
& \begin{aligned}[t]
&\text{there exist }r,\tau\in\mathbb C
\text{ such that }a_{20}b_{20}r(7\tau-2)\neq0,\\
&a_{11}=b_{11}
=a_{02}a_{20}r^2(7\tau-2)-b_{20}^2(6r\tau-2r-\tau)\\
&\quad\
=b_{02}b_{20}(7\tau-2)-a_{20}^2r(-r\tau+6\tau-2) =a_{30}b_{20}^2-a_{20}^2b_{30}r^2=a_{12}-b_{30}(6\tau-3)
\\
&\quad\
=b_{12}-a_{30}(6\tau-3)=a_{03}a_{20}r-b_{20}b_{30}\tau=b_{03}b_{20}^3-a_{20}^3b_{30}r^3\tau
\\
&\quad\ =a_{21}b_{20}-a_{20}b_{30}r(5\tau-2)=0;
\end{aligned}
\end{align*}
\begin{align*}
\text{F$_{13}$}\quad
& \begin{aligned}[t]
&\text{there exist }c,e\in\mathbb C
\text{ such that }a_{20}b_{20}ce(ce^2-1)\neq0,\\
&a_{11}=b_{11}=a_{12}=b_{12}
=a_{02}a_{20}(ce^2-1)-b_{20}^2e(e-1)\\
& \quad \ =b_{02}b_{20}(ce^2-1)-a_{20}^2ce(ce-1)
=a_{30}b_{20}^2-a_{20}^2b_{30}c=a_{03}a_{20}-b_{20}b_{30}e
\\
&\quad \ =b_{03}b_{20}^3-a_{20}^3b_{30}c^2e
=a_{21}b_{20}-a_{20}b_{30}ce=0;
\end{aligned}
\\
\text{F$_{14}$}\quad
& \begin{aligned}[t]
&a_{20}b_{20}\neq0, \enspace
a_{11}=b_{11}=a_{12}=b_{12}
=a_{02}a_{20}^3+b_{02}b_{20}^3-a_{20}^2b_{20}^2=a_{30}b_{20}^2-a_{20}^2b_{30}
\\
&\quad\
=a_{03}a_{20}-b_{20}b_{30}=b_{03}b_{20}-a_{20}a_{30}=a_{20}a_{21}-a_{30}b_{20}=0;
\end{aligned}
\\ 
\text{F$_{15}$}\quad
& \begin{aligned}[t]
&a_{20}b_{20}\neq0, \enspace
a_{11}=b_{11}
=a_{02}a_{20}^3+b_{02}b_{20}^3+2a_{20}^2b_{20}^2=a_{30}b_{20}^2-a_{20}^2b_{30}
=7a_{12}+9b_{30}\\
&\quad\
=7b_{12}+9a_{30}=7a_{03}a_{20}+2b_{20}b_{30}
=7b_{03}b_{20}+2a_{20}a_{30}
=7a_{20}a_{21}-4a_{30}b_{20}=0;
\end{aligned}
\\
\text{F$_{16}$}\quad
& \begin{aligned}[t]
&a_{20}b_{20}\neq0, \enspace
a_{11}=b_{11}
=a_{02}b_{02}-a_{20}b_{20}
=a_{02}a_{30}+a_{20}b_{30}=a_{12}+b_{30}
=b_{12}+a_{30}\\
&\quad\
=a_{03}a_{20}^2-a_{02}^2b_{03}
=a_{20}a_{21}-a_{02}b_{03}=0;
\end{aligned}
\\ 
\text{F$_{17}$}\quad
& \begin{aligned}[t]
&\text{there exists }r\in\mathbb C
\text{ such that }a_{20}b_{20}r\neq0,\\
&a_{11}=b_{11}=a_{30}=b_{30}
=a_{02}a_{20}r^2-b_{20}^2(2r+1)
=b_{02}b_{20}-a_{20}^2r(r+2)\\
&\quad\
=b_{12}b_{20}^2-a_{12}a_{20}^2r^2=2a_{03}a_{20}r+a_{12}b_{20}=2b_{03}b_{20}^3+a_{12}a_{20}^3r^3
\\
&\quad\ =2a_{21}b_{20}-3a_{12}a_{20}r=0;
\end{aligned}
\\ 
\text{F$_{18}$}\quad
& \begin{aligned}[t]
&\text{there exists }r\in\mathbb C
\text{ such that }a_{20}b_{20}r\neq0,\\
&a_{11}=b_{11}=a_{30}=b_{30}
=7a_{02}a_{20}r^2-b_{20}^2(6r-1)
=7b_{02}b_{20}-a_{20}^2r(6-r)\\
&\quad\
=b_{12}b_{20}^2-a_{12}a_{20}^2r^2=6a_{03}a_{20}r-a_{12}b_{20}=6b_{03}b_{20}^3-a_{12}a_{20}^3r^3
\\
&\quad\ =6a_{21}b_{20}-5a_{12}a_{20}r=0;
\end{aligned}
\\ 
\text{F$_{19}$(a)}\quad
& \begin{aligned}[t]
&a_{20}=a_{02}=a_{21}=a_{12}=a_{03}
=b_{11}=b_{30}=b_{03}
=2a_{11}+b_{20}
=b_{12}+2a_{30}=0;
\end{aligned}
\\ 
\text{F$_{19}$(b)}\quad
& \begin{aligned}[t]
&b_{20}=b_{02}=b_{21}=b_{12}=b_{03}
=a_{11}=a_{30}=a_{03}
=2b_{11}+a_{20}
=a_{12}+2b_{30}=0;
\end{aligned}
\\ 
\text{F$_{20}$(a)}\quad
& \begin{aligned}[t]
&a_{20}=a_{02}=a_{30}=a_{21}=a_{12}=a_{03}
=b_{11}=b_{02}=b_{30}=b_{12}
=3a_{11}+b_{20}=0;
\end{aligned}
\\ 
\text{F$_{20}$(b)}\quad
& \begin{aligned}[t]
&b_{20}=b_{02}=b_{30}=b_{21}=b_{12}=b_{03}
=a_{11}=a_{02}=a_{30}=a_{12}
=3b_{11}+a_{20}=0;
\end{aligned}
\\ 
\text{F$_{21}$(a)}\quad
& \begin{aligned}[t]
&a_{20}=a_{02}=a_{30}=a_{21}=a_{03}
=b_{11}=b_{02}=b_{12}
=3a_{11}-b_{20}
=3a_{12}+b_{30}=0;
\end{aligned}
\\
\text{F$_{21}$(b)}\quad
& \begin{aligned}[t]
&b_{20}=b_{02}=b_{30}=b_{21}=b_{03}
=a_{11}=a_{02}=a_{12}
=3b_{11}-a_{20}
=3b_{12}+a_{30}=0;
\end{aligned}
\\ 
\text{F$_{22}$(a)}\quad
& \begin{aligned}[t]
&b_{20}\neq0, \enspace
a_{20}=a_{11}=b_{11}=b_{02}=a_{03}=b_{03}
=b_{20}^2b_{30}-a_{02}^2a_{30}
=a_{12}+3b_{30}
\\
&\quad \ =b_{12}+3a_{30}
=a_{21}b_{20}+2a_{02}a_{30}=0;
\end{aligned}
\\ 
\text{F$_{22}$(b)}\quad
& \begin{aligned}[t]
&a_{20}\neq0, \enspace
b_{20}=b_{11}=a_{11}=a_{02}=b_{03}=a_{03}
=a_{20}^2a_{30}-b_{02}^2b_{30}
=b_{12}+3a_{30}
\\
&\quad \ =a_{12}+3b_{30}
=a_{20}b_{21}+2b_{02}b_{30}=0;
\end{aligned}
\\ 
\text{F$_{23}$(a)}\quad
& \begin{aligned}[t]
&b_{02}b_{20}\neq0, \enspace
a_{20}=a_{11}=b_{11}=a_{02}
=b_{02}b_{30}+a_{30}b_{20}
=a_{12}+b_{30}
 =b_{12}+a_{30}\\
&\quad \
=b_{03}b_{20}^2-a_{03}b_{02}^2
=a_{21}b_{20}-a_{03}b_{02}=0;
\end{aligned}
\\ 
\text{F$_{23}$(b)}\quad
& \begin{aligned}[t]
&a_{02}a_{20}\neq0, \enspace
b_{20}=b_{11}=a_{11}=b_{02}
=a_{02}a_{30}+a_{20}b_{30}
=b_{12}+a_{30}
 =a_{12}+b_{30}
\\
&\quad \ =a_{03}a_{20}^2-a_{02}^2b_{03}
=a_{20}b_{21}-a_{02}b_{03}=0;
\end{aligned}
\\ 
\text{F$_{24}$(a)}\quad
& \begin{aligned}[t]
&a_{02}b_{20}\neq0, \enspace
a_{20}=a_{11}=b_{11}=a_{30}=b_{30}
=a_{02}^2b_{02}-4b_{20}^3
=a_{02}^2b_{12}-4a_{12}b_{20}^2
\\
&\quad \ =4a_{03}b_{20}+a_{02}a_{12}
=a_{02}b_{03}+b_{12}b_{20}
=a_{02}a_{21}-3a_{12}b_{20}=0;
\end{aligned}
\\ 
\text{F$_{24}$(b)}\quad
& \begin{aligned}[t]
&a_{20}b_{02}\neq0, \enspace
b_{20}=b_{11}=a_{11}=b_{30}=a_{30}
=a_{02}b_{02}^2-4a_{20}^3
=a_{12}b_{02}^2-4a_{20}^2b_{12}
\\
&\quad \ =4a_{20}b_{03}+b_{02}b_{12}
=a_{03}b_{02}+a_{12}a_{20}
=b_{02}b_{21}-3a_{20}b_{12}=0;
\end{aligned}
\\
\text{F$_{25}$(a)}\quad
& \begin{aligned}[t]
&\text{there exists }r\in\mathbb C
\text{ such that }a_{02}b_{20}r(r-1)\neq0,\\
&a_{20}=a_{11}=b_{11}=a_{12}=b_{12}
=a_{02}^2b_{02}-b_{20}^3r^2(r-1)=a_{02}^2a_{30}-b_{20}^2b_{30}r(r-1)
\\
&\quad\ =a_{03}b_{20}(r-1)+a_{02}b_{30}=a_{02}b_{03}+b_{02}b_{30}
=a_{02}a_{21}+b_{20}b_{30}r=0;
\end{aligned}
\end{align*}
\begin{align*}
\text{F$_{25}$(b)}\quad
& \begin{aligned}[t]
&\text{there exists }r\in\mathbb C
\text{ such that }a_{20}b_{02}r(r-1)\neq0,\\
&b_{20}=b_{11}=a_{11}=b_{12}=a_{12}
=a_{02}b_{02}^2-a_{20}^3r^2(r-1)=b_{02}^2b_{30}-a_{20}^2a_{30}r(r-1)\\
&\quad\
=a_{20}b_{03}(r-1)+a_{30}b_{02}=a_{03}b_{02}+a_{02}a_{30}
=b_{02}b_{21}+a_{20}a_{30}r=0;
\end{aligned}
\\
\text{F$_{26}$(a)}\quad
& \begin{aligned}[t]
&\text{there exists }r\in\mathbb C
\text{ such that }a_{02}b_{20}r(r-1)\neq0,\\
&a_{20}=a_{11}=b_{11}
=a_{02}^2b_{02}-b_{20}^3r^2(r-1)
=a_{02}^2a_{30}-b_{20}^2b_{30}r^2=a_{02}^2b_{12}-a_{12}b_{20}^2r^2\\
&\quad\
=a_{12}(7r-6)-b_{30}(6-9r)=6a_{03}b_{20}r
+a_{02}(7a_{12}+9b_{30})(r-1)\\
&\quad\ =6a_{02}^3b_{03}
+b_{20}^3r^3(7a_{12}+9b_{30})(r-1)=6a_{02}a_{21}
+b_{20}r(1-2r)(7a_{12}+9b_{30})=0;
\end{aligned}
\\
\text{F$_{26}$(b)}\quad
& \begin{aligned}[t]
&\text{there exists }r\in\mathbb C
\text{ such that }a_{20}b_{02}r(r-1)\neq0,\\
&b_{20}=b_{11}=a_{11}
=a_{02}b_{02}^2-a_{20}^3r^2(r-1)
=b_{02}^2b_{30}-a_{20}^2a_{30}r^2=a_{12}b_{02}^2-a_{20}^2b_{12}r^2\\
&\quad\
=b_{12}(7r-6)-a_{30}(6-9r)=6a_{20}b_{03}r
+b_{02}(9a_{30}+7b_{12})(r-1)\\
&\quad\ =6a_{03}b_{02}^3
+a_{20}^3r^3(9a_{30}+7b_{12})(r-1)=6b_{02}b_{21}
+a_{20}r(1-2r)(9a_{30}+7b_{12})=0;
\end{aligned}
\\
\text{F$_{27}$(a)}\quad
& \begin{aligned}[t]
&a_{20}=a_{11}=a_{02}=a_{30}=a_{21}=a_{03}
=b_{20}=b_{11}=b_{12}=b_{03}
=2a_{12}+b_{30}=0;
\end{aligned}
\\
\text{F$_{27}$(b)}\quad
& \begin{aligned}[t]
&b_{20}=b_{11}=b_{02}=b_{30}=b_{21}=b_{03}
=a_{20}=a_{11}=a_{12}=a_{03}
=2b_{12}+a_{30}=0;
\end{aligned}
\\
\text{F$_{28}$}\quad
& \begin{aligned}[t]
&a_{30}b_{30}\neq0, \enspace
a_{20}=a_{11}=a_{02}=b_{20}=b_{11}=b_{02}=a_{21}=3a_{12}+b_{30}
=3b_{12}+a_{30}\\
&\quad\
=9a_{03}b_{03}-4a_{30}b_{30}=0.
\end{aligned}
\end{align*}
\end{thm}

\begin{proof}
We first prove the necessity. Suppose that the origin of
system~\eqref{sys_general} is a weakly persistent center. Then all
singular point quantities introduced in Section~\ref{sec_pre} must
vanish identically with respect to $\varepsilon$.

The first singular point quantity is
\[
g_1(\varepsilon)
=
(a_{21}-b_{21})\varepsilon
+
(b_{11}b_{20}-a_{11}a_{20})\varepsilon^2.
\]
Hence,
\begin{equation}\label{eq:common_center_conditions}
a_{21}=b_{21},
\qquad
a_{11}a_{20}=b_{11}b_{20}.
\end{equation}
Moreover, the second singular point quantity has the form
\[
g_2(\varepsilon)
=
G_{2,2}\varepsilon^2
+
G_{2,3}\varepsilon^3
+
G_{2,4}\varepsilon^4.
\]
To solve the resulting polynomial equations, we distinguish three
cases according to the values of $a_{20}$ and $b_{20}$.

\medskip
\noindent\textbf{Case 1. $a_{20}b_{20}\neq0$.}

Under the nonsingular scaling
\[
Z=a_{20}z,\qquad W=b_{20}w,
\]
we introduce
\[
t=\frac{a_{11}}{b_{20}}
=\frac{b_{11}}{a_{20}},
\qquad
A=\frac{a_{02}a_{20}}{b_{20}^2},
\qquad
B=\frac{b_{02}b_{20}}{a_{20}^2},
\qquad
C=\frac{a_{30}}{a_{20}^2},
\qquad
D=\frac{b_{30}}{b_{20}^2},
\]
\[
U=\frac{a_{12}}{b_{20}^2},
\qquad
V=\frac{b_{12}}{a_{20}^2},
\qquad
E=\frac{a_{03}a_{20}}{b_{20}^3},
\qquad
F=\frac{b_{03}b_{20}}{a_{20}^3}.
\]
In these normalized parameters,
\[
G_{2,2}=DV-CU,
\]
\[
\begin{aligned}
G_{2,3}
=\frac{1}{3}\bigl(
&5ACt+2AC+AFt-2AF-3AVt+2AV
-5BDt-2BD-BEt+2BE
\\
&+3BUt-2BU+6Ct^2-3Ct
-6Dt^2+3Dt+3Ut-3Vt
\bigr),
\end{aligned}
\]
and
\[
G_{2,4}
=
\frac{1}{3}t(A-B)(t-2)(2t+1).
\]
Thus, $G_{2,4}=0$ yields the branches
\[
t=0,\qquad
t=2,\qquad
A=B,\qquad
t=-\frac12.
\]

The higher-order singular point quantities are then computed
successively on these branches. In particular, the branch $A=B$,
together with its rank-degenerate subcases, yields the families
corresponding to $F_{01}$, $F_{02}$, $F_{04}$, and
$F_{09}$--$F_{11}$. The further decomposition of the branch $t=0$,
including the cases
\[
C=D=0,
\qquad
(C,D)\neq(0,0),
\]
and the subsequent vanishing or nonvanishing of the relevant
elimination factors, gives the families corresponding to
$F_{03}$, $F_{07}(a)$--$F_{08}$, and $F_{12}$--$F_{18}$,
together with families already obtained on the preceding branches.
Consequently, the case $a_{20}b_{20}\neq0$ yields
\[
F_{01}\text{--}F_{04},
\qquad
F_{07}(a),\ F_{07}(b),\ F_{08}	-F_{18}.
\]

\medskip
\noindent\textbf{Case 2. $a_{20}b_{20}=0$ and
$(a_{20},b_{20})\neq(0,0)$.}

By exchanging the two coefficient sets if necessary, we first assume
\[
a_{20}=0,\qquad b_{20}=d\neq0.
\]
It follows from \eqref{eq:common_center_conditions} that $b_{11}=0$.
After the scaling
\[
Z=z,\qquad W=dw,
\]
we set
\[
\begin{aligned}
t&=\frac{a_{11}}{d},
&
A&=\frac{a_{02}}{d^2},
&
B&=db_{02},&
C&=a_{30},
&
D&=\frac{b_{30}}{d^2},
\\
U&=\frac{a_{12}}{d^2},
&
V&=b_{12},&
E&=\frac{a_{03}}{d^3},
&
F&=db_{03}.
\end{aligned}
\]
Then
\[
G_{2,2}=DV-CU,
\]
\[
G_{2,3}
=
\frac{1}{3}
\bigl(
AFt-2AF-5BDt-2BD+3BUt-2BU
+6Ct^2-3Ct-3Vt
\bigr),
\]
and
\[
G_{2,4}
=
-\frac{1}{3}Bt(t-2)(2t+1).
\]
Hence, the relevant branches are
\[
t=0,\qquad
t=2,\qquad
t=-\frac12,
\qquad
B=0.
\]

The higher-order singular point quantities are evaluated successively
on these branches, together with all their degenerate boundary cases.
Solving the resulting polynomial systems gives
\[
\begin{gathered}
F_{01},\quad F_{03},\quad F_{04},\quad
F_{06}(a),\quad F_{07}(a),\quad F_{08},\quad
F_{19}(a),\quad F_{20}(a),\\ F_{21}(a),\quad
 F_{22}(a),\quad
F_{23}(a),\quad F_{24}(a),\quad F_{25}(a),\quad F_{26}(a).
\end{gathered}
\]
Thus, these are the conditions obtained when
$a_{20}=0$ and $b_{20}\neq0$.

When
\[
b_{20}=0,\qquad a_{20}\neq0,
\]
the corresponding conditions are obtained by exchanging the two
coefficient sets. They are
\[
\begin{gathered}
F_{01},\quad F_{03},\quad F_{04},\quad
F_{06}(b),\quad F_{07}(b),\quad F_{08},\quad
F_{19}(b),\quad F_{20}(b),\\ F_{21}(b),\quad F_{22}(b),\quad
F_{23}(b),\quad F_{24}(b),\quad F_{25}(b),\quad F_{26}(b).
\end{gathered}
\]
Hence, these two lists exhaust the singly degenerate case
$a_{20}b_{20}=0$ with $(a_{20},b_{20})\neq(0,0)$.

\medskip
\noindent\textbf{Case 3. $a_{20}=b_{20}=0$.}

Set
\[
\begin{alignedat}{5}
u&=a_{11},
&\qquad v&=b_{11},
&\qquad A&=a_{02},
&\qquad B&=b_{02},
&\qquad C&=a_{30},
\\[0.8ex]
D&=b_{30},
&\qquad V&=b_{12},
&\qquad U&=a_{12},
&\qquad E&=a_{03},
&\qquad F&=b_{03}.
\end{alignedat}
\]
The relation $a_{21}=b_{21}$ remains in force. In this case,
\[
G_{2,2}=DV-CU,
\]
\[
G_{2,3}
=
\frac{1}{3}
\bigl(
5ACv+AFu-3AVv-5BDu-BEv+3BUu
+6Cu^2-6Dv^2
\bigr),
\]
and
\[
G_{2,4}
=
\frac{2}{3}\left(Av^3-Bu^3\right).
\]
We further distinguish
\[
uv\neq0,\qquad
uv=0,\quad (u,v)\neq(0,0),
\qquad
u=v=0.
\]

The higher-order singular point quantities are then computed
successively in each subcase. When $u=v=0$, the remaining quadratic
coefficients $A$ and $B$ are further separated according to their
vanishing, including the purely cubic case. Solving the resulting
polynomial systems gives
\[
F_{01}-F_{05},\quad F_{06}(a),\quad F_{06}(b),
\quad
F_{21}(a),\quad F_{21}(b),\quad
F_{27}(a),\quad F_{27}(b),\quad F_{28}.
\]
Thus, these conditions exhaust the doubly degenerate case
$a_{20}=b_{20}=0$.

Combining the above three cases, all the parameter relations
$F_{01}$--$F_{28}$ stated in Theorem~\ref{thm:cubic_center} are
obtained. Therefore, the necessity is proved.


We now prove the sufficiency of the conditions listed in
Theorem~\ref{thm:cubic_center}. Throughout the proof, we write
\[
S=a_{21}=b_{21}.
\]

When condition $\mathrm{F}_{01}$ holds, system~\eqref{sys_general} can be rewritten as 
\[
\begin{cases} 
\dot z =z+\varepsilon\bigl( a_{20}z^2+2b_{20}zw+a_{02}w^2 +a_{30}z^3+Sz^2w+3b_{30}zw^2+a_{03}w^3 \bigr),\\[1ex]
\dot w =-w-\varepsilon\bigl( b_{20}w^2+2a_{20}zw+b_{02}z^2 +b_{30}w^3+Szw^2+3a_{30}z^2w+b_{03}z^3 \bigr). 
\end{cases}
\]
It admits the following first integral 
\[ 
\begin{aligned}
H_{01}={}&zw+\varepsilon\Bigg( a_{20}z^2w+b_{20}zw^2 +\frac{a_{02}}{3}w^3+\frac{b_{02}}{3}z^3+a_{30}z^3w+b_{30}zw^3 +\frac{S}{2}z^2w^2\\ 
& +\frac{a_{03}}{4}w^4+\frac{b_{03}}{4}z^4 \Bigg).
\end{aligned} 
\]
Therefore, the origin is a weakly persistent center.

When condition $\mathrm{F}_{02}$ holds,
system~\eqref{sys_general} can be rewritten as
\[
\begin{cases}
\dot z=z+\varepsilon\left(
a_{20}z^2+a_{11}zw+a_{02}w^2
+a_{30}z^3+Sz^2w+a_{12}zw^2+a_{03}w^3
\right),\\[1ex]
\dot w=-w-\varepsilon\left(
ha_{20}w^2+\dfrac{a_{11}}{h}zw+\dfrac{a_{02}}{h^3}z^2
+h^2a_{30}w^3+Szw^2
+\dfrac{a_{12}}{h^2}z^2w+\dfrac{a_{03}}{h^4}z^3
\right),
\end{cases}
\]
where $h\neq0$. Taking
\[
u=z+hw,
\qquad
v=z-hw,
\]
the system becomes
\[
\dot u=v f(u,v^2),
\qquad
\dot v=g(u,v^2),
\]
where $f$ and $g$ are analytic near the origin and satisfy
\[
f(0,0)=1,
\qquad
g(u,v^2)=u+O(2).
\]
Indeed, the above system becomes 
\begin{equation*}
\begin{cases} 
\dot u =v+\frac{\varepsilon v}{4h^3} \Big[ 4h(a_{20}h^2-a_{02})u +(a_{21}h^2-a_{12}h-3a_{03}+3a_{30}h^3)u^2 \\ \qquad +(-a_{21}h^2+a_{12}h-a_{03}+a_{30}h^3)v^2 \Big],\\
\dot v =u+\frac{\varepsilon}{4h^3} \Big[ 2h(a_{02}+a_{11}h+a_{20}h^2)u^2 +(a_{21}h^2+a_{12}h+a_{03}+a_{30}h^3)u^3\\
\qquad +2h(a_{02}-a_{11}h+a_{20}h^2)v^2 +(-a_{21}h^2-a_{12}h+3a_{03}+3a_{30}h^3)uv^2 \Big]. 
\end{cases} 
\end{equation*}
Setting $T=v^2$, we obtain
\[
\frac{dT}{du}
=
\frac{2g(u,T)}{f(u,T)}.
\]
Since $f(0,0)=1$, the equation is regular analytic near the origin and
admits a local analytic invariant $K(u,T)$ such that
\[
K(u,v^2)=v^2-u^2+O(3).
\]
Since $v^2-u^2=-4hzw$, it follows that
\[
H_{02}(z,w)=-\frac{1}{4h}K(u,v^2)=zw+O(3)
\]
is an analytic first integral. Hence, the origin is a weakly persistent
center.

When condition $\mathrm{F}_{03}$ holds,
system~\eqref{sys_general} reduces to
\[
\begin{cases}
\dot z=z+\varepsilon\left(
a_{20}z^2+a_{02}w^2+a_{30}z^3-b_{30}zw^2
\right),\\
\dot w=-w-\varepsilon\left(
b_{20}w^2+b_{02}z^2+b_{30}w^3-a_{30}z^2w
\right).
\end{cases}
\]
Let
\[
V_{03}
=
1+2\varepsilon M+\varepsilon^2N+\varepsilon^3T,
\]
where
\begin{align*}
M={}&
a_{20}z+b_{20}w+a_{30}z^2+b_{30}w^2,
\\
N={}&
M^2+(a_{20}b_{20}-3a_{02}b_{02})zw
+a_{02}a_{20}w^2+b_{02}b_{20}z^2+4a_{02}a_{30}w^2z
+4b_{02}b_{30}wz^2
\\
&-4a_{30}b_{30}z^2w^2,
\\
T={}&
\left(a_{02}w^3+a_{20}wz^2+b_{02}z^3+b_{20}w^2z\right)\\
&\times \left(a_{02}a_{30}w-a_{02}b_{02}+a_{20}b_{20}
+a_{20}b_{30}w+a_{30}b_{20}z+b_{02}b_{30}z\right).
\end{align*}
Denote the corresponding vector field by $X$. Direct differentiation
shows that
\[
XV_{03}
=
(\operatorname{div}X)V_{03}.
\]
Since
\[
V_{03}(0,0)=1,
\]
Lemma~\ref{lem_inverse_integrating_factor} implies that the system
admits a local analytic first integral
\[
H_{03}(z,w)=zw+O(3).
\]
Hence, the origin is a weakly persistent center.

When condition $\mathrm{F}_{04}$ holds,
system~\eqref{sys_general} can be rewritten as
\[
\begin{cases}
\dot z=z+\varepsilon\left(
a_{20}z^2+tb_{20}zw+a_{30}z^3
+Sz^2w+(2t-1)b_{30}zw^2
\right),\\[1ex]
\dot w=-w-\varepsilon\left(
b_{20}w^2+ta_{20}zw+b_{30}w^3
+Szw^2+(2t-1)a_{30}z^2w
\right).
\end{cases}
\]
Set
\(
q=zw,\quad
R=a_{20}z+b_{20}w+a_{30}z^2+b_{30}w^2,
\quad
K=a_{20}z-b_{20}w+2a_{30}z^2-2b_{30}w^2.
\) 
A direct calculation gives
\[
\dot q=\varepsilon(1-t)Kq,
\qquad
\dot R=K\bigl(1+\varepsilon R+\varepsilon Sq\bigr).
\]

If $t\neq0$, let \(F=1+\varepsilon\left(R+\frac{S}{t}q\right)\), 
then $\dot F=\varepsilon KF$, and hence
\[
H_{04}=qF^{\,t-1}=zw+O(3)
\]
is an analytic first integral.

If $t=0$, then
\[
\frac{dq}{dR}
=
\frac{\varepsilon q}
{1+\varepsilon R+\varepsilon Sq}.
\]
This regular analytic equation admits a local analytic invariant
$\mathcal K(R,q)$ with $\mathcal K(0,q)=q$. Therefore,
\[
H_{04}(z,w)
=
\mathcal K(R(z,w),q(z,w))
=
zw+O(3)
\]
is an analytic first integral.

Thus, the origin is a weakly persistent center.

When condition $\mathrm{F}_{05}$ holds,
system~\eqref{sys_general} reduces to
\[
\begin{cases}
\dot z=z+\varepsilon\left(
a_{11}zw+Sz^2w+a_{12}zw^2
\right),\\
\dot w=-w-\varepsilon\left(
b_{11}zw+Szw^2+b_{12}z^2w
\right).
\end{cases}
\]
Define
\[
H_{05}
=
zw\exp\left[
\varepsilon\left(
b_{11}z+a_{11}w+Szw
+\frac{b_{12}}{2}z^2
+\frac{a_{12}}{2}w^2
\right)
\right].
\]
A direct differentiation gives
\[
\frac{dH_{05}}{dt}=0.
\]
Moreover,
\[
H_{05}(z,w)=zw+O(3).
\]
Hence, $H_{05}$ is a local analytic first integral, and the origin is a
weakly persistent center.

When condition $\mathrm{F}_{06}(a)$ holds,
system~\eqref{sys_general} can be rewritten as
\[
\begin{cases}
\dot z=z+\varepsilon\left(
a_{11}zw+a_{02}w^2+Sz^2w+a_{12}zw^2+a_{03}w^3
\right),\\
\dot w=-w-\varepsilon\left(
b_{20}w^2+Szw^2+b_{30}w^3
\right).
\end{cases}
\]
Let $q=zw$. Then
\[
\frac{dq}{dw}
=
-\frac{
\varepsilon\left[
(a_{11}-b_{20})q+a_{02}w^2
+(a_{12}-b_{30})qw+a_{03}w^3
\right]
}{
1+\varepsilon b_{20}w+\varepsilon Sq
+\varepsilon b_{30}w^2
}.
\]
Since the denominator equals one at the origin, this equation is
regular analytic near $(q,w)=(0,0)$ and admits a local analytic
invariant $K_{06}(q,w)$ with
\[
K_{06}(q,0)=q.
\]
Hence,
\[
H_{06}(z,w)
=
K_{06}(zw,w)
=
zw+O(3)
\]
is an analytic first integral, and the origin is a weakly persistent
center.

When condition $\mathrm{F}_{06}(b)$ holds, the corresponding system is
\[
\begin{cases}
\dot z=z+\varepsilon\left(
a_{20}z^2+a_{30}z^3+Sz^2w
\right),\\
\dot w=-w-\varepsilon\left(
b_{11}zw+b_{02}z^2+Szw^2
+b_{12}z^2w+b_{03}z^3
\right).
\end{cases}
\]
This is the exchanged form of the system corresponding to
$\mathrm{F}_{06}(a)$. Therefore, by
Lemma~\ref{lem_exchange}, the analytic first integral constructed for
$\mathrm{F}_{06}(a)$ yields an analytic first integral
\[
H_{06b}(z,w)=zw+O(3).
\]
Hence, the origin is a weakly persistent center.

When condition $\mathrm{F}_{07}(a)$ holds,
system~\eqref{sys_general} reduces to
\[
\begin{cases}
\dot z=z+\varepsilon\left(
a_{20}z^2+a_{30}z^3
\right),\\
\dot w=-w-\varepsilon\left(
b_{20}w^2+b_{02}z^2+b_{12}z^2w+b_{03}z^3
\right).
\end{cases}
\]
Write
\[
p(z)
=
1+\varepsilon a_{20}z+\varepsilon a_{30}z^2.
\]
The system admits an analytic invariant curve
\[
w=\phi(z),
\qquad
\phi(z)=O(z^2).
\]
Set
\[
u=w-\phi(z),
\qquad
a_0(z)
=
-1-\varepsilon b_{12}z^2
-2\varepsilon b_{20}\phi(z),
\]
and define
\[
\eta(z)
=
\frac{p(z)+a_0(z)}{zp(z)},
\qquad
\mu(z)
=
\exp\left(
-\int_0^z\eta(s)\,ds
\right),
\qquad
\rho=\mu(z)zu.
\]
Furthermore, let
\[
B(z)
=
\frac{\varepsilon b_{20}}{p(z)\mu(z)},
\qquad
J(z)
=
\int_0^z
\frac{B(s)-\varepsilon b_{20}}{s^2}\,ds.
\]
Since $(p\mu)'(0)=0$, we have
\[
B(z)=\varepsilon b_{20}+O(z^2),
\]
and hence $J(z)$ is analytic at the origin. A direct calculation then
shows that
\[
H_{07}(z,w)
=
\frac{\rho}
{1+\varepsilon b_{20}\mu(z)u-\rho J(z)}
=
zw+O(3)
\]
is an analytic first integral. Therefore, the origin is a weakly
persistent center.

When condition $\mathrm{F}_{07}(b)$ holds, system~\eqref{sys_general}
can be rewritten as
\[
\begin{cases}
\dot z=z+\varepsilon\left(
a_{20}z^2+a_{02}w^2+a_{12}zw^2+a_{03}w^3
\right),\\
\dot w=-w-\varepsilon\left(
b_{20}w^2+b_{30}w^3
\right).
\end{cases}
\]
This system is the exchanged form of the system corresponding to
condition $\mathrm{F}_{07}(a)$. Hence, by
Lemma~\ref{lem_exchange}, the analytic first integral obtained for
$\mathrm{F}_{07}(a)$ yields an analytic first integral
\[
H_{07b}(z,w)=zw+O(3).
\]
Therefore, the origin is a weakly persistent center.

When condition $\mathrm{F}_{08}$ holds,
system~\eqref{sys_general} reduces to
\[
\begin{cases}
\dot z=z+\varepsilon\left(
a_{20}z^2+a_{30}z^3
\right),\\
\dot w=-w-\varepsilon\left(
b_{20}w^2+b_{30}w^3
\right).
\end{cases}
\]
Let
\[
p(z)
=
1+\varepsilon a_{20}z+\varepsilon a_{30}z^2,
\quad
r(w)
=
1+\varepsilon b_{20}w+\varepsilon b_{30}w^2,
\]
and define
\[
\Phi(z)
=
z\exp\left[
\int_0^z
\left(
\frac{1}{s p(s)}-\frac{1}{s}
\right)ds
\right],\quad
\Psi(w)
=
w\exp\left[
\int_0^w
\left(
\frac{1}{s r(s)}-\frac{1}{s}
\right)ds
\right].
\]
The integrands are analytic at the origin, and direct differentiation
gives
\[
X\Phi=\Phi,
\qquad
X\Psi=-\Psi.
\]
Consequently,
\[
H_{08}(z,w)
=
\Phi(z)\Psi(w)
=
zw+O(3)
\]
is an analytic first integral. Hence, the origin is a weakly persistent
center.

For the following conditions with $a_{20}b_{20}\neq0$, we use the
normalization
\begin{equation}\label{eq:common_normalization}
\begin{alignedat}{4}
Z&=a_{20}z,
&\qquad W&=b_{20}w,
&\qquad t&=\frac{a_{11}}{b_{20}}
           =\frac{b_{11}}{a_{20}},
&\qquad \widehat S&=\frac{S}{a_{20}b_{20}},
\\[0.8ex]
A&=\frac{a_{02}a_{20}}{b_{20}^2},
& B&=\frac{b_{02}b_{20}}{a_{20}^2},
& C&=\frac{a_{30}}{a_{20}^2},
& D&=\frac{b_{30}}{b_{20}^2},
\\[0.8ex]
U&=\frac{a_{12}}{b_{20}^2},
& V&=\frac{b_{12}}{a_{20}^2},
& E&=\frac{a_{03}a_{20}}{b_{20}^3},
& F&=\frac{b_{03}b_{20}}{a_{20}^3}.
\end{alignedat}
\end{equation}
Then system~\eqref{sys_general} takes the form
\begin{equation}\label{eq:normalized_general}
\begin{cases}
\dot Z
=
Z+\varepsilon\left(
Z^2+tZW+AW^2
+CZ^3+\widehat S Z^2W+UZW^2+EW^3
\right),\\
\dot W
=
-W-\varepsilon\left(
W^2+tZW+BZ^2
+DW^3+\widehat S ZW^2+VZ^2W+FZ^3
\right).
\end{cases}
\end{equation}

When condition $\mathrm{F}_{09}$ holds,
\eqref{eq:common_normalization} gives
\[
t=-\frac12,\qquad
AB=\frac14,
\qquad
C=D=U=V=E=F=\widehat S=0.
\]
Thus, writing $B=1/(4A)$, system~\eqref{eq:normalized_general}
reduces to
\[
\begin{cases}
\dot Z
=
Z+\varepsilon\left(
Z^2-\dfrac12ZW+AW^2
\right),\\[1ex]
\dot W
=
-W-\varepsilon\left(
W^2-\dfrac12ZW+\dfrac{1}{4A}Z^2
\right).
\end{cases}
\]
Set
\[
s=\varepsilon(Z+W),
\qquad
\ell=\varepsilon(Z+2AW),
\]
and define
\[
V_{09}
=
\frac{
\left[8A(1+s)+\ell^2\right]
\left[
64A^2+96A^2s
+12A(2A+1)s\ell
+(2A+1)\ell^3
\right]
}{
512A^3
}.
\]
A direct calculation gives
\[
XV_{09}=(\operatorname{div}X)V_{09},
\qquad
V_{09}(0,0)=1.
\]
Hence, by Lemma~\ref{lem_inverse_integrating_factor}, after pulling
back and normalizing, system~\eqref{sys_general} admits an analytic
first integral
\[
H_{09}(z,w)=zw+O(3).
\]
Therefore, the origin is a weakly persistent center.

When condition $\mathrm{F}_{10}$ holds,
system~\eqref{eq:normalized_general} reduces to
\[
\begin{cases}
\dot Z
=
Z+\varepsilon\left[
Z^2-\dfrac12ZW-\dfrac12W^2
+CZ^3-\dfrac32(C+D)Z^2W
+3DZW^2+\dfrac{C-3D}{2}W^3
\right],\\[1ex]
\dot W
=
-W-\varepsilon\left[
W^2-\dfrac12ZW-\dfrac12Z^2
+DW^3-\dfrac32(C+D)ZW^2
+3CZ^2W+\dfrac{D-3C}{2}Z^3
\right].
\end{cases}
\]
Let
\[
h=C-D,\qquad
j=C+D,\qquad
x=Z+W,\qquad
y=Z-W,
\]
and set $p=\dot y$. Then
\[
\dot y=p,
\qquad
\dot p
=
y+\frac{5\varepsilon}{2}yp
+\left(\varepsilon j-\frac{3\varepsilon^2}{4}\right)y^3
+\frac{3\varepsilon^2h^2}{4}y^5.
\]
With $u=y^2$,
\[
\frac{du}{dp}
=
\frac{2p}{
1+\dfrac{5\varepsilon}{2}p
+\left(\varepsilon j-\dfrac{3\varepsilon^2}{4}\right)u
+\dfrac{3\varepsilon^2h^2}{4}u^2
}.
\]
This equation is regular analytic at the origin and admits a local
analytic invariant $K_{10}(p,u)$ with $K_{10}(0,u)=u$. Since
\[
y^2-p^2=-4ZW+O(3),
\]
we obtain
\[
\widetilde H_{10}
=
-\frac14K_{10}(p,y^2)
=
ZW+O(3).
\]
After pulling back and normalizing,
\[
H_{10}(z,w)=zw+O(3).
\]
Hence, the origin is a weakly persistent center.

When condition $\mathrm{F}_{11}$ holds, the normalized parameters
satisfy
\[
t=A=B=\frac14,\quad
\widehat S=\frac32(C+D),
\qquad
U=3D,\qquad
V=3C,
\qquad
E=\frac{C-3D}{2},
\]
\[
F=\frac{D-3C}{2}, \quad \text{and} \quad
C^2-4CD+D^2=0.
\]
Hence, system~\eqref{eq:normalized_general} becomes
\[
\begin{cases}
\dot Z
=
Z+\varepsilon\left[
Z^2+\dfrac14ZW+\dfrac14W^2
+CZ^3+\dfrac32(C+D)Z^2W
+3DZW^2+\dfrac{C-3D}{2}W^3
\right],\\[1ex]
\dot W
=
-W-\varepsilon\left[
W^2+\dfrac14ZW+\dfrac14Z^2
+DW^3+\dfrac32(C+D)ZW^2
+3CZ^2W+\dfrac{D-3C}{2}Z^3
\right].
\end{cases}
\]
We may write
\[
C=\frac{h(r+1)}{2},
\qquad
D=\frac{h(r-1)}{2},
\qquad
r^2=3.
\]

For $\varepsilon\neq0$, the scaling
\(
(\widetilde Z,\widetilde W)=(\varepsilon Z,\varepsilon W)
\)
reduces the proof to $\varepsilon=1$; the case $\varepsilon=0$ is
linear. Suppressing the tildes and relabeling the cubic parameter, set
\[
x=Z+W,\qquad
y=Z-W,\qquad
u=y+rx,
\]
and
\[
p=\dot u
=
ru+\frac{u^2}{2}+\frac{hu^3}{4}
-x\left(2+\frac{ru}{4}\right),
\qquad
\mu=\left(1+\frac{ru}{8}\right)^{-4}.
\]
Define
\[
T=\left(\frac{ru}{8+ru}\right)^2,
\qquad
F_0(T)
=
-\frac{14}{9}T(16rhT-3T+6),
\]
and
\[
\begin{aligned}
V_0(T)=&-\frac{8}{27}T\Big(768h^2T^3-1024h^2T^2
-96rhT^3+256rhT^2-192rhT+9T^3-36T^2\\
&+54T-36
\Big).
\end{aligned}
\]
With \(R=\mu p-F_0(T)\), we obtain
\[
\frac{dT}{dR}=\frac{R+F_0(T)}{V_0'(T)}.
\]
This equation is regular analytic at the origin and admits a local
analytic invariant satisfying
\[
K_{11}(T,R)
=
T-\frac{3}{64}R^2+O(3).
\]
Since \(u^2-p^2=8ZW+O(3)\), it follows that
\[
\widetilde H_{11}
=
\frac83K_{11}(T,R)
=
ZW+O(3)
\]
is an analytic first integral. Undoing the scalings and normalizing
the quadratic term gives
\[
H_{11}(z,w)=zw+O(3).
\]
Hence, the origin is a weakly persistent center.

When condition $\mathrm{F}_{12}$ holds, there exist
$r,\tau\in\mathbb C$ such that
\[
a_{20}b_{20}r(7\tau-2)\neq0.
\]
Under the normalization \eqref{eq:common_normalization}, put \(h=D\).
The defining relations of $\mathrm{F}_{12}$ give
\[
A=
\frac{6r\tau-2r-\tau}{r^2(7\tau-2)},
\quad
B=
\frac{r(6\tau-2-r\tau)}{7\tau-2},
 \quad C=hr^2,\]
\[
U=h(6\tau-3),\quad
V=hr^2(6\tau-3),\quad
E=\frac{h\tau}{r},
\quad
F=h\tau r^3,
\quad
\widehat S=hr(5\tau-2).
\]
Hence, system~\eqref{eq:normalized_general} becomes
\[
\begin{cases}
\dot Z=Z+\varepsilon\left(Z^2+AW^2+hr^2Z^3+hr(5\tau-2)Z^2W+h(6\tau-3)ZW^2+\dfrac{h\tau}{r}W^3\right),\\[1ex]
\dot W=-W-\varepsilon\left(W^2+BZ^2+hW^3+hr(5\tau-2)ZW^2+hr^2(6\tau-3)Z^2W+h\tau r^3Z^3\right).
\end{cases}
\]
Let \(x=rZ+W,\quad y=rZ-W\), 
and put
\[
\alpha=
\frac{\tau(1+1/r)}{7\tau-2},
\qquad
\beta=
\frac{(3\tau-1)(1+1/r)}{7\tau-2},
\qquad
\gamma=\frac1r-1.
\]
Then
\[
\dot x=yd(x),
\qquad
\dot y=xQ_0(x)+\varepsilon\gamma xy+A_0(x)y^2,
\]
where
\[
d(x)=1+\varepsilon\alpha x
+\varepsilon h(1-\tau)x^2,\quad
Q_0(x)=1+\varepsilon\beta x
+\varepsilon h(3\tau-1)x^2,
\quad
A_0(x)=\varepsilon\beta
+2\varepsilon h(1-\tau)x.
\]
Setting
\[
k(x)=\frac{d'(x)+A_0(x)}{d(x)},
\]
we have
\[
(dQ_0)'-kdQ_0
=
\left[
\varepsilon^2\beta(\alpha-\beta)
+4\varepsilon h(2\tau-1)
\right]x.
\]
Therefore, Lemma~\ref{lem_regular_IVP} yields, after returning to the
original coordinates, an analytic first integral
\[
H_{12}(z,w)=zw+O(3).
\]
Hence, the origin is a weakly persistent center.

When condition $\mathrm{F}_{13}$ holds, there exist $c,e\in\mathbb C$
such that
\[
a_{20}b_{20}ce(ce^2-1)\neq0.
\]
Under the normalization~\eqref{eq:common_normalization}, set
\[
h=D=\frac{b_{30}}{b_{20}^2}.
\]
Choose $r_0\in\mathbb C$ such that $r_0^2=c$, and let
\[
u=r_0Z,\qquad
v=W,\qquad
\kappa=\frac{1}{r_0},
\qquad
\theta=er_0.
\]
Then $\theta(\theta^2-1)\neq0$, and the defining relations of
$\mathrm{F}_{13}$ reduce the system to
\begin{equation}\label{eq:C15_certificate_system}
\begin{cases}
\displaystyle
\dot u
=
u+\varepsilon\left[
\kappa u^2
+\frac{\theta(\theta\kappa-1)}{\theta^2-1}v^2
+h\left(
u^3+\theta u^2v+\theta v^3
\right)
\right],\\[2ex]
\displaystyle
\dot v
=
-v-\varepsilon\left[
v^2
+\frac{\theta(\theta-\kappa)}{\theta^2-1}u^2
+h\left(
v^3+\theta uv^2+\theta u^3
\right)
\right].
\end{cases}
\end{equation}

Set
\[
q=\theta(u^2+v^2)+uv,
\qquad
\Delta=\theta^2-1,
\]
and
\[
m=\frac{2\kappa\theta^2-\kappa-\theta}{\Delta},
\qquad
n=\frac{2\theta^2-1-\kappa\theta}{\Delta},
\qquad
c_0=
\frac{(\kappa^2+1)\theta^3
-3\kappa\theta^2+\kappa}{\Delta^2}.
\]
Define
\[
V_1
=
1+\varepsilon\left(
\frac{\theta-\kappa}{\Delta}u
+
\frac{\kappa\theta-1}{\Delta}v
\right)
+\frac{\varepsilon h}{2\theta}q,
\]
and
\[
\begin{aligned}
V_2
={}&
1+\varepsilon(mu+nv)
+\varepsilon h(u^2+v^2+4\theta uv)
+\varepsilon^2c_0q
\\
&+\varepsilon^2h\left[
\theta n u^3
+(2\kappa\theta+1)u^2v
+(\kappa+2\theta)uv^2
+\theta m v^3
\right]
+\varepsilon^2h^2q^2.
\end{aligned}
\]
For
\[
V_{\kappa,\theta,h}=V_1V_2,
\]
direct polynomial differentiation gives
\begin{equation}\label{eq:C15_IIF_identity}
XV_{\kappa,\theta,h}
=
(\operatorname{div}X)V_{\kappa,\theta,h},
\qquad
V_{\kappa,\theta,h}(0,0)=1.
\end{equation}
Notice that neither $V_{\kappa,\theta,h}$ nor
\eqref{eq:C15_IIF_identity} involves division by $\kappa$; hence the
identity remains valid for $\kappa=0$ whenever
$\theta(\theta^2-1)\neq0$.

By Lemma~\ref{lem_inverse_integrating_factor},
system~\eqref{eq:C15_certificate_system} admits a local analytic first
integral
\[
\widetilde H_{13}(u,v)=uv+O(3).
\]
Since
\[
uv=r_0a_{20}b_{20}zw
\]
and $r_0a_{20}b_{20}\neq0$, pulling the first integral back and
normalizing its quadratic term yields
\[
H_{13}(z,w)=zw+O(3).
\]
Therefore, the origin is a weakly persistent center.

When condition $\mathrm{F}_{14}$ holds, under the normalization
\eqref{eq:common_normalization}, set \(r=A,\quad h=C=D\). 
The defining relations of $\mathrm{F}_{14}$ give
\[
B=1-r,\qquad
U=V=0,\qquad
E=F=\widehat S=h.
\]
Hence, system~\eqref{eq:normalized_general} reduces to
\[
\begin{cases}
\dot Z
=
Z+\varepsilon\left[
Z^2+rW^2
+h\left(Z^3+Z^2W+W^3\right)
\right],\\
\dot W
=
-W-\varepsilon\left[
W^2+(1-r)Z^2
+h\left(W^3+ZW^2+Z^3\right)
\right].
\end{cases}
\]
Set
\begin{equation}\label{eq:common_abq}
a=\varepsilon,\qquad
b=\varepsilon h,\qquad
q=Z^2+ZW+W^2.
\end{equation}
Define
\[
\begin{aligned}
F_1={}&
bq+2a\bigl[rW+(1-r)Z\bigr]+2,\\
F_2={}&
b^2\left(
W^4+2W^3Z+3W^2Z^2+2WZ^3+Z^4
\right)\\
&+ab\left[
(r+1)W^3+3W^2Z+3WZ^2+(2-r)Z^3
\right]\\
&+a^2(r^2-r+1)q
+b\left(W^2+4WZ+Z^2\right)
+a\left[(2-r)W+(r+1)Z\right]+1.
\end{aligned}
\]
Then \(V_{14}=\frac{F_1F_2}{2}\) satisfies
\[
XV_{14}=(\operatorname{div}X)V_{14},
\qquad
V_{14}(0,0)=1.
\]
Therefore, Lemma~\ref{lem_inverse_integrating_factor} yields, after
pulling back and normalizing, an analytic first integral
\[
H_{14}(z,w)=zw+O(3).
\]
Hence, the origin is a weakly persistent center.

When condition $\mathrm{F}_{15}$ holds, under the normalization
\eqref{eq:common_normalization}, set \(r=A,\enspace h=C=D.\)
The defining relations of $F_{15}$ give
\[
B=-2-r,\qquad
U=V=-\frac{9h}{7},\qquad
E=F=-\frac{2h}{7},
\qquad
\widehat S=\frac{4h}{7}.
\]
Hence, system~\eqref{eq:normalized_general} reduces to
\[
\begin{cases}
\dot Z
=
Z+\varepsilon\left[
Z^2+rW^2+hZ^3
+\dfrac{4h}{7}Z^2W
-\dfrac{9h}{7}ZW^2
-\dfrac{2h}{7}W^3
\right],\\
\dot W
=
-W-\varepsilon\left[
W^2-(2+r)Z^2+hW^3
+\dfrac{4h}{7}ZW^2
-\dfrac{9h}{7}Z^2W
-\dfrac{2h}{7}Z^3
\right].
\end{cases}
\]
Using the notation \eqref{eq:common_abq}, define
\[
G_1
=
5b(W-Z)^2+7a(r+1)(W-Z)+7,
\]
\[
\begin{aligned}
G_2={}&
b^2(W-Z)^4
+7ab\left[
r(W-Z)^3-W^3-W^2Z+5WZ^2-3Z^3
\right]\\
&+49a^2\left[
r^2q+r(W^2+2WZ+3Z^2)+WZ+2Z^2
\right]\\
&+b\left(70W^2+196WZ+70Z^2\right)
+49a\left[(1-r)W+(r+3)Z\right]+49.
\end{aligned}
\]
Then \(V_{15}=\frac{G_1G_2}{343}\) satisfies
\[
XV_{15}=(\operatorname{div}X)V_{15},
\qquad
V_{15}(0,0)=1.
\]
Therefore, Lemma~\ref{lem_inverse_integrating_factor} yields, after
pulling back and normalizing, an analytic first integral
\[
H_{15}(z,w)=zw+O(3).
\]
Hence, the origin is a weakly persistent center.

When condition $\mathrm{F}_{16}$ holds, under the normalization
\eqref{eq:common_normalization}, the defining relations give
\[
t=0,\qquad
AB=1,\qquad
D=-AC,\qquad
U=-D,\qquad
V=-C,
\]
and
\[
E=A^2F,\qquad
\widehat S=AF.
\]
In particular, $A\neq0$. Setting \(d=F\), 
system~\eqref{eq:normalized_general} reduces to
\[
\begin{cases}
\dot Z
=
Z+\varepsilon\left[
Z^2+AW^2
+CZ^3+AdZ^2W+ACZW^2+A^2dW^3
\right],\\
\dot W
=
-W-\varepsilon\left[
W^2+\dfrac1A Z^2
-ACW^3+AdZW^2-CZ^2W+dZ^3
\right].
\end{cases}
\]

If $C=0$, then $D=0$. Let \(Q_0=Z^2+AW^2\) and define
\[
\begin{aligned}
V_{16}={}&
1+\varepsilon\left[
2(Z+W)+4AdZW
\right]\\
&+\varepsilon^2\left[
\left(1+\frac1A\right)Z^2
+(1+A)W^2
+2d(Z+AW)Q_0
+Ad^2Q_0^2
\right].
\end{aligned}
\]
A direct calculation gives
\[
XV_{16}
=
(\operatorname{div}X)V_{16},
\qquad
V_{16}(0,0)=1.
\]
Hence, Lemma~\ref{lem_inverse_integrating_factor} yields an analytic
first integral with quadratic leading term $ZW$.

Suppose now that $C\neq0$. Choose $\rho\in\mathbb C$ such that \(\rho^2=-\frac1A\), and set \(h=-AC,\quad
j=\frac{A^2d}{h}\).
Then
\[
A=-\frac1{\rho^2},\quad
B=-\rho^2,\quad
C=h\rho^2,\quad
D=h,\quad E=hj,\quad
F=hj\rho^4,\quad
\widehat S=-hj\rho^2.
\]
With \(x=\rho Z+W,\qquad
y=\rho Z-W\),  the system becomes
\[
\dot x=yd_0(x),
\quad
\dot y=x+\varepsilon\gamma xy+\varepsilon\nu xy^2,
\]
where
\[
d_0(x)=1+\varepsilon\alpha x+\varepsilon\beta x^2,
\]
with
\[
\alpha=1+\frac1\rho,\qquad
\gamma=\frac1\rho-1,
\qquad
\beta=h(1-j\rho),
\qquad
\nu=h(1+j\rho).
\]
Setting
\[
k(x)=\frac{d_0'(x)+\varepsilon\nu x}{d_0(x)},
\]
we have
\[
d_0'(x)-k(x)d_0(x)=-\varepsilon\nu x.
\]
Therefore, Lemma~\ref{lem_regular_IVP} applies.

Thus, in both cases, after returning to the original coordinates and
normalizing the quadratic term, system~\eqref{sys_general} admits an
analytic first integral
\[
H_{16}(z,w)=zw+O(3).
\]
Hence, the origin is a weakly persistent center.

For conditions $\mathrm{F}_{17}$ and $\mathrm{F}_{18}$, under the
normalization \eqref{eq:common_normalization}, set $h=U$ and use the
transformation
\begin{equation}\label{eq:C19_C20_transformation}
x=rZ+W,\qquad
y=rZ-W,\qquad
b=\frac1r-1.
\end{equation}

When condition $\mathrm{F}_{17}$ holds, the defining relations give
\[
t=C=D=0,\qquad
A=\frac{2r+1}{r^2},\qquad
B=r(r+2),
\]
\[
\widehat S=\frac{3hr}{2},\qquad
V=hr^2,\qquad
E=-\frac{h}{2r},\qquad
F=-\frac{hr^3}{2}.
\]
Hence, system~\eqref{eq:normalized_general} reduces to
\[
\begin{cases}
\dot Z
=
Z+\varepsilon\left[
Z^2+\dfrac{2r+1}{r^2}W^2
+\dfrac{3hr}{2}Z^2W
+hZW^2-\dfrac{h}{2r}W^3
\right],\\[1ex]
\dot W
=
-W-\varepsilon\left[
W^2+r(r+2)Z^2
+\dfrac{3hr}{2}ZW^2
+hr^2Z^2W-\dfrac{hr^3}{2}Z^3
\right].
\end{cases}
\]

Under the transformation \eqref{eq:C19_C20_transformation}, set
$a=1+\frac{1}{r}$. Then
\[
\dot x=yd_0(x),
\qquad
\dot y=xQ_0(x)+\varepsilon bxy+A_0(x)y^2,
\]
where
\[
d_0(x)=1-\varepsilon ax+\frac{\varepsilon h}{2}x^2,
\qquad
Q_0(x)=1+\varepsilon ax+\frac{\varepsilon h}{2}x^2,
\]
and $A_0(x)=\varepsilon a-\varepsilon hx$. In particular,
\[
\frac{d_0'(x)+A_0(x)}{d_0(x)}=0.
\]

Set $p=yd_0(x)$, $T=\frac{x^2}{2}$, and $R=p-\varepsilon bT$. Then
\[
\dot T=x(R+\varepsilon bT),
\qquad
\dot R
=
x\left[
1+2(\varepsilon h-\varepsilon^2a^2)T
+\varepsilon^2h^2T^2
\right].
\]
Therefore, the extended form of
Lemma~\ref{lem_regular_IVP} yields, after pulling back and normalizing,
an analytic first integral
\[
H_{17}(z,w)=zw+O(3).
\]
Hence, the origin is a weakly persistent center.

When condition $\mathrm{F}_{18}$ holds, the defining relations give
\[
t=C=D=0,\qquad
A=\frac{6r-1}{7r^2},\qquad
B=\frac{r(6-r)}{7},
\]
\[
\widehat S=\frac{5hr}{6},\qquad
V=hr^2,\qquad
E=\frac{h}{6r},\qquad
F=\frac{hr^3}{6}.
\]
Hence, system~\eqref{eq:normalized_general} reduces to
\[
\begin{cases}
\dot Z
=
Z+\varepsilon\left[
Z^2+\dfrac{6r-1}{7r^2}W^2
+\dfrac{5hr}{6}Z^2W
+hZW^2+\dfrac{h}{6r}W^3
\right],\\[1ex]
\dot W
=
-W-\varepsilon\left[
W^2+\dfrac{r(6-r)}{7}Z^2
+\dfrac{5hr}{6}ZW^2
+hr^2Z^2W+\dfrac{hr^3}{6}Z^3
\right].
\end{cases}
\]

Using \eqref{eq:C19_C20_transformation}, set
$a=(1+1/r)/7$. Then
\[
\dot x=yd_0(x),
\qquad
\dot y=xQ_0(x)+\varepsilon bxy+A_0(x)y^2,
\]
where
\[
d_0(x)=1+\varepsilon ax-\frac{\varepsilon h}{6}x^2,
\qquad
Q_0(x)=1+3\varepsilon ax+\frac{\varepsilon h}{2}x^2,
\]
and $A_0(x)=3\varepsilon a-\varepsilon hx/3$. Setting
\[
k(x)
=
\frac{d_0'(x)+A_0(x)}{d_0(x)}
=
\frac{4\varepsilon a-\dfrac{2\varepsilon h}{3}x}
     {d_0(x)},
\]
we obtain
\[
(d_0Q_0)'-kd_0Q_0
=
\left(
\frac{4\varepsilon h}{3}
-6\varepsilon^2a^2
\right)x.
\]
Therefore, Lemma~\ref{lem_regular_IVP} yields, after pulling back and
normalizing, an analytic first integral
\[
H_{18}(z,w)=zw+O(3).
\]
Hence, the origin is a weakly persistent center.

When condition $\mathrm{F}_{19}(a)$ holds, system~\eqref{sys_general} can be
rewritten as
\begin{equation}\label{eq:C21_reduced}
\begin{cases}
\displaystyle
\dot z
=
z+\varepsilon\left(
-\frac{b_{20}}{2}zw+a_{30}z^3
\right),\\[1ex]
\displaystyle
\dot w
=
-w-\varepsilon\left(
b_{20}w^2+b_{02}z^2-2a_{30}z^2w
\right).
\end{cases}
\end{equation}
Set
\[
W_{19}(z,w)
=
1+\varepsilon b_{20}w
+\varepsilon\left(
a_{30}+\frac{\varepsilon b_{20}b_{02}}{2}
\right)z^2.
\]
A direct calculation gives
\[
XW_{19}
=
\varepsilon\left(
2a_{30}z^2-b_{20}w
\right)W_{19},
\]
and
\[
\operatorname{div}X
=
\frac{5}{2}\varepsilon\left(
2a_{30}z^2-b_{20}w
\right).
\]
Therefore, by Lemma~\ref{lem_inverse_integrating_factor},
\[
V_{19}(z,w)
=
W_{19}^{5/2}
=
\left[
1+\varepsilon b_{20}w
+\varepsilon\left(
a_{30}+\frac{\varepsilon b_{20}b_{02}}{2}
\right)z^2
\right]^{5/2}
\]
is an analytic inverse integrating factor near the origin. Since
$W_{19}(0,0)=1$, the local analytic branch satisfying
$V_{19}(0,0)=1$ is well defined. Hence,
system~\eqref{eq:C21_reduced} admits a local analytic first integral
\[
H_{19}(z,w)=zw+O(3).
\]
Therefore, the origin is a weakly persistent center.

Condition $\mathrm{F}_{19}(b)$ is obtained from $\mathrm{F}_{19}(a)$ by
exchanging the two coefficient sets
\[
a_{ij}\longleftrightarrow b_{ij}.
\]
Thus, by Lemma~\ref{lem_exchange} and the result for
system~\eqref{eq:C21_reduced}, the origin is also a weakly persistent
center under condition $\mathrm{F}_{19}(b)$.

When condition $\mathrm{F}_{20}(a)$ holds, system~\eqref{sys_general} can be
rewritten as
\begin{equation}\label{eq:C23_reduced}
\begin{cases}
\displaystyle
\dot z
=
z-\frac{\varepsilon b_{20}}{3}zw,\\[1ex]
\displaystyle
\dot w
=
-w-\varepsilon\left(
b_{20}w^2+b_{03}z^3
\right).
\end{cases}
\end{equation}
Set
\[
W_{20}(z,w)
=
1+\varepsilon b_{20}w
+\frac{\varepsilon^2b_{20}b_{03}}{3}z^3.
\]
A direct calculation gives
\[
XW_{20}
=
-\varepsilon b_{20}w\,W_{20},\quad
\operatorname{div}X
=
-\frac{7}{3}\varepsilon b_{20}w.
\]
Therefore, by Lemma~\ref{lem_inverse_integrating_factor},
\[
V_{20}(z,w)
=
W_{20}^{7/3}
=
\left(
1+\varepsilon b_{20}w
+\frac{\varepsilon^2b_{20}b_{03}}{3}z^3
\right)^{7/3}
\]
is an analytic inverse integrating factor near the origin. Since
$W_{20}(0,0)=1$, system~\eqref{eq:C23_reduced} admits a local analytic
first integral
\[
H_{20}(z,w)=zw+O(3).
\]
Therefore, the origin is a weakly persistent center.

Condition $\mathrm{F}_{20}(b)$ is the exchanged form of
$\mathrm{F}_{20}(a)$. Hence, by Lemma~\ref{lem_exchange}, the origin is
also a weakly persistent center.

When condition $\mathrm{F}_{21}(a)$ holds, system~\eqref{sys_general} can be
rewritten as
\begin{equation}\label{eq:C25_reduced}
\begin{cases}
\displaystyle
\dot z
=
z+\varepsilon\left(
\frac{b_{20}}{3}zw-\frac{b_{30}}{3}zw^2
\right),\\[1ex]
\displaystyle
\dot w
=
-w-\varepsilon\left(
b_{20}w^2+b_{30}w^3+b_{03}z^3
\right).
\end{cases}
\end{equation}
Set
\[
f(w)
=
1+\varepsilon b_{20}w+\varepsilon b_{30}w^2,
\qquad
g(w)
=
f(w)-\frac{2}{3}wf'(w).
\]
Then
\[
g(w)
=
1+\frac{\varepsilon b_{20}}{3}w
-\frac{\varepsilon b_{30}}{3}w^2,
\]
and system~\eqref{eq:C25_reduced} takes the form
\[
\dot z=zg(w),
\qquad
\dot w=-wf(w)-\varepsilon b_{03}z^3.
\]

Let $p_0(w)=w$, and define $p_m(w)$ recursively by
\[
\left(
w\frac{d}{dw}-(3m+1)
\right)p_m(w)
=
-f(w)p_{m-1}'(w)
+
\left(
2m-\frac{4}{3}
\right)
f'(w)p_{m-1}(w),
\qquad m\geq1.
\]
We first show that this recursion uniquely determines $p_m$ and that
\[
\deg p_m\leq m+1.
\]
Indeed, if $\deg p_{m-1}\leq m$, then the right-hand side of the above
recursion is a polynomial of degree at most $m+1$. Moreover, on a
monomial $w^j$, the operator on the left-hand side acts as
\[
\left(
w\frac{d}{dw}-(3m+1)
\right)w^j
=
\bigl(j-(3m+1)\bigr)w^j.
\]
Since $0\leq j\leq m+1$, we have
\[
|j-(3m+1)|\geq 2m.
\]
Hence none of the corresponding divisors vanishes, so the recursion
uniquely determines each $p_m$, and $\deg p_m\leq m+1$ follows by
induction.

For a polynomial
\[
P(w)=\sum_j c_jw^j,
\]
define
\[
\|P\|_1=\sum_j|c_j|.
\]
Using the preceding lower bound for the divisors, we obtain
\[
\|p_m\|_1
\leq
\frac{1}{2m}
\left[
\|f p_{m-1}'\|_1
+
\left(
2m-\frac{4}{3}
\right)
\|f'p_{m-1}\|_1
\right].
\]
Since $\deg p_{m-1}\leq m$,
\[
\|p_{m-1}'\|_1
\leq
m\|p_{m-1}\|_1.
\]
Therefore,
\[
\begin{aligned}
\|p_m\|_1
&\leq
\frac{1}{2m}
\left[
m\|f\|_1
+
\left(
2m-\frac{4}{3}
\right)
\|f'\|_1
\right]
\|p_{m-1}\|_1\\
&\leq
\left(
\frac{\|f\|_1}{2}
+
\|f'\|_1
\right)
\|p_{m-1}\|_1.
\end{aligned}
\]
Since $\|p_0\|_1=1$, it follows inductively that
\[
\|p_m\|_1
\leq
\left(
\frac{\|f\|_1}{2}
+
\|f'\|_1
\right)^m.
\]

Now define
\[
H_{21}(z,w)
=
zf(w)^{-\frac{2}{3}}
\sum_{m=0}^{\infty}
\left(
\frac{\varepsilon b_{03}z^3}{f(w)^2}
\right)^m
p_m(w).
\]
Since $f(0)=1$, we may restrict to a sufficiently small neighborhood
of the origin in which $f(w)\neq0$. Moreover,
\[
|p_m(w)|
\leq
\|p_m\|_1
\leq
\left(
\frac{\|f\|_1}{2}
+
\|f'\|_1
\right)^m
\]
for $|w|$ sufficiently small. Hence the $m$th term of the series is
bounded by
\[
\left[
\left|
\frac{\varepsilon b_{03}z^3}{f(w)^2}
\right|
\left(
\frac{\|f\|_1}{2}
+
\|f'\|_1
\right)
\right]^m.
\]
Therefore, after further restricting the neighborhood if necessary so
that
\[
\left|
\frac{\varepsilon b_{03}z^3}{f(w)^2}
\right|
\left(
\frac{\|f\|_1}{2}
+
\|f'\|_1
\right)
<1,
\]
the series defining $H_{21}$ converges normally and hence defines an
analytic function near the origin.

Direct substitution, together with the recursion for $p_m$, gives
\[
XH_{21}=0.
\]
Since $p_0(w)=w$ and $f(0)=1$, we have
\[
H_{21}(z,w)=zw+O(3).
\]
Hence, $H_{21}$ is a local analytic first integral, and the origin is a
weakly persistent center.

When condition $\mathrm{F}_{21}(b)$ holds, the conclusion follows from
the corresponding exchanged system by Lemma~\ref{lem_exchange}.

When condition $\mathrm{F}_{22}(a)$ holds, system~\eqref{sys_general} can be
rewritten as
\[
\begin{cases}
\displaystyle
\dot z
=
z+\varepsilon\left(
a_{02}w^2+a_{30}z^3
+a_{21}z^2w-3b_{30}zw^2
\right),\\[1ex]
\displaystyle
\dot w
=
-w-\varepsilon\left(
b_{20}w^2+b_{30}w^3
+a_{21}zw^2-3a_{30}z^2w
\right),
\end{cases}
\]
where
\[
b_{20}\neq0,\qquad
b_{20}^2b_{30}=a_{02}^2a_{30},\qquad
a_{21}b_{20}=-2a_{02}a_{30}.
\]
Taking
\[
Z=z,\qquad W=b_{20}w,
\qquad
A=\frac{a_{02}}{b_{20}^2},
\qquad
C=a_{30},
\]
we obtain
\begin{equation}\label{eq:C27_normalized}
\begin{cases}
\dot Z
=
Z+\varepsilon\left(
AW^2+CZ^3-2ACZ^2W-3A^2CZW^2
\right),\\[1ex]
\dot W
=
-W-\varepsilon\left(
W^2+A^2CW^3-2ACZW^2-3CZ^2W
\right).
\end{cases}
\end{equation}

If $A=0$, system~\eqref{eq:C27_normalized} reduces to
\[
\dot Z=Z+\varepsilon CZ^3,\qquad
\dot W=-W-\varepsilon\left(W^2-3CZ^2W\right).
\]
Set
\[
p(Z)=1+\varepsilon CZ^2,\qquad
\mu(Z)=p(Z)^{-2},\qquad
q=\mu(Z)ZW.
\]
A direct calculation shows that
\[
\widetilde H_{22}(Z,W)
=
\frac{q}
{1+\varepsilon\mu(Z)W-\varepsilon^2CZq}
\]
satisfies
\[
X\widetilde H_{22}=0,
\qquad
\widetilde H_{22}(Z,W)=ZW+O(3).
\]

Now suppose that $A\neq0$. Taking
\[
x=Z+AW,\qquad y=Z-AW,
\]
system~\eqref{eq:C27_normalized} is transformed into
\[
\dot x=yd(x),\qquad
\dot y=xQ_0(x)+bxy+A_0(x)y^2,
\]
where
\[
a=\frac{\varepsilon}{2A},\quad
c=\varepsilon C,\quad
d(x)=1+cx^2,\quad
Q_0(x)=1+ax-cx^2,
\quad
b=-2a,\quad
A_0(x)=a+2cx.
\]
For
\[
k(x)=\frac{d'(x)+A_0(x)}{d(x)},
\]
a direct calculation gives
\[
\bigl(d(x)Q_0(x)\bigr)'
-k(x)d(x)Q_0(x)
=
-(a^2+4c)x.
\]
Hence, by Lemma~\ref{lem_regular_IVP}, system~\eqref{eq:C27_normalized}
admits a local analytic first integral
\[
\widetilde H_{22}(Z,W)=ZW+O(3).
\]
Thus, in either case, pulling the first integral back to the original
coordinates and normalizing its quadratic part gives
\[
H_{22}(z,w)=zw+O(3).
\]
Therefore, the origin is a weakly persistent center.

Condition $\mathrm{F}_{22}(b)$ is the exchanged form of
$\mathrm{F}_{22}(a)$. Hence, by Lemma~\ref{lem_exchange}, the origin is
also a weakly persistent center.

When condition $\mathrm{F}_{23}(a)$ holds, system~\eqref{sys_general} can be
rewritten as
\[
\begin{cases}
\displaystyle
\dot z
=
z+\varepsilon\left(
a_{30}z^3+a_{21}z^2w
-b_{30}zw^2+a_{03}w^3
\right),\\[1ex]
\displaystyle
\dot w
=
-w-\varepsilon\left(
b_{20}w^2+b_{02}z^2+b_{30}w^3
+a_{21}zw^2-a_{30}z^2w+b_{03}z^3
\right),
\end{cases}
\]
where
\[
b_{02}b_{20}\neq0,\quad
b_{02}b_{30}=-a_{30}b_{20},
\quad
b_{03}b_{20}^2=a_{03}b_{02}^2,\quad
a_{21}b_{20}=a_{03}b_{02}.
\]
Choose $\lambda\in\mathbb C\setminus\{0\}$ such that
$\lambda^2=b_{20}b_{02}$, and set
\[
Z=\lambda z,\qquad
W=b_{20}w,\qquad
C=\frac{a_{30}}{b_{20}b_{02}},\qquad
E=\frac{\lambda a_{03}}{b_{20}^3}.
\]
Then the system becomes
\begin{equation}\label{eq:C29_normalized}
\begin{cases}
\dot Z
=
Z+\varepsilon\left(
CZ^3+EZ^2W+CZW^2+EW^3
\right),\\[1ex]
\dot W
=
-W-\varepsilon\left(
W^2+Z^2-CW^3
+EZW^2-CZ^2W+EZ^3
\right).
\end{cases}
\end{equation}
Let $q=Z^2+W^2$ and define
\[
\begin{aligned}
V_{23}(Z,W)
={}
1+2\varepsilon\left[
W+C(Z^2-W^2)+2EZW
\right]+\varepsilon^2q
\left[
1-2CW+2EZ+(C^2+E^2)q
\right].
\end{aligned}
\]
A direct calculation gives
\[
XV_{23}=(\operatorname{div}X)V_{23},
\qquad
V_{23}(0,0)=1.
\]
Therefore, by Lemma~\ref{lem_inverse_integrating_factor},
system~\eqref{eq:C29_normalized} admits a local analytic first integral
with leading term $ZW$. Since $ZW=\lambda b_{20}zw$ and
$\lambda b_{20}\neq0$, after pulling it back and normalizing we obtain
\[
H_{23}(z,w)=zw+O(3).
\]
Hence, the origin is a weakly persistent center.

Condition $\mathrm{F}_{23}(b)$ is the exchanged form of
$\mathrm{F}_{23}(a)$. Hence, by Lemma~\ref{lem_exchange}, the origin is
also a weakly persistent center.

When condition $\mathrm{F}_{24}(a)$ holds, system~\eqref{sys_general} can be
rewritten as
\[
\begin{cases}
\displaystyle
\dot z
=
z+\varepsilon\left(
a_{02}w^2+a_{21}z^2w+a_{12}zw^2+a_{03}w^3
\right),\\[1ex]
\displaystyle
\dot w
=
-w-\varepsilon\left(
b_{20}w^2+b_{02}z^2+a_{21}zw^2
+b_{12}z^2w+b_{03}z^3
\right),
\end{cases}
\]
where
\[
a_{02}b_{20}\neq0,
\qquad
a_{02}^2b_{02}=4b_{20}^3,
\qquad
a_{02}^2b_{12}=4a_{12}b_{20}^2,
\]
\[
4a_{03}b_{20}=-a_{02}a_{12},
\qquad
a_{02}b_{03}=-b_{12}b_{20},
\qquad
a_{02}a_{21}=3a_{12}b_{20}.
\]
Introduce the normalization
\begin{equation}\label{eq:one_sided_normalization}
A=\frac{a_{02}}{b_{20}^2},
\qquad
Z=\frac{z}{A},
\qquad
W=b_{20}w.
\end{equation}
Since $a_{02}b_{20}\neq0$, this transformation is invertible. Setting
\[
h=\frac{4a_{12}}{b_{20}^2},
\]
the defining relations of $\mathrm{F}_{24}(a)$ give
\begin{equation}\label{eq:C31_normalized}
\begin{cases}
\displaystyle
\dot Z
=
Z+\varepsilon\left(
W^2+\frac{3h}{4}Z^2W
+\frac{h}{4}ZW^2-\frac{h}{16}W^3
\right),\\[1ex]
\displaystyle
\dot W
=
-W-\varepsilon\left(
W^2+4Z^2+\frac{3h}{4}ZW^2
+hZ^2W-hZ^3
\right).
\end{cases}
\end{equation}

Taking
\[
x=2Z+W,\qquad y=2Z-W,
\]
and setting
\[
a=\varepsilon,
\qquad
c=\frac{\varepsilon h}{8},
\]
system~\eqref{eq:C31_normalized} becomes
\[
\dot x=yd(x),
\qquad
\dot y=xQ_0(x)-axy+(a-2cx)y^2,
\]
where
\[
d(x)=1-ax+cx^2,
\qquad
Q_0(x)=1+ax+cx^2.
\]
Since
\[
d'(x)+a-2cx=0,
\]
let
\[
p=yd(x),
\qquad
T=\frac{x^2}{2},
\qquad
R=p+aT.
\]
A direct calculation gives
\[
\dot T=x(R-aT),
\qquad
\dot R
=
x\left[
1+2(2c-a^2)T+4c^2T^2
\right].
\]
Hence,
\[
\frac{dT}{dR}
=
\frac{R-aT}
{1+2(2c-a^2)T+4c^2T^2}
\]
is regular analytic near $(T,R)=(0,0)$ and admits a local analytic
invariant $K_{24}(T,R)$ satisfying $K_{24}(T,0)=T$. Its leading terms
are
\[
K_{24}(T,R)
=
T-\frac{R^2}{2}+O(3)
=
4ZW+O(3).
\]
Therefore,
\[
\widetilde H_{24}(Z,W)
=
\frac14K_{24}(T,R)
=
ZW+O(3)
\]
is a local analytic first integral. Pulling it back through
\eqref{eq:one_sided_normalization} and normalizing its quadratic term
gives
\[
H_{24}(z,w)=zw+O(3).
\]
Therefore, the origin is a weakly persistent center.

Condition $\mathrm{F}_{24}(b)$ is the exchanged form of
$\mathrm{F}_{24}(a)$. Hence, by Lemma~\ref{lem_exchange}, the origin is
also a weakly persistent center.

When condition $\mathrm{F}_{25}(a)$ holds, there exists $r\in\mathbb C$
such that
\[
a_{02}b_{20}r(r-1)\neq0,
\]
and system~\eqref{sys_general} can be rewritten as
\[
\begin{cases}
\dot z
=
z+\varepsilon\left(
a_{02}w^2+a_{30}z^3+a_{21}z^2w+a_{03}w^3
\right),\\[1ex]
\dot w
=
-w-\varepsilon\left(
b_{20}w^2+b_{02}z^2+b_{30}w^3
+a_{21}zw^2+b_{03}z^3
\right),
\end{cases}
\]
where
\[
a_{02}^2b_{02}
=
b_{20}^3r^2(r-1),
\qquad
a_{02}^2a_{30}
=
b_{20}^2b_{30}r(r-1),
\]
\[
a_{03}b_{20}(r-1)
=
-a_{02}b_{30},
\qquad
a_{02}b_{03}
=
-b_{02}b_{30},
\qquad
a_{02}a_{21}
=
-b_{20}b_{30}r.
\]

Applying the normalization~\eqref{eq:one_sided_normalization} and
setting \(h=\frac{b_{30}}{b_{20}^2}\), we obtain
\begin{equation}\label{eq:C33_normalized}
\begin{cases}
\dot Z
=
Z+\varepsilon\left[
W^2+hr(r-1)Z^3
-hrZ^2W-\dfrac{h}{r-1}W^3
\right],\\[1ex]
\dot W
=
-W-\varepsilon\left[
W^2+r^2(r-1)Z^2+hW^3
-hrZW^2-hr^2(r-1)Z^3
\right].
\end{cases}
\end{equation}

Choose $\theta\in\mathbb C$ such that
\[
\theta^2=\frac{r}{r-1},
\]
and set
\[
\rho=-\frac{\theta}{\theta^2-1}.
\]
Since $r(r-1)\neq0$, we have
\[
\theta(\theta^2-1)\neq0,
\qquad
\rho^2=r(r-1),
\qquad
-\frac{r}{\rho}=\theta.
\]
Under the scaling
\[
u=\rho Z,\qquad v=W,
\]
system~\eqref{eq:C33_normalized} reduces precisely to
system~\eqref{eq:C15_certificate_system} with $\kappa=0$.

By \eqref{eq:C15_IIF_identity}, the corresponding function
$V_{0,\theta,h}$ satisfies
\[
XV_{0,\theta,h}
=
(\operatorname{div}X)V_{0,\theta,h},
\qquad
V_{0,\theta,h}(0,0)=1.
\]
Therefore, Lemma~\ref{lem_inverse_integrating_factor} yields a local
analytic first integral with leading term $uv$. Since all the above
scalings are invertible, pulling this first integral back and
normalizing its quadratic term gives
\[
H_{25}(z,w)=zw+O(3).
\]
Therefore, the origin is a weakly persistent center.

Condition $\mathrm{F}_{25}(b)$ is the exchanged form of
$\mathrm{F}_{25}(a)$. Hence, by Lemma~\ref{lem_exchange}, the origin is
also a weakly persistent center.

When condition $\mathrm{F}_{26}(a)$ holds, there exists $r\in\mathbb C$
such that
\[
a_{02}b_{20}r(r-1)\neq0,
\]
and system~\eqref{sys_general} can be rewritten as
\[
\begin{cases}
\dot z
=
z+\varepsilon\left(
a_{02}w^2+a_{30}z^3+a_{21}z^2w
+a_{12}zw^2+a_{03}w^3
\right),\\[1ex]
\dot w
=
-w-\varepsilon\left(
b_{20}w^2+b_{02}z^2+b_{30}w^3
+a_{21}zw^2+b_{12}z^2w+b_{03}z^3
\right),
\end{cases}
\]
where
\[
a_{02}^2b_{02}
=
b_{20}^3r^2(r-1),
\quad
a_{02}^2a_{30}
=
b_{20}^2b_{30}r^2,\quad
a_{02}^2b_{12}
=
a_{12}b_{20}^2r^2,
\quad
a_{12}(7r-6)
=
b_{30}(6-9r),
\]
\[
6a_{03}b_{20}r
=
-a_{02}(7a_{12}+9b_{30})(r-1),
\quad
6a_{02}^3b_{03}
=
-b_{20}^3r^3(7a_{12}+9b_{30})(r-1),
\]
and
\[
6a_{02}a_{21}
=
-b_{20}r(1-2r)(7a_{12}+9b_{30}).
\]

Applying the normalization~\eqref{eq:one_sided_normalization}, set
\[
D=\frac{b_{30}}{b_{20}^2},
\qquad
U=\frac{a_{12}}{b_{20}^2}.
\]
The defining relations give
\[
U(7r-6)=D(6-9r).
\]
The relation
\[
U(7r-6)=3D(2-3r)
\]
can be parameterized as
\[
D=(7r-6)v,
\quad
U=3(2-3r)v,
\]
for some $v\in\mathbb C$. This parametrization also includes
$r=6/7$, in particular, no division by $7r-6$ is required.
 Substituting the remaining
relations, the normalized system becomes
\begin{equation}\label{eq:C35_normalized}
\begin{cases}
\begin{aligned}
\dot Z &=Z+\varepsilon\left[
W^2+r^2(7r-6)vZ^3
+2r(1-2r)vZ^2W
+3(2-3r)vZW^2
+\frac{2(r-1)v}{r}W^3
\right],\\
\dot W
&=-W-\varepsilon\left[
W^2+r^2(r-1)Z^2
+(7r-6)vW^3
+2r(1-2r)vZW^2 \right.
+3r^2(2-3r)vZ^2W
\\
&\quad\left.+2(r-1)r^3vZ^3
\right].
\end{aligned}
\end{cases}
\end{equation}

Taking
\[
x=rZ+W,\qquad y=rZ-W,
\]
system~\eqref{eq:C35_normalized} is transformed into
\[
\dot x=yd(x),
\qquad
\dot y=xQ_0(x)+bxy+A_0(x)y^2,
\]
where
\[
d(x)
=
1+\varepsilon(1-r)x
+\varepsilon v(5r-4)x^2,
\quad
Q_0(x)
=
1+\frac{\varepsilon r}{2}x
-\varepsilon vrx^2,
\quad
b=-\varepsilon,
\]
and
\[
A_0(x)
=
\frac{\varepsilon r}{2}
+2\varepsilon v(5r-4)x.
\]
For
\[
k(x)=\frac{d'(x)+A_0(x)}{d(x)},
\]
a direct calculation gives
\[
\bigl(d(x)Q_0(x)\bigr)'
-k(x)d(x)Q_0(x)
=
\gamma x,
\]
where
\[
\gamma
=
(2-3r)
\left(
\frac{\varepsilon^2r}{4}
+4\varepsilon v
\right).
\]
Hence, by Lemma~\ref{lem_regular_IVP},
system~\eqref{eq:C35_normalized} admits a local analytic first integral
with leading term $ZW$. Pulling it back through
\eqref{eq:one_sided_normalization} and normalizing the quadratic term
gives
\[
H_{26}(z,w)=zw+O(3).
\]
Therefore, the origin is a weakly persistent center.

Condition $\mathrm{F}_{26}(b)$ is the exchanged form of
$\mathrm{F}_{26}(a)$. Hence, by Lemma~\ref{lem_exchange}, the origin is
also a weakly persistent center.

When condition $\mathrm{F}_{27}(a)$ holds, system~\eqref{sys_general}
reduces to
\begin{equation}\label{eq:C37_reduced}
\begin{cases}
\displaystyle
\dot z
=
z-\frac{\varepsilon b_{30}}{2}zw^2,\\[1ex]
\displaystyle
\dot w
=
-w-\varepsilon\left(
b_{02}z^2+b_{30}w^3
\right).
\end{cases}
\end{equation}
Set
\[
\Phi_{27}(z,w)
=
1+\varepsilon b_{30}w^2
+2\varepsilon^2b_{02}b_{30}z^2w
+\frac{\varepsilon^3b_{02}^2b_{30}}{2}z^4.
\]
A direct calculation gives
\[
X\Phi_{27}
=
-2\varepsilon b_{30}w^2\Phi_{27},
\qquad
\operatorname{div}X
=
-\frac{7}{2}\varepsilon b_{30}w^2.
\]
Therefore,
\[
V_{27}(z,w)
=
\Phi_{27}(z,w)^{7/4}
\]
is an analytic inverse integrating factor near the origin. Since
$\Phi_{27}(0,0)=1$, Lemma~\ref{lem_inverse_integrating_factor}
yields a local analytic first integral
\[
H_{27}(z,w)=zw+O(3).
\]
Therefore, the origin is a weakly persistent center.

Condition $\mathrm{F}_{27}(b)$ is the exchanged form of
$\mathrm{F}_{27}(a)$. Hence, by Lemma~\ref{lem_exchange}, the origin is
also a weakly persistent center.

When condition $\mathrm{F}_{28}$ holds, system~\eqref{sys_general}
reduces to
\[
\begin{cases}
\displaystyle
\dot z
=
z+\varepsilon\left(
a_{30}z^3+a_{12}zw^2+a_{03}w^3
\right),\\[1ex]
\displaystyle
\dot w
=
-w-\varepsilon\left(
b_{30}w^3+b_{12}z^2w+b_{03}z^3
\right),
\end{cases}
\]
where
\[
a_{30}b_{30}\neq0,
\quad
3a_{12}=-b_{30},
\quad
3b_{12}=-a_{30},
\quad
9a_{03}b_{03}=4a_{30}b_{30}.
\]
Choose $\alpha,\beta\in\mathbb C\setminus\{0\}$ such that
\[
\alpha^2=a_{30},
\quad
\beta^2=b_{30},
\]
and set
\[
Z=\alpha z,\quad W=\beta w,
\quad
a=\frac{\alpha a_{03}}{\beta^3}.
\]
Since $9a_{03}b_{03}=4a_{30}b_{30}$ and
$a_{30}b_{30}\neq0$, we have $a\neq0$, and the system becomes
\begin{equation}\label{eq:C39_normalized}
\begin{cases}
\displaystyle
\dot Z
=
Z+\varepsilon\left(
Z^3-\frac13ZW^2+aW^3
\right),\\[1ex]
\displaystyle
\dot W
=
-W-\varepsilon\left(
W^3-\frac13Z^2W+\frac{4}{9a}Z^3
\right).
\end{cases}
\end{equation}

Let
\[
Q_0(Z,W)
=
\frac{a}{2}W^2
+\frac{2}{3}ZW
+\frac{2}{9a}Z^2,
\]
and define
\[
\Phi_{28}(Z,W)
=
1+\varepsilon\left[
\frac23(Z^2+W^2)
+\left(
a+\frac{4}{9a}
\right)ZW
\right]
+\varepsilon^2Q_0(Z,W)^2.
\]
A direct calculation gives
\[
X\Phi_{28}
=
\frac{4}{3}\varepsilon(Z^2-W^2)\Phi_{28}
=
\frac{2}{5}(\operatorname{div}X)\Phi_{28}.
\]
Hence,
\[
V_{28}(Z,W)
=
\Phi_{28}(Z,W)^{5/2}
\]
is an analytic inverse integrating factor near the origin. Since
$\Phi_{28}(0,0)=1$, Lemma~\ref{lem_inverse_integrating_factor}
gives a local analytic first integral with leading term $ZW$.
Because $ZW=\alpha\beta zw$ and $\alpha\beta\neq0$, pulling it back
and normalizing the quadratic term yields
\[
H_{28}(z,w)=zw+O(3).
\]
Therefore, the origin is a weakly persistent center.
\end{proof}

\begin{remark}
The completeness of the necessary parameter decomposition is established by exact computer-algebraic certificates provided in the Supplementary Material. More precisely, the obstruction ideal through order seven is decomposed by an exhaustive system of saturated algebraic charts, with every complementary zero locus treated separately, yielding $V(I_7)=V(I_{\infty})$. Thus no additional parameter components can arise from higher-order singular point quantities. The sufficiency of each resulting family has been proved analytically above.
\end{remark}
\subsection{Reduction to the real cubic system}
\label{subsec_real_reduction}

The preceding classification is obtained for the general complex cubic
system, where the two sets of coefficients $a_{ij}$ and $b_{ij}$ are
independent. We now briefly explain its relation to the corresponding
real cubic system.

With the notation adopted in system~\eqref{sys_general}, impose the
conjugacy conditions
\begin{equation}\label{eq:conjugacy_condition}
\begin{aligned}
b_{20}&=\overline{a_{20}},&
b_{11}&=\overline{a_{11}},&
b_{02}&=\overline{a_{02}},\\
b_{30}&=\overline{a_{30}},&
b_{21}&=\overline{a_{21}},&
b_{12}&=\overline{a_{12}},&
b_{03}&=\overline{a_{03}}.
\end{aligned}
\end{equation}
After multiplying the vector field by $i$, system~\eqref{sys_general}
takes the form
\[
\begin{cases}
\dot z
=
iz+i\varepsilon F(z,w),\\
\dot w
=
-iw-i\varepsilon G(z,w),
\end{cases}
\]
where, under \eqref{eq:conjugacy_condition}, the second equation is the
complex conjugate of the first one on $w=\overline z$. Hence, setting
$z=x+iy$ and $w=x-iy$ reduces the complex system to a real planar
cubic differential system.

The correspondence between weakly persistent centers of complex
systems and those of the associated real systems under the conjugacy
condition is well known; see, for example, Chen et al.~\cite{CHEN2014}.
Therefore, the classification obtained above immediately yields the
following consequence.

\begin{corollary}\label{cor_real_cubic_center}
Under the conjugacy conditions \eqref{eq:conjugacy_condition}, the
origin of the corresponding real cubic system is a weakly persistent
center if and only if the coefficients satisfy one of the weakly
persistent center conditions given in the preceding theorem, restricted
to \eqref{eq:conjugacy_condition}.
\end{corollary}

Thus, the weakly persistent center problem for the corresponding
general real cubic system is obtained directly as the conjugate
restriction of the complex classification.
\section{Weakly persistent isochronous centers in the general complex cubic system} \label{sec_cubic_iso} 

We now investigate the weakly persistent isochronous center problem for system~\eqref{sys_general}. Since analytic linearizability implies the existence of an analytic first integral, every weakly persistent isochronous center is necessarily a weakly persistent center. Hence, the weakly persistent center conditions obtained in the preceding section serve as the starting point for the present analysis.

Using the linearization quantities introduced in the preliminary section, we compute some linearization quantities at first, the finite collection of necessary conditions is decomposed completely into candidate algebraic components, and then every surviving component is also proved sufficient by an independent analytic argument. Then we obtain a complete classification of the weakly persistent isochronous centers of system~\eqref{sys_general}. In terms of the original parameters, the classification consists of $13$ basic families, which yield $21$ explicit parameter conditions after the corresponding exchanged cases are included. Our main result is stated as follows.

\begin{thm}\label{thm_cubic_iso}
The origin of system~\eqref{sys_general} is a weakly persistent
isochronous center if and only if
\[
a_{21}=b_{21}=0
\]
and one of the following conditions holds:
\begin{align*}
\text{I$_{01}$}\quad
& \begin{aligned}[t]
&a_{20}=a_{30}=b_{11}=b_{02}=b_{12}=b_{03}=0;
\end{aligned}
\\[1ex]
\text{I$_{02}$}\quad
& \begin{aligned}[t]
&b_{20}=b_{30}=a_{11}=a_{02}=a_{12}=a_{03}=0;
\end{aligned}
\\[1ex]
\text{I$_{03}$}\quad
& \begin{aligned}[t]
&a_{11}=a_{02}=a_{12}=a_{03}=b_{11}=b_{30}=0;
\end{aligned}
\\[1ex]
\text{I$_{04}$}\quad
& \begin{aligned}[t]
&b_{11}=b_{02}=b_{12}=b_{03}=a_{11}=a_{30}=0;
\end{aligned}
\\[1ex]
\text{I$_{05}$}\quad
& \begin{aligned}[t]
&a_{11}=a_{02}=a_{12}=a_{03}
=b_{11}=b_{02}=b_{12}=b_{03}=0;
\end{aligned}
\\[1ex]
\text{I$_{06}$}\quad
& \begin{aligned}[t]
&a_{11}=b_{11}=a_{02}=a_{03}=b_{03}
=a_{12}+b_{30}=a_{30}+b_{12}=0;
\end{aligned}
\\[1ex]
\text{I$_{07}$}\quad
& \begin{aligned}[t]
&a_{11}=b_{11}=b_{02}=b_{03}=a_{03}
=a_{30}+b_{12}=a_{12}+b_{30}=0;
\end{aligned}
\\[1ex]
\text{I$_{08}$}\quad
& \begin{aligned}[t]
&a_{20}=a_{02}=a_{12}=a_{03}
=b_{11}=b_{02}=b_{12}=b_{03}
=2a_{11}-b_{20}=0;
\end{aligned}
\\[1ex]
\text{I$_{09}$}\quad
& \begin{aligned}[t]
&b_{20}=b_{02}=b_{12}=b_{03}
=a_{11}=a_{02}=a_{12}=a_{03}
=2b_{11}-a_{20}=0;
\end{aligned}
\\[1ex]
\text{I$_{10}$}\quad
& \begin{aligned}[t]
&a_{20}=a_{02}=a_{12}=a_{03}
=b_{11}=b_{30}=b_{03}
=2a_{11}+b_{20}=2a_{30}+b_{12}=0;
\end{aligned}
\\[1ex]
\text{I$_{11}$}\quad
& \begin{aligned}[t]
&b_{20}=b_{02}=b_{12}=b_{03}
=a_{11}=a_{30}=a_{03}
=a_{20}+2b_{11}=a_{12}+2b_{30}=0;
\end{aligned}
\\[1ex]
\text{I$_{12}$}\quad
& \begin{aligned}[t]
&a_{20}=a_{02}=a_{30}=a_{12}=a_{03}
=b_{11}=b_{02}=b_{30}=b_{12}
=3a_{11}+b_{20}=0;
\end{aligned}
\\[1ex]
\text{I$_{13}$}\quad
& \begin{aligned}[t]
&b_{20}=b_{02}=b_{30}=b_{12}=b_{03}
=a_{11}=a_{02}=a_{30}=a_{12}
=a_{20}+3b_{11}=0;
\end{aligned}
\\[1ex]
\text{I$_{14}$}\quad
& \begin{aligned}[t]
&a_{20}=a_{02}=a_{30}=a_{03}
=b_{11}=b_{02}=b_{12}
=3a_{11}-b_{20}=3a_{12}+b_{30}=0;
\end{aligned}
\\[1ex]
\text{I$_{15}$}\quad
& \begin{aligned}[t]
&b_{20}=b_{02}=b_{30}=b_{03}
=a_{11}=a_{02}=a_{12}
=3b_{11}-a_{20}=a_{30}+3b_{12}=0;
\end{aligned}
\\[1ex]
\text{I$_{16}$}\quad
& \begin{aligned}[t]
&a_{20}=a_{11}=a_{02}=a_{30}=a_{03}
=b_{20}=b_{11}=b_{12}=b_{03}
=2a_{12}+b_{30}=0;
\end{aligned}
\\[1ex]
\text{I$_{17}$}\quad
& \begin{aligned}[t]
&b_{20}=b_{11}=b_{02}=b_{30}=b_{03}
=a_{20}=a_{11}=a_{12}=a_{03}
=a_{30}+2b_{12}=0;
\end{aligned}
\\[1ex]
\text{I$_{18}$}\quad
& \begin{aligned}[t]
&a_{02}=b_{02}=a_{03}=b_{03}
=a_{11}+b_{20}=a_{20}+b_{11}
=a_{12}+b_{30}=a_{30}+b_{12}
\\
& \quad \ =a_{30}b_{20}^{2}-a_{20}^{2}b_{30}=0;
\end{aligned}
\\[1ex]
\text{I$_{19}$}\quad
& \begin{aligned}[t]
&a_{20}b_{20}\neq0, \enspace
a_{30}=a_{12}=a_{03}=b_{30}=b_{12}=b_{03}
=7a_{11}+6b_{20}
=6a_{20}+7b_{11}
\\
&\quad \ =7a_{02}a_{20}-3b_{20}^{2}=-3a_{20}^{2}+7b_{02}b_{20}=0;
\end{aligned}
\\[1ex]
\text{I$_{20}$}\quad
& \begin{aligned}[t]
&a_{20}b_{20}\neq0, \enspace
a_{30}=a_{12}=a_{03}=b_{30}=b_{12}=b_{03}
=5a_{11}+2b_{20}
=2a_{20}+5b_{11}
\\
&\quad \ =5a_{02}a_{20}+3b_{20}^{2}
=3a_{20}^{2}+5b_{02}b_{20}=0;
\end{aligned}
\\[1ex]
\text{I$_{21}$}\quad
& \begin{aligned}[t]
&a_{30}b_{30}\neq0, \enspace
a_{20}=a_{11}=a_{02}=b_{20}=b_{11}=b_{02}
=7a_{12}+3b_{30}
=3a_{30}+7b_{12}
\\
&\quad=49a_{03}b_{03}-16a_{30}b_{30}
=a_{03}a_{30}^{2}-b_{03}b_{30}^{2}=0.
\end{aligned}
\end{align*}
In each case, the displayed relations are imposed together with
$a_{21}=b_{21}=0$, while all coefficients not explicitly restricted
are arbitrary complex numbers.
\end{thm}
\begin{proof}
We first prove the necessity. Suppose that the origin of
system~\eqref{sys_general} is a weakly persistent isochronous center.
Then, by the definition of the linearization quantities introduced in
the preliminary section,
\[
L_m^{+}(\varepsilon)\equiv0,
\qquad
L_m^{-}(\varepsilon)\equiv0,
\qquad m\geq1.
\]
Since each $L_m^{\pm}(\varepsilon)$ is a polynomial in
$\varepsilon$, all of its coefficients must vanish.

The first pair of linearization quantities is
\begin{align*}
L_1^{+}(\varepsilon)
&=
a_{21}\varepsilon
-
\left[
a_{11}(a_{20}+b_{11})
+\frac{2}{3}a_{02}b_{02}
\right]\varepsilon^2,\\
L_1^{-}(\varepsilon)
&=
-b_{21}\varepsilon
+
\left[
b_{11}(b_{20}+a_{11})
+\frac{2}{3}a_{02}b_{02}
\right]\varepsilon^2.
\end{align*}
Hence,
\[
a_{21}=b_{21}=0,\quad
3a_{11}(a_{20}+b_{11})+2a_{02}b_{02}=0,
\quad
3b_{11}(b_{20}+a_{11})+2a_{02}b_{02}=0.
\]
Subtracting the last two relations gives
\[
a_{11}a_{20}=b_{11}b_{20}.
\]
Thus, the first pair of linearization quantities yields the common
necessary relations
\[
a_{21}=b_{21}=0,
\qquad
a_{11}a_{20}=b_{11}b_{20}.
\]

As observed above, every weakly persistent isochronous center is
necessarily a weakly persistent center. Hence, by the weakly persistent
center classification obtained in the preceding section, it remains to
impose the linearization conditions on the center families listed there.

Substituting these center conditions into the coefficients of
\[
L_m^{\pm}(\varepsilon),
\qquad 1\leq m\leq7,
\]
and solving the resulting polynomial systems by exact algebraic
decomposition, together with the corresponding nonzero restrictions
and boundary cases, yields precisely the $21$ parameter conditions
$\mathrm{I}_{01}$--$\mathrm{I}_{21}$ listed in the theorem.

Therefore, the parameters must satisfy
\[
a_{21}=b_{21}=0
\]
together with one of
$\mathrm{I}_{01}$--$\mathrm{I}_{21}$.
This proves the necessity.

We next prove the sufficiency. Fix an arbitrary admissible value of the
common perturbation parameter $\varepsilon$. Throughout the proof, the
common condition
\[
a_{21}=b_{21}=0
\]
is understood. Whenever a condition below belongs to a weakly persistent
center family established in the preceding section, we use the
corresponding analytic first integral
\[
H(z,w)=zw+O(3).
\]

When condition $\mathrm{I}_{01}$ holds,
system~\eqref{sys_general} can be rewritten as
\[
\begin{cases}
\dot z
=
z+\varepsilon
\left(
a_{11}zw+a_{02}w^2+a_{12}zw^2+a_{03}w^3
\right),\\[1ex]
\dot w
=
-w-\varepsilon
\left(
b_{20}w^2+b_{30}w^3
\right).
\end{cases}
\]
The second equation is autonomous. Hence, by
Lemma~\ref{lem_one_dim_linearization}, there exists an analytic function
$V=w+O(2)$ satisfying $XV=-V$. By the weakly persistent center result
proved in the preceding section, the system admits an analytic first
integral $H=zw+O(3)$. Therefore, Lemma~\ref{lem_first_integral_linearization}
gives a time-preserving analytic linearization. Thus, the origin is a
weakly persistent isochronous center.

\medskip

When condition $\mathrm{I}_{02}$ holds, it is obtained from
condition $\mathrm{I}_{01}$ by exchanging the two coefficient sets.
Hence, by Lemma~\ref{lem_exchange}, the origin is also a weakly
persistent isochronous center.

\medskip

When condition $\mathrm{I}_{03}$ holds,
system~\eqref{sys_general} can be rewritten as
\[
\begin{cases}
\dot z
=
z+\varepsilon
\left(
a_{20}z^2+a_{30}z^3
\right),\\[1ex]
\dot w
=
-w-\varepsilon
\left(
b_{20}w^2+b_{02}z^2+b_{12}z^2w+b_{03}z^3
\right).
\end{cases}
\]
The first equation is autonomous. By
Lemma~\ref{lem_one_dim_linearization}, there exists an analytic function
$U=z+O(2)$ satisfying $XU=U$. Moreover, the preceding center result
provides an analytic first integral $H=zw+O(3)$. The symmetric part of
Lemma~\ref{lem_first_integral_linearization} therefore gives the
required analytic linearization. Hence, the origin is a weakly
persistent isochronous center.

\medskip

When condition $\mathrm{I}_{04}$ holds, it is obtained from
condition $\mathrm{I}_{03}$ by exchanging the two coefficient sets.
Hence, the conclusion follows from Lemma~\ref{lem_exchange}.

\medskip

When condition $\mathrm{I}_{05}$ holds,
system~\eqref{sys_general} can be rewritten as
\[
\begin{cases}
\dot z
=
z+\varepsilon
\left(
a_{20}z^2+a_{30}z^3
\right),\\[1ex]
\dot w
=
-w-\varepsilon
\left(
b_{20}w^2+b_{30}w^3
\right).
\end{cases}
\]
Both equations are autonomous. Applying
Lemma~\ref{lem_one_dim_linearization} to the two equations separately
gives analytic functions $U=z+O(2)$ and $V=w+O(2)$ satisfying
\[
XU=U,
\qquad
XV=-V.
\]
Therefore, $(U,V)$ is a time-preserving analytic linearization, and the
origin is a weakly persistent isochronous center.

\medskip

When condition $\mathrm{I}_{06}$ holds,
system~\eqref{sys_general} can be rewritten as
\[
\begin{cases}
\dot z
=
z+\varepsilon
\left(
a_{20}z^2+a_{30}z^3-b_{30}zw^2
\right),\\[1ex]
\dot w
=
-w-\varepsilon
\left(
b_{20}w^2+b_{02}z^2+b_{30}w^3-a_{30}z^2w
\right).
\end{cases}
\]
Set
\[
A=\varepsilon a_{20},
\qquad
B=\varepsilon b_{20},
\qquad
C=\varepsilon a_{30},
\qquad
D=\varepsilon b_{30},
\qquad
E=\varepsilon b_{02}.
\]
Then the system reduces to
\begin{equation}\label{eq:I06_system}
\begin{cases}
\dot z=z(1+Az+Cz^2-Dw^2),\\
\dot w=-w-Bw^2-Ez^2-Dw^3+Cz^2w.
\end{cases}
\end{equation}
Let $(p,q)$ satisfy
\begin{equation}\label{eq:I06_pq}
q^2-Bq+D=0,
\qquad
p^2-Ap+C+qE=0,
\end{equation}
and define $L_{p,q}=1+pz+qw$. A direct calculation using
\eqref{eq:I06_system} and \eqref{eq:I06_pq} gives
\begin{equation}\label{eq:I06_darboux}
XL_{p,q}
=
\left(
pz-qw+Cz^2-Dw^2
\right)L_{p,q}.
\end{equation}

For generic parameter values, the corresponding roots are distinct and
three linearly independent cofactors can be chosen. Hence, there exist
constants $\gamma_{p,q}$ satisfying
\[
\sum\gamma_{p,q}=1,
\qquad
\sum\gamma_{p,q}p=A,
\qquad
\sum\gamma_{p,q}q=0.
\]
It follows from \eqref{eq:I06_darboux} that
\[
U
=
z\prod L_{p,q}^{-\gamma_{p,q}}
\]
is analytic near the origin and satisfies $XU=U$. The preceding center
result gives an analytic first integral $H=zw+O(3)$, so
Lemma~\ref{lem_first_integral_linearization} yields an analytic
linearization on this generic parameter subset.

The parameter family defined by $\mathrm{I}_{06}$ is affine and hence
irreducible. Therefore, the remaining degenerate parameter values,
including the cases in which some roots in \eqref{eq:I06_pq} coincide,
follow from Lemma~\ref{lem_linearization_extension}. Thus, the origin is
a weakly persistent isochronous center.

\medskip

When condition $\mathrm{I}_{07}$ holds, it is obtained from
condition $\mathrm{I}_{06}$ by exchanging the two coefficient sets.
Hence, by Lemma~\ref{lem_exchange}, the origin is a weakly persistent
isochronous center.

\medskip

When condition $\mathrm{I}_{08}$ holds,
system~\eqref{sys_general} can be rewritten as
\[
\begin{cases}
\dot z
=
z+\varepsilon
\left(
a_{11}zw+a_{30}z^3
\right),\\[1ex]
\dot w
=
-w-\varepsilon
\left(
2a_{11}w^2+b_{30}w^3
\right).
\end{cases}
\]
The second equation is autonomous. Hence,
Lemma~\ref{lem_one_dim_linearization} gives an analytic function
$V=w+O(2)$ satisfying $XV=-V$. Since the preceding center result
provides an analytic first integral $H=zw+O(3)$,
Lemma~\ref{lem_first_integral_linearization} gives the desired analytic
linearization. Therefore, the origin is a weakly persistent isochronous
center.

\medskip

When condition $\mathrm{I}_{09}$ holds, it is obtained from
condition $\mathrm{I}_{08}$ by exchanging the two coefficient sets.
Hence, the conclusion follows from Lemma~\ref{lem_exchange}.

\medskip

When condition $\mathrm{I}_{10}$ holds, set
\[
A=\varepsilon a_{11},
\qquad
B=\varepsilon a_{30},
\qquad
C=\varepsilon b_{02}.
\]
Then system~\eqref{sys_general} can be rewritten as
\begin{equation*}
\begin{cases}
\dot z=z(1+Aw+Bz^2),\\
\dot w=-w+2Aw^2+2Bz^2w-Cz^2.
\end{cases}
\end{equation*}
Define
\begin{equation*}
L=1-2Aw+(B-AC)z^2,
\qquad
F=w+\frac{C}{3}z^2.
\end{equation*}
A direct calculation gives
\[
XL=2(Aw+Bz^2)L,
\qquad
XF=(-1+2Aw+2Bz^2)F.
\]
Thus, with
\[
U=zL^{-1/2},
\qquad
V=FL^{-1},
\]
we obtain
\[
XU=U,
\qquad
XV=-V.
\]
Since $L(0,0)=1$, the required analytic branches are well defined near
the origin. Hence, $(U,V)$ gives a time-preserving analytic
linearization, and the origin is a weakly persistent isochronous
center.

\medskip

When condition $\mathrm{I}_{11}$ holds, it is obtained from
condition $\mathrm{I}_{10}$ by exchanging the two coefficient sets.
Hence, its sufficiency follows from Lemma~\ref{lem_exchange}.

\medskip

When condition $\mathrm{I}_{12}$ holds, set
\[
A=\varepsilon a_{11},
\qquad
C=\varepsilon b_{03}.
\]
Then system~\eqref{sys_general} can be rewritten as
\begin{equation*}
\begin{cases}
\dot z=z(1+Aw),\\
\dot w=-w+3Aw^2-Cz^3.
\end{cases}
\end{equation*}
Define
\begin{equation*}
L=1-3Aw-ACz^3,
\qquad
F=w+\frac{C}{4}z^3.
\end{equation*}
A direct calculation yields
\[
XL=3AwL,
\qquad
XF=(-1+3Aw)F.
\]
Therefore,
\[
U=zL^{-1/3},
\qquad
V=FL^{-1}
\]
satisfy
\[
XU=U,
\qquad
XV=-V.
\]
Since $L(0,0)=1$, both functions are analytic near the origin.
Consequently, the origin is a weakly persistent isochronous center.

\medskip

When condition $\mathrm{I}_{13}$ holds, it is obtained from
condition $\mathrm{I}_{12}$ by exchanging the two coefficient sets.
Hence, its sufficiency follows from Lemma~\ref{lem_exchange}.

\medskip

When condition $\mathrm{I}_{14}$ holds, set
\[
A=\varepsilon a_{11},
\qquad
B=\varepsilon a_{12},
\qquad
C=\varepsilon b_{03}.
\]
Then system~\eqref{sys_general} can be rewritten as
\begin{equation}\label{eq:I14_system}
\begin{cases}
\dot z=z(1+Aw+Bw^2),\\
\dot w=-w-3Aw^2+3Bw^3-Cz^3.
\end{cases}
\end{equation}
Put
\[
d(w)=1+3Aw-3Bw^2
\]
and define
\begin{equation*}
f_0(w)
=
w\exp\left(
\int_0^w
\frac{d(s)^{-1}-1}{s}\,ds
\right).
\end{equation*}
We seek an analytic function of the form
\[
V(z,w)
=
\sum_{m=0}^{\infty}f_m(w)z^{3m}
\]
satisfying $XV=-V$. Substituting this series into
\eqref{eq:I14_system} gives
\begin{equation}\label{eq:I14_recursion}
-wd\,f_m'
+
\left[
3m(1+Aw+Bw^2)+1
\right]f_m
=
Cf_{m-1}',
\qquad
m\geq1.
\end{equation}
Since
\[
3(1+Aw+Bw^2)=3d-2wd',
\]
the equations in \eqref{eq:I14_recursion} can be solved recursively.
Set $P_1=C/4$. For $m\geq2$, write
\[
N_m
=
\left[
1-d-(2m-2)wd'
\right]P_{m-1}
+
wdP_{m-1}'
=
\sum_{k=1}^{m}n_{m,k}w^k
\]
and define
\[
P_m
=
C\sum_{k=1}^{m}
\frac{n_{m,k}}{3m+2-k}w^{k-1},
\qquad
f_m
=
\frac{f_0}{w}d^{-2m}P_m.
\]
Since
\[
3m+2-k\geq2m+2,
\qquad
1\leq k\leq m,
\]
no denominator vanishes. Moreover, on a sufficiently small disk the
recurrence gives at most geometric growth of the coefficients.
Therefore, the series defining $V$ converges normally and gives an
analytic function
\[
V=w+O(2),
\qquad
XV=-V.
\]
The preceding center result provides an analytic first integral
$H=zw+O(3)$. Hence, by
Lemma~\ref{lem_first_integral_linearization}, the system is analytically
linearizable. Therefore, the origin is a weakly persistent isochronous
center.

\medskip

When condition $\mathrm{I}_{15}$ holds, it is obtained from
condition $\mathrm{I}_{14}$ by exchanging the two coefficient sets.
Hence, its sufficiency follows from Lemma~\ref{lem_exchange}.

\medskip

When condition $\mathrm{I}_{16}$ holds, set
\[
A=\varepsilon a_{12},
\qquad
C=\varepsilon b_{02}.
\]
Then system~\eqref{sys_general} can be rewritten as
\begin{equation*}
\begin{cases}
\dot z=z(1+Aw^2),\\
\dot w=-w+2Aw^3-Cz^2.
\end{cases}
\end{equation*}
Define
\begin{equation*}
L
=
1-2Aw^2-4ACz^2w-AC^2z^4,
\qquad
F=w+\frac{C}{3}z^2.
\end{equation*}
A direct calculation gives
\[
XL=4Aw^2L,
\qquad
XF=(-1+2Aw^2)F.
\]
Thus,
\[
U=zL^{-1/4},
\qquad
V=FL^{-1/2}
\]
satisfy
\[
XU=U,
\qquad
XV=-V.
\]
Since $L(0,0)=1$, the required analytic branches exist near the origin.
Hence, the origin is a weakly persistent isochronous center.

\medskip

When condition $\mathrm{I}_{17}$ holds, it is obtained from
condition $\mathrm{I}_{16}$ by exchanging the two coefficient sets.
Hence, its sufficiency follows from Lemma~\ref{lem_exchange}.

\medskip

When condition $\mathrm{I}_{18}$ holds,
system~\eqref{sys_general} can be rewritten as
\begin{equation*}
\begin{cases}
\displaystyle
\dot z
=
z\left[
1+\varepsilon
\left(
a_{20}z-b_{20}w+a_{30}z^2-b_{30}w^2
\right)
\right],\\[1ex]
\displaystyle
\dot w
=
-w\left[
1-\varepsilon
\left(
a_{20}z-b_{20}w+a_{30}z^2-b_{30}w^2
\right)
\right].
\end{cases}
\end{equation*}
Therefore,
\begin{equation}\label{eq:I18_log_difference}
\frac{\dot z}{z}
-
\frac{\dot w}{w}
=
2.
\end{equation}
By the weakly persistent center result established in the preceding
section, the system admits an analytic first integral of the form
\[
H=zw\,h_0(z,w),
\qquad
h_0(0,0)=1.
\]
Taking the analytic branch of $\sqrt{h_0}$ satisfying
$\sqrt{h_0(0,0)}=1$, define
\[
U=z\sqrt{h_0},
\qquad
V=w\sqrt{h_0}.
\]
Using $XH=0$ and \eqref{eq:I18_log_difference}, we obtain
\[
\frac{XU}{U}
=
\frac12
\left(
\frac{\dot z}{z}
-
\frac{\dot w}{w}
\right)
=1,
\]
and similarly $XV/V=-1$. Hence,
\[
XU=U,
\qquad
XV=-V.
\]
Thus, $(U,V)$ is a time-preserving analytic linearization and the
origin is a weakly persistent isochronous center.

\medskip

When condition $\mathrm{I}_{19}$ holds, since
$a_{20}b_{20}\neq0$, set
\[
h=\frac{b_{20}}{a_{20}},
\qquad
\zeta=z,
\qquad
\eta=hw.
\]
Then system~\eqref{sys_general} can be rewritten as
\begin{equation}\label{eq:I19_scaled}
\begin{cases}
\displaystyle
\dot\zeta
=
\zeta+\varepsilon a_{20}
\left(
\zeta^2-\frac67\zeta\eta+\frac37\eta^2
\right),\\[1ex]
\displaystyle
\dot\eta
=
-\eta-\varepsilon a_{20}
\left(
\eta^2-\frac67\zeta\eta+\frac37\zeta^2
\right).
\end{cases}
\end{equation}
Set
\[
x=\zeta+\eta,
\qquad
y=\zeta-\eta,
\qquad
a=\frac47\varepsilon a_{20}.
\]
Then \eqref{eq:I19_scaled} reduces to
\begin{equation}\label{eq:I19_xy}
\dot x=y(1+ax),
\qquad
\dot y=x+\frac{a}{2}x^2+2ay^2.
\end{equation}
Define
\begin{equation*}
u
=
\frac{x(1+ax/2)}{(1+ax)^2},
\qquad
v
=
\frac{y}{(1+ax)^2}.
\end{equation*}
A direct calculation using \eqref{eq:I19_xy} gives
\begin{equation}\label{eq:uv_linear}
\dot u=v,
\qquad
\dot v=u.
\end{equation}
Hence, setting
\begin{equation}\label{eq:uv_to_eigen}
\widehat Z=\frac{u+v}{2},
\qquad
\widehat W=\frac{u-v}{2},
\end{equation}
we obtain
\[
\dot{\widehat Z}=\widehat Z,
\qquad
\dot{\widehat W}=-\widehat W.
\]
Since $\widehat Z=\zeta+O(2)$ and
$\widehat W=\eta+O(2)$, the transformation
\[
Z=\widehat Z,
\qquad
W=h^{-1}\widehat W
\]
is near identity in the original variables. Therefore, the origin is a
weakly persistent isochronous center.

When condition $\mathrm{I}_{20}$ holds, set
\[
h=\frac{b_{20}}{a_{20}},
\qquad
\zeta=z,
\qquad
\eta=hw.
\]
Then system~\eqref{sys_general} can be rewritten as
\begin{equation*}
\begin{cases}
\displaystyle
\dot\zeta
=
\zeta+\varepsilon a_{20}
\left(
\zeta^2-\frac25\zeta\eta-\frac35\eta^2
\right),\\[1ex]
\displaystyle
\dot\eta
=
-\eta-\varepsilon a_{20}
\left(
\eta^2-\frac25\zeta\eta-\frac35\zeta^2
\right).
\end{cases}
\end{equation*}
Set
\[
x=\zeta+\eta,
\qquad
y=\zeta-\eta,
\qquad
a=\frac85\varepsilon a_{20}.
\]
Then
\begin{equation}\label{eq:I20_xy}
\dot x=y(1+ax),
\qquad
\dot y=x+\frac{a}{4}y^2.
\end{equation}
If $a=0$, the system is already linear. Suppose $a\neq0$, and take
the analytic branch
\[
r=(1+ax)^{1/4},
\qquad
r(0)=1.
\]
Define
\begin{equation*}
u
=\frac{4}{a}(1-r^{-1}),
\qquad
v=yr^{-1},
\qquad
s
=
\frac{2}{a}(r-r^{-1}).
\end{equation*}
A direct calculation using \eqref{eq:I20_xy} gives
\begin{equation*}
\dot s=\frac{v}{1+\chi(s)},
\qquad
\dot v=\frac{s}{1+\chi(s)},
\end{equation*}
where
\[
u
=
s+\frac{4}{a}
\left(
1-\sqrt{1+\frac{a^2s^2}{16}}
\right),
\qquad
\chi(s)=\frac{du}{ds}-1.
\]
The function $\chi$ is analytic and odd. Set
\[
\xi=\frac{s+v}{2},
\qquad
\eta_1=\frac{s-v}{2},
\qquad
\mathcal L
=
\xi\frac{\partial}{\partial\xi}
-
\eta_1\frac{\partial}{\partial\eta_1}.
\]
Writing
\[
\chi(\xi+\eta_1)
=
\sum_{i,j}c_{ij}\xi^i\eta_1^j,
\]
the oddness of $\chi$ implies that $i+j$ is odd whenever
$c_{ij}\neq0$, and hence $i-j\neq0$. Thus,
\[
\varphi
=
\sum_{i,j}
\frac{c_{ij}}{i-j}\xi^i\eta_1^j
\]
is analytic and satisfies
\[
\mathcal L\varphi=\chi.
\]
Consequently,
\[
\widehat Z=\xi e^\varphi,
\qquad
\widehat W=\eta_1e^{-\varphi}
\]
satisfy
\[
X\widehat Z=\widehat Z,
\qquad
X\widehat W=-\widehat W.
\]
Since $\widehat Z=\zeta+O(2)$ and
$\widehat W=\eta+O(2)$, rescaling by
\[
Z=\widehat Z,
\qquad
W=h^{-1}\widehat W
\]
gives a near-identity time-preserving analytic linearization in the
original variables. Hence, the origin is a weakly persistent
isochronous center.

\medskip

When condition $\mathrm{I}_{21}$ holds,
system~\eqref{sys_general} can be rewritten as
\[
\begin{cases}
\dot z
=
z+\varepsilon
\left(
a_{30}z^3+a_{12}zw^2+a_{03}w^3
\right),\\[1ex]
\dot w
=
-w-\varepsilon
\left(
b_{30}w^3+b_{12}z^2w+b_{03}z^3
\right).
\end{cases}
\]
Since $a_{30}b_{30}\neq0$, the relations in
$\mathrm{I}_{21}$ imply $a_{03}b_{03}\neq0$. Set \(h
=
-\frac{4a_{30}}{7b_{03}}.\)
Then
\[
h^2=\frac{b_{30}}{a_{30}},
\qquad
h^4b_{03}=a_{03}.
\]
Let \(\zeta=z, \enspace
\eta=hw, \enspace d=hb_{03}\). 
The cubic coefficients then reduce to
\[
a_{30}=-\frac{7}{4}d,
\qquad
a_{12}=\frac{3}{4}h^2d,
\qquad
a_{03}=h^3d.
\]
Set
\[
x=\zeta+\eta,
\qquad
y=\zeta-\eta,
\qquad
k=\frac34\varepsilon d.
\]
The system becomes
\begin{equation}\label{eq:I21_xy}
\dot y=x(1-ky^2),
\qquad
\dot x
=
y\left(
1-\frac23ky^2-3kx^2
\right).
\end{equation}
Take
\[
\mu=(1-ky^2)^{-3/2}
\]
and define
\begin{equation*}
u
=
\mu y
\left(
1-\frac23ky^2
\right),
\qquad
v=\mu x.
\end{equation*}
A direct calculation using \eqref{eq:I21_xy} gives
\eqref{eq:uv_linear}. Hence, applying
\eqref{eq:uv_to_eigen}, we obtain
\[
\dot{\widehat Z}=\widehat Z,
\qquad
\dot{\widehat W}=-\widehat W.
\]
Since $\widehat Z=\zeta+O(2)$ and
$\widehat W=-\eta+O(2)$, the transformation
\[
Z=\widehat Z,
\qquad
W=-h^{-1}\widehat W
\]
is near identity in the original variables and does not change the time
variable. Therefore, the origin is a weakly persistent isochronous
center.

The above arguments prove the sufficiency of
$\mathrm{I}_{01}$--$\mathrm{I}_{21}$. Together with the necessity
proved above, this completes the proof.
\end{proof}

\begin{remark}
Similarly, the completeness of the necessary parameter decomposition is also established by exact computer-algebraic certificates, see Supplementary Material.  By an exhaustive system of saturated algebraic charts, the obstruction ideal through order seven is decomposed, with every complementary zero locus treated separately, yielding $V(L_7)=V(L_{\infty})$. Thus no additional parameter components can arise from higher-order  linearization quantities. The sufficiency of each resulting family has been proved analytically above.
\end{remark}

\section{Weakly persistent centers in nilpotent systems}\label{sec_nilpotent}

Having established the weakly persistent center conditions for the systems considered in the preceding sections, we now turn to the nilpotent case. Namely, we consider the following system.
\begin{equation}\label{sys_nilp}
\begin{split}
\frac{dx}{dt}&=y+ \varepsilon (a_{20} x^2+a_{11} x y+a_{02} y^2+a_{30} x^3+a_{21} x^2y+a_{12} x y^2+a_{03} y^3),\\
\frac{dy}{dt}&=-2x^3+ \varepsilon (2a_{20} xy+b_{02} y^2+b_{21} x^2y+b_{12} x y^2+b_{03} y^3).
\end{split}
\end{equation}
A complete characterization of the weakly persistent center conditions for system \eqref{sys_nilp} is given in the following theorem.

\begin{thm}\label{thm_6}
The origin of system \eqref{sys_nilp} is a weakly persistent center if and only if one of the following conditions holds:
\begin{align*}
\text{$C_{1}$} & \quad 
a_{11}=a_{30}=b_{02}=b_{21}=a_{12}=b_{03}=0; \\
\text{$C_{2}$} & \quad 
b_{11}=a_{20}=b_{21}=a_{02}=a_{30}=b_{03}=a_{12}=0; \\
\text{$C_{3}$} & \quad 
b_{11}=a_{20}=b_{02}+\frac{a_{11}}{2}=b_{21}+3a_{30}=b_{03}+\frac{a_{12}}{3}=b_{12}+a_{21}=0. 
\end{align*}
\end{thm}

\begin{proof}
We first derive the necessary center conditions. The auxiliary coefficient $w_5$ arising in the recursive computation of the quasi-Lyapunov constants is given by
\begin{equation}\label{w5}
    \begin{aligned}
    w_5 = & -\frac{4a_{20}s \varepsilon}{(1+s)^2}  \left(1+2s+s^2+2a_{20}^2s^2\varepsilon^2 \right).
\end{aligned}
\end{equation}
Since the factor in parentheses has the nonzero constant term
$8(1+s)^2$, the condition $w_5=0$ leads to two cases:
$$
\text{(i) } 2a_{20}\neq 0 \ (s=0), \qquad \text{(ii) } a_{20}=0.
$$
We consider these two cases separately.

Under $a_{20}\neq 0$, the first two  quasi-Lyapunov constants are
\begin{align*}
 v_{1}= &-\frac{1}{15} \varepsilon \left[-15 a_{30} - 5 b_{21} +(13 a_{11} a_{20} +6 a_{20} b_{02}) \varepsilon +(a_{20}^2b_{21}-3a_{20}^2a_{30}) \varepsilon^2 +a_{11}a_{20}^3 \varepsilon^3 \right],\\
 v_{2}=&\frac{1}{35} \varepsilon  \left[14 a_{12} + 42 b_{03} + (43 a_{12} a_{20}^2 + 61 a_{20}^2 b_{03} ) \varepsilon^2 + (9 a_{12} a_{20}^4 +  9 a_{20}^4 b_{03}) \varepsilon^4 \right].
\end{align*}
For the origin to be a weakly persistent center, these quasi-Lyapunov constants must vanish identically with respect to $\varepsilon$. 
With the aid of Mathematica, we equate to zero all the coefficients of $v_{j}$ in powers of $\varepsilon$, for $j=1,2$, and solve the resulting polynomial system,
 we get the unique solution $a_{11}=a_{30}=b_{02}=b_{21}=a_{12}=b_{03}=0$, namely, the condition $C_{1}$ . 

Under $a_{20}=0$ , the first two quasi-Lyapunov constants are
\begin{align*}
 v'_{1}= &\frac{ \varepsilon }{3}\left( 3 a_{30}+ b_{21} \right),\\
 v'_{2}=&\frac{ \varepsilon }{5}\left[2 a_{12} + 6 b_{03} -2(a_{02} a_{11}+  a_{21} a_{30} +2 a_{02} b_{02} + a_{30} b_{12}) \varepsilon \right. \\
&\left. + (a_{11}^2 a_{30}+ a_{11} a_{30} b_{02} - 2 a_{30} b_{02}^2) \varepsilon^2\right].
\end{align*}
Similarly, we equate to zero all the coefficients of $v'_{j}$ with respect to $\varepsilon$, for $j=1,2$, and solve the resulting polynomial system using Mathematica. 
After removing inadmissible, duplicate, and properly contained parameter families, we obtain the following conditions
\begin{align*}
1. &a_{02}=a_{30}=a_{12}+3b_{03}=0;\\
2. &a_{30}=a_{11}+2b_{02}=a_{12}+3b_{03}=0; \\
3. &b_{12}+a_{21}=a_{11}+2b_{02}=a_{12}+3b_{03}=0;\\
4. &b_{02}-a_{11}=a_{12}+3b_{03}=a_{30}b_{12}+3a_{02}a_{11}+a_{21}a_{30}=0, a_{30}\neq 0.
\end{align*}
Then, we substitute the above conditions one by one to compute the third, fourth and fifth  quasi-Lyapunov constants again, and  equate to zero all the coefficients of $v'_{j}$ with respect to $\varepsilon$, for $j=3,4,5$,
we get the conditions  $C_{1},C_{2}$ and  $C_{3}$. Combining the results of the two cases, we obtain exactly the three parameter families $C_{1}-C_{3}$ listed in Theorem \ref{thm_6}. 
This completes the proof of necessity. We next establish the sufficiency by considering the three parameter families separately.

When condition $C_{1}$ holds, system  \eqref{sys_nilp} can be rewritten as
\begin{equation*}
     \begin{cases}
\dot x=y+\varepsilon(a_{20} x^{2}  + a_{02} y^{2}+ a_{21} x^{2} y + a_{03} y^{3} ),\\
\dot y=-2x^3+\varepsilon(2a_{20} x y + b_{12} x y^{2}).
\end{cases}
\end{equation*}
which is symmetric with respect to the $y$-axis, hence the origin is a center.

When condition $C_{2}$ holds,
system  \eqref{sys_nilp} becomes
\begin{equation*}
     \begin{cases}
\dot x=y+\varepsilon(a_{11} x y + a_{21} x^{2} y + a_{03} y^{3} ),\\
\dot y=-2x^3+\varepsilon(b_{02} y^{2} + b_{12} x y^{2}),
\end{cases}
\end{equation*}
which is symmetric with respect to the $x-$axis, hence the origin is a center.

When condition $C_{3}$ holds, system  \eqref{sys_nilp} can be changed to
\begin{equation*}
     \begin{cases}
\dot x=y+\varepsilon( a_{11} x y + a_{02} y^{2} + a_{30} x^{3} + a_{21} x^{2} y + a_{12} x y^{2} + a_{03} y^{3}),\\
\dot y=-2x^3+\varepsilon( - \frac{a_{11} }{2}y^{2}- 3 a_{30} x^{2} y  - a_{21} x y^{2} - \frac{a_{12} }{3}y^{3}),
\end{cases}
\end{equation*}
which admits a first integral
$$H_{5} = \frac{x^4}{2} + \frac{y^2}{2} + \varepsilon \left( a_{30} x^3 y + \frac{1}{2} a_{11} x y^2 + \frac{1}{2} a_{21} x^2 y^2 + \frac{1}{3} a_{02} y^3 + \frac{1}{3} a_{12} x y^3 + \frac{1}{4} a_{03} y^4 \right).$$
We have thus completed the proof of Theorem \ref{thm_6}.
\end{proof}

\section{Conclusion}\label{sec_conclusion}

In this paper, we have investigated the weakly persistent center problem for  planar cubic differential systems, including
systems with a linear  complex center and systems with a nilpotent center. By
computing the corresponding singular point quantities or
quasi-Lyapunov constants, we have derived the necessary algebraic
relations among the perturbation parameters, and the sufficiency of
the resulting conditions has been established by means of symmetry,
explicit first integrals, inverse integrating factors, suitable
analytic transformations, regular analytic initial-value problems, and
convergent power-series constructions. For the general complex cubic
system, we have further studied the weakly persistent isochronous
center problem. By combining the weakly persistent center conditions
with the vanishing of the linearization quantities and constructing
time-preserving analytic linearizations, we have obtained the
corresponding necessary and sufficient conditions for the origin to be
a weakly persistent isochronous center.

Although the results obtained here provide complete characterizations
for weakly persistent center and  isochronous center of cubic systems under consideration,
 the center problem for general cubic differential systems is still far from being completely
resolved. In particular, as the number of parameters increases, the
algebraic conditions arising from the vanishing of the singular point
quantities become increasingly complicated,
and the verification of their sufficiency also requires more delicate analytic constructions.
It would therefore be of interest to develop more effective algebraic
and analytic methods for the general cubic center problem.

\section*{Acknowledgement}
This research is partially supported by National Natural Science
Foundation of China, No. 12471166 (F. Li), the Natural Science Foundation of Shandong Province, China, No. ZR2024MA037 (F. Li) 
and National Natural Science Foundation of China, No. 12671199 (C. Cao),

\section*{Declaration of generative AI and AI-assisted technologies in the manuscript preparation process}
During the preparation of this work, the author(s) used ChatGPT for computing singular point
quantities. The author(s) reviewed and edited the output as needed and take full responsibility for the content of the published article. 

\bibliography{sin}
\end{document}


\maketitle

\section{System, definitions, and scope}
We consider
\begin{equation}\label{sys}
\begin{aligned}
\dot z={}&z+\eps\bigl(a_{20}z^2+a_{11}zw+a_{02}w^2+a_{30}z^3+a_{21}z^2w+a_{12}zw^2+a_{03}w^3\bigr),\\
\dot w={}&-w-\eps\bigl(b_{20}w^2+b_{11}zw+b_{02}z^2+b_{30}w^3+b_{21}zw^2+b_{12}z^2w+b_{03}z^3\bigr).
\end{aligned}
\end{equation}
All $a_{ij},b_{ij}$ are independent elements of $\C$; no conjugacy relation is imposed.  A weakly persistent center has, for every $\eps\in\C$, a local analytic first integral $H_\eps=zw+\ord(3)$.  A weakly persistent isochronous center is analytically conjugate, with time unchanged, to $\dot U=U$, $\dot V=-V$ for every $\eps$.  This paper proves the necessity and completeness of the parameter decompositions.  The analytic all-order constructions proving sufficiency are recorded in the companion classification paper.

The computations use exact rational arithmetic.  Whenever a chart declares $\delta\ne0$, its ideal is replaced by $J:\delta^\infty$.  Every factor put into $\delta$ is also treated on a separate zero chart.  Thus no conclusion below is obtained by dividing by a parameter before its zero locus has been analyzed.

\section{Weakly persistent centers: necessity}
\begin{theorem}[Finite necessity and exhaustive center decomposition]\label{thm:center}
If the origin of \eqref{sys} is a weakly persistent center, then
\begin{equation}\label{common:center}
a_{21}=b_{21},\qquad a_{11}a_{20}=b_{11}b_{20},
\end{equation}
and at least one of the 39 original-coordinate groups F01--F28 listed in Appendix~\ref{app:center-list} holds.  Conversely, the exact obstruction variety through order seven has no other point:
\[
V(I_7)=\bigcup_{\mathrm{F}\nu}V(K_{\mathrm{F}\nu}),\qquad
\sqrt{I_7}=\bigcap_{\mathrm{F}\nu}\sqrt{K_{\mathrm{F}\nu}}
\quad\text{over }\C,
\]
with saturation by each group's stated nonzero factors.  Hence the displayed alternatives are a complete necessary list.
\end{theorem}

\subsection{Homological recursion and saddle quantities}
Write $X_\eps=X_0+\eps X_2+\eps X_3$, where
$X_0=z\partial_z-w\partial_w$ and $X_d$ has homogeneous coefficients of
degree $d$.  Starting with $F_1=0$ and $F_2=zw$, define at total degree $n$
\begin{equation}\label{rec:center}
T_n=\eps(X_2F_{n-1}+X_3F_{n-2}),\qquad
[z^iw^j]F_n=-\frac{[z^iw^j]T_n}{i-j}\quad(i\ne j).
\end{equation}
The resonant coefficient $[z^mw^m]F_{2m}$ is set to zero.  The term that
cannot be removed is
\begin{equation}\label{saddle}
g_k(\eps)=[z^{k+1}w^{k+1}]T_{2k+2}
=\sum_{j=k}^{2k}G_{kj}\eps^j,
\qquad I_N=\id{G_{kj}:1\le k\le N}.
\end{equation}
A center for every $\eps$ requires every $G_{kj}$ to vanish.  At the first
level,
\[
G_{11}=a_{21}-b_{21},\qquad
G_{12}=b_{11}b_{20}-a_{11}a_{20},
\]
which gives \eqref{common:center}.  The next lowest coefficient is
\[
G_{22}=b_{30}b_{12}-a_{30}a_{12}.
\]
The recursion gives 27 coefficient polynomials through $g_6$:
$2+3+4+5+6+7=27$.  They were computed over $\mathbb Q$ in the full
fourteen-parameter ring.  The complete expansions are retained in the
electronic calculation bundle; below we display the reduced combinations
that determine each radical split.

On a chart with current ideal $J$ and a product $\delta$ declared nonzero,
the relevant ideal is $J:\delta^\infty$.  An equality
\[
\sqrt{J:\delta^\infty}
=\bigcap_q\sqrt{Q_q:\delta^\infty}
\]
is verified in both directions by exact normal-form calculations.  Whenever
a factor enters $\delta$, its zero set is treated as a separate child chart.
Thus the use of normalized variables below does not remove a degenerate
boundary.

\subsection{Successive center branches and their complete coverage}\label{sec:center-branches-prose}
We now follow the order in which the saddle quantities are actually solved.
This is more informative than listing the final ideals without their parent
branches.  On a chart with current ideal $J$ and declared nonzero product
$\delta$, write
\[
 \overline G_{kj}=\operatorname{NF}(G_{kj};J),\qquad
 \mathcal R_7(J,\delta)=
 \sqrt{\id{\overline G_{kj}:1\le k\le7}:\delta^\infty}.
\]
Every denominator introduced below is a factor of $\delta$.  Its zero set is
handled before that denominator is used.

\subsubsection{The nondegenerate quadratic branch}
Assume first that $a_{20}b_{20}\ne0$.
Impose $G_{11}=G_{12}=0$ and use the normalized variables
\begin{equation}\label{eq:prose-normalization}
\begin{gathered}
t={a_{11}}/{b_{20}}={b_{11}}/{a_{20}},\quad
A={a_{02}a_{20}}/{b_{20}^{2}},\quad
B={b_{02}b_{20}}/{a_{20}^{2}},\\
C={a_{30}}/{a_{20}^{2}},\quad D={b_{30}}/{b_{20}^{2}},\quad
U={a_{12}}/{b_{20}^{2}},\quad V={b_{12}}/{a_{20}^{2}},\\
E={a_{03}a_{20}}/{b_{20}^{3}},\quad
F={b_{03}b_{20}}/{a_{20}^{3}},\quad
S={a_{21}}/{(a_{20}b_{20})}.
\end{gathered}
\end{equation}
The first three reduced quantities needed on this chart are
\begin{equation}\label{eq:nondeg-prose-entry}
 G_{22}=DV-CU,\qquad
 G_{24}=\frac13t(A-B)(t-2)(2t+1),\qquad
 \left.G_{36}\right|_{t=-1/2}
 =-\frac{25}{256}(A-B)(4AB-1).
\end{equation}
After saturation by $a_{20}b_{20}$, the complement of
\[
 t=0,\qquad t=2,\qquad A=B,
 \qquad\text{or}\qquad t=-\tfrac12,\ 4AB=1
\]
has unit obstruction ideal.  These four alternatives therefore cover the
nondegenerate chart.

First take $t=2$.  The remaining reduced quantities split into
$U=3D,V=3C$, which restores to F01, and the equal-coefficient boundary
$A=B,C=D,E=F,U=V$, which is contained in F02.  Hence no additional branch is
created at $t=2$.

Next take $A=B$.  The five quantities
$\overline G_{k,2k-1}$, $2\le k\le6$, are linear in
$(C-D,E-F,U-V)$.  Their coefficient matrix $M(A,t)$ satisfies
\[
 M(A,t)(-A-1,4A,3A+2t-1)^T=0.
\]
Away from $t=0,2$, the ideal of its minors has only the rank-drop points
\[
(A,t)=(1/2,-1/2),\ (-1/2,-1/2),\ (-2/3,-1/3),\
(2/7,1/7),\ (1/4,1/4).
\]
At generic rank two the solution is either the reversible family F02 or the
Lotka family F04.  The points $(1/2,-1/2)$, $(-2/3,-1/3)$ and $(2/7,1/7)$
reduce to F02.  The two remaining rank drops give F10 and F11.  Since the
entire minor ideal, rather than a selected minor, is decomposed, this step
does not miss an exceptional curve or an isolated rank drop.

On the exceptional part $t=-1/2$, $A\ne B$, equation
\eqref{eq:nondeg-prose-entry} forces $4AB=1$.  Reduction of the next
quantities gives
\[
 C=D=E=F=U=V=S=0,
\]
which is F09.  The saturated complement is empty.

It remains to treat $t=0$.  Here $G_{22}=0$ is $CU=DV$.  If
$(C,D)\ne(0,0)$, introduce $L$ by $U=LD$, $V=LC$.  This is done only on
that open chart.  Three useful reductions are
\begin{align}
G_{33}&=-\frac18(L-3)(3L+1)(EC^2-FD^2),\label{eq:t0-prose-1}\\
G_{23}&=\frac23\{(L+1)(AC-BD)-AF+BE\},\label{eq:t0-prose-2}\\
G_{35}+\frac{5AB+6A+123}{12}G_{23}
&=-\frac13(A-B)\{S-(L+1)AC-BE\}.\label{eq:t0-prose-3}
\end{align}
The factors in \eqref{eq:t0-prose-1} separate $L=3$, $L=-1/3$, and the
relation $EC^2=FD^2$.  Equation \eqref{eq:t0-prose-2} then separates
$L=-1$ from its complement, while \eqref{eq:t0-prose-3} separates $A=B$
from $A\ne B$.  Decomposing the remaining radicals according to $AB=0$,
$CD=0$, and $C=D$ gives F02, F03, F04, F07, F07b, and F12--F16.  More
precisely, $L=3$ contains the Hamiltonian boundary and the F12 continuation;
$L=-1/3$ contains F07/F07b and F12; $L=-1$ contains F02--F04 and F16; the
general equal-quadratic locus gives F02, F07/F07b, F08, F12 and F13; the
remaining $A\ne B$ loci give F12--F15.

If $C=D=0$, the parameter $L$ is not introduced, because $U$ and $V$ are
then independent.  Direct decomposition of $\mathcal R_7$ gives six
components.  After restoration they are F02, F03, F07, F07b, F16, F17 and
F18; one component and its exchanged form are counted separately.  This
direct treatment is essential: imposing $U=LD,V=LC$ at $C=D=0$ would
incorrectly remove the independent-$U,V$ components F17 and F18.

Thus the nondegenerate quadratic branch has been exhausted.  Every point on
it belongs to F01--F04, F07--F18, or to an overlap of these families.

\subsubsection{The one-sided quadratic branches}
Begin with $a_{20}=0$, $b_{20}\ne0$.
The equation $G_{12}=0$ gives $b_{11}=0$.  Scale $b_{20}$ to one and put
$t=a_{11}/b_{20}$.  The first decisive reductions are
\begin{equation}\label{eq:onesided-prose}
G_{24}=-\frac13Bt(t-2)(2t+1),\qquad
G_{36}=-\frac1{72}AB^2t(t-2)(73t+14).
\end{equation}
On $B=0$ they are supplemented by
\begin{equation}\label{eq:onesided-prose-next}
G_{23}=\frac13AF(t-2)+t\{(2t-1)C-V\},\qquad
G_{35}=\frac18Ft(t-2)(3t-1)(3t+1).
\end{equation}
If $t=2$, the reduced ideal gives the Hamiltonian boundary F01 and the
triangular boundary F06.  If $B\ne0$ and $t\ne0,2$, equations
\eqref{eq:onesided-prose} and the next nonzero reductions force $t=-1/2$;
the resulting component is F19.  If $B=0$ and $t\ne0,2$, equation
\eqref{eq:onesided-prose-next} yields $t=-1/3$ or $t=1/3$, producing F20
and F21, together with the Lotka and triangular boundaries F04 and F06.

For $t=0$, the radical is decomposed without division into $A=0,B=0$, and
$AB\ne0$.  The locus $B=0$ gives F03, F04, F06--F08 and F22.  The locus
$A=0,B\ne0$ gives F03, F07 and F23.  The locus $AB\ne0$ gives F03 and
F24--F26.  In the F26 calculation the original scale is eliminated with
$9D+7U=-12v$; no division by $7r-6$ is made, so the value $r=6/7$ remains
in the family.  This completes the chart $a_{20}=0,b_{20}\ne0$.

The chart $a_{20}\ne0,b_{20}=0$ is its image under the invertible exchange
of the two equations.  It gives F06b, F07b and F19b--F26b, including all
their zero boundaries.  No new calculation or unexamined denominator is
hidden in the suffix b.

\subsubsection{The doubly degenerate quadratic branch}
Assume now that $a_{20}=b_{20}=0$.
Put $u=a_{11}$, $v=b_{11}$, $A=a_{02}$ and $B=b_{02}$.  The first useful
quantities are
\begin{equation}\label{eq:double-prose}
G_{24}=\frac23(Av^3-Bu^3),\qquad
\left.G_{35}\right|_{u=0,v\ne0}=-\frac98E,
\qquad
\left.G_{23}\right|_{u=0,v\ne0}=-\frac13(BE+6D).
\end{equation}
We first separate $uv\ne0$, exactly one of $u,v$ equal to zero, and
$u=v=0$.  In the first two cases \eqref{eq:double-prose} and its exchanged
form give F02, F05, F06 and F06b.  On $u=v=0$ we next separate $AB\ne0$,
exactly one of $A,B$ equal to zero, and $A=B=0$.  The first two alternatives
give F01--F03, F06/F06b, and F27/F27b.

On the last alternative all quadratic terms vanish.  If $CD\ne0$, scale the
self-cubic coefficients and put $U=V=L$.  The radical is
\[
 \id{L-3}\ \cap\ \id{EC^2-FD^2}\ \cap\
 \id{3L+1,S,9EF-4CD}.
\]
These three components restore respectively to F01, F02, and F28.  If one
of $C,D$ is zero, the corresponding radical is
\[
 \id{U-3}\ \cap\ \id{F}\ \cap\ \id{3U+1,E,S}
\]
or its exchanged form, giving F01, F02, F06/F06b, and F21/F21b.  If
$C=D=0$, the remaining equation is $U^2F-EV^2=0$, already covered by the
same families and their boundaries.

A finite-order residual candidate survives the first six saddle quantities
on one doubly degenerate chart.  In its quotient ring the next quantity is
\[
 g_7=-\frac{200}{15309}D^6\eps^8.
\]
The open locus $D\ne0$ is therefore empty, and $D=0$ has already been sent
to the preceding boundary chart.  This is the only place where the seventh
saddle quantity is needed to close the center decomposition.

\subsubsection{Why the center branches cover the full parameter space}
The four constructible root charts
\[
\{a_{20}b_{20}\ne0\},\quad
\{a_{20}=0,b_{20}\ne0\},\quad
\{a_{20}\ne0,b_{20}=0\},\quad
\{a_{20}=b_{20}=0\}
\]
are disjoint and exhaustive.  Within each chart every factor used for a
division is placed in the saturation product $\delta$, while its zero set is
one of the next branches described above.  At every terminal node the exact
certificate verifies
\[
 \sqrt{J:\delta^\infty}
 =\bigcap_q\sqrt{Q_q:\delta^\infty};
\]
an empty terminal node is certified by $1\in J:\delta^\infty$.  Restoring
the diagonal scaling is invertible and sends $g_k$ to a nonzero multiple of
$g_k$.  Consequently all points of $V(I_7)$ occur in F01--F28 or an
explicit exchanged form.  Together with the all-order center constructions
in the companion paper this gives
\[
 V(I_7)=V(I_\infty)=\bigcup_{\mathrm F\nu}V(K_{\mathrm F\nu}).
\]

\section{Weakly persistent isochronous centers: necessity}
\begin{theorem}[Finite necessity and exhaustive isochronous decomposition]\label{thm:iso}
If the origin of \eqref{sys} is a weakly persistent isochronous center, then
\begin{equation}\label{common:iso}a_{21}=b_{21}=0\end{equation}
and at least one of the 21 original-coordinate groups I01--I13 listed in Appendix~\ref{app:iso-list} holds.  Exact decomposition through the seventh resonant pair gives
\[
V(\mathcal L_7)=\bigcup_{\mathrm I\mu}V(K_{\mathrm I\mu}),\qquad
\sqrt{\mathcal L_7}=\bigcap_{\mathrm I\mu}\sqrt{K_{\mathrm I\mu}}
\quad\text{over }\C.
\]
Thus the isochronous necessity list is also exhaustive.
\end{theorem}

\subsection{Linearization recursions and resonant obstructions}
Seek eigenfunctions
\[
U=z+\sum_{n\ge2}U_n,\qquad V=w+\sum_{n\ge2}V_n,
\qquad X_\eps U=U,\quad X_\eps V=-V.
\]
With $U_1=z$, $V_1=w$, and $U_0=V_0=0$, put
\[
R_n^+=\eps(X_2U_{n-1}+X_3U_{n-2}),\qquad
R_n^-=\eps(X_2V_{n-1}+X_3V_{n-2}).
\]
For nonresonant monomials the coefficients are determined by
\[
[z^iw^j]U_n=-\frac{[z^iw^j]R_n^+}{i-j-1},\qquad
[z^iw^j]V_n=-\frac{[z^iw^j]R_n^-}{i-j+1}.
\]
The two resonant coefficients at odd degree are
\[
\ell_k^+=[z^{k+1}w^k]R_{2k+1}^+,\qquad
\ell_k^-=[z^kw^{k+1}]R_{2k+1}^-,
\qquad
\ell_k^\pm=\sum_{j=k}^{2k}\ell_{k,j}^\pm\eps^j.
\]
The first pair is
\begin{align*}
\ell_1^+&=a_{21}\eps-
\left[a_{11}(a_{20}+b_{11})+\tfrac23a_{02}b_{02}\right]\eps^2,\\
\ell_1^-&=-b_{21}\eps+
\left[b_{11}(b_{20}+a_{11})+\tfrac23a_{02}b_{02}\right]\eps^2.
\end{align*}
It gives $a_{21}=b_{21}=0$ and the common quadratic relation.  The lowest
coefficients of the second pair are
\[
[\eps^2]\ell_2^+=-a_{12}a_{30}-a_{12}b_{12}
-\tfrac34a_{03}b_{03},\qquad
[\eps^2]\ell_2^-=b_{12}b_{30}+a_{12}b_{12}
+\tfrac34a_{03}b_{03}.
\]
There are 70 coefficient polynomials through the seventh pair.  Denote by
\[
\mathcal L_N=
\id{\ell_{k,j}^+,\ell_{k,j}^-:1\le k\le N,\ k\le j\le2k}
\]
their successive ideals.

If the first seven pairs vanish, the truncated eigenfunctions give
$X(UV)=\ord(17)$, hence the first seven saddle quantities vanish.  The
complete center decomposition therefore places every point of
$V(\mathcal L_7)$ in one of F01--F28 or an exchanged form.  It is then
legitimate to restrict the linearization quantities to each center ideal.
This reduction uses the already proved finite center identity
$V(I_7)=V(I_\infty)$; it does not assume isochronicity from finite-order
vanishing.

\subsection{Restriction of the linearization quantities to every center family}\label{sec:iso-family-prose}
Let $J_\nu$ be the defining ideal of F$\nu$, including its auxiliary
relations, and let $\delta_\nu$ be the product of the factors declared
nonzero in that family.  For $m\le7$ set
\[
 \mathfrak R_\nu^{(m)}=
 \sqrt{(J_\nu+\mathcal L_m):\delta_\nu^\infty}.
\]
Thus $\mathfrak R_\nu^{(m)}=\id1$ means that the stated open part of
F$\nu$ contains no isochronous point.  When the radical has several
components, every zero factor is retained and followed on its own boundary.
The first three resonant pairs are used initially; pairs four through seven
are evaluated only on the smaller candidate components that survive.

The common first pair gives $a_{21}=b_{21}=0$ and
\[
a_{11}(a_{20}+b_{11})+\frac23a_{02}b_{02}=0,
\qquad
b_{11}(b_{20}+a_{11})+\frac23a_{02}b_{02}=0.
\]
The lowest coefficients of the second pair give
\[
a_{12}a_{30}+a_{12}b_{12}+\frac34a_{03}b_{03}=0,
\qquad
b_{12}b_{30}+a_{12}b_{12}+\frac34a_{03}b_{03}=0.
\]
The statements below are the radicals obtained after reducing all
coefficients of the indicated pairs, rather than conclusions drawn from
these four displayed equations alone.

\paragraph{F01: the Hamiltonian center family.}
Reduction of the first three pairs produces seven components.  Five lie on
axes where one equation becomes triangular or Riccati and reduce to I01,
I01b, I02, or I02b.  The two remaining cubic candidates are separated by
the fourth pair.  On their quotient rings a nonzero reduced obstruction is,
up to exchange,
\[
 [\eps^4]\ell_4^+=\frac{105}{32}a_{03}^2b_{03}^2.
\]
Its two zero factors return to the preceding triangular/Riccati components.
Hence F01 contributes no additional open isochronous family.

\paragraph{F02: the reversible center family.}
Use $(\zeta,\eta)=(z,hw)$ to make the two coefficient sets equal.  Denote
the three quadratic coefficients in the transformed first equation by
$(A,B,C)$.  The radical of the first linearization pair is the union
\begin{align*}
&B+2C=0,\quad 3A-7C=0;\\
&3B-2C=0,\quad 3A+5C=0;\\
&C=B=0;\\
&C=0,\quad A+B=0.
\end{align*}
Reducing the second and third pairs, and then the higher pairs on the cubic
boundaries, gives I11 from the first component, I12 from the second, and
I13 from the nondegenerate pure-cubic continuation.  The last two quadratic
components and their cubic boundaries are contained in I03, I04/I04b, and
I10.  Restoring $h\ne0$ gives exactly the original-coordinate conditions
listed for I11--I13; no point is lost because the diagonal change is
invertible.

\paragraph{F03: the polynomial inverse-integrating-factor family.}
The first three pairs give exactly two radical components.  In original
coordinates they are characterized respectively by
\[
\begin{gathered}
a_{11}=b_{11}=a_{03}=b_{03}=b_{02}=0,\qquad
a_{12}+b_{30}=a_{30}+b_{12}=0,\\
a_{11}=b_{11}=a_{03}=b_{03}=a_{02}=0,\qquad
a_{12}+b_{30}=a_{30}+b_{12}=0,
\end{gathered}
\]
together with the equations already defining F03 and
$a_{21}=b_{21}=0$.  These are I04 and I04b.  All later pairs vanish on both
components.

\paragraph{F04: the Lotka-type center family.}
The first three pairs split according to the auxiliary ratio $t$.  The
one-sided components with $t=1/2$ restore to I05 and I05b.  The symmetric
component with $t=-1$ and its cubic compatibility equation is I10.  The
$t=0$ and $t=-1/2$ boundaries reduce to already listed descendants in
I01--I04 and I06/I06b.  The saturated complement of these components is
the unit ideal, so F04 produces no further family.

\paragraph{F05: the exponential-integral family.}
After the first three pairs, both components force all coefficients of one
nonlinear equation except its autonomous terms to vanish.  They are the
triangular loci I01 and I01b.  There is no open isochronous component of
F05 away from these boundaries.

\paragraph{F06 and F06b: the triangular center families.}
On F06 the radical of the first three pairs is
\[
\sqrt{J_{06}+\mathcal L_3}
=J_{06}+\id{a_{21},b_{21},a_{20},a_{30},b_{11},b_{02},b_{12},b_{03}},
\]
which is I01 after redundant F06 equations are removed.  Exchanging the two
equations gives I01b from F06b.  Pairs four through seven vanish identically.

\paragraph{F07 and F07b: the Riccati center families.}
For F07 the first three pairs add
\[
a_{21}=b_{21}=a_{11}=a_{02}=a_{12}=a_{03}=b_{11}=b_{30}=0.
\]
This is I02.  The exchanged computation gives I02b.  No later obstruction
survives.

\paragraph{F08: the separated center family.}
The reduced radical is
\[
a_{21}=b_{21}=a_{11}=b_{11}=a_{02}=b_{02}
=a_{12}=b_{12}=a_{03}=b_{03}=0,
\]
with the autonomous coefficients unrestricted.  This is precisely I03,
and every later pair vanishes.

\paragraph{F09: the exceptional quadratic family.}
After saturation by $a_{20}b_{20}$, the first three pairs give
$\mathfrak R_{09}^{(3)}=\id1$.  Hence the declared nondegenerate part of F09
has no isochronous point.  Its zero boundaries were assigned to earlier
center families before saturation.

\paragraph{F10: the mixed family at ratio $-1/2$.}
Substitution of the F10 equations gives
$\mathfrak R_{10}^{(3)}=\id1$ on $a_{20}b_{20}\ne0$.  Thus F10 contributes
no isochronous component; its degenerate boundaries occur in one-sided
families.

\paragraph{F11: the mixed family at ratio $1/4$.}
The first three pairs likewise give
$\mathfrak R_{11}^{(3)}=\id1$ after saturation by $a_{20}b_{20}$.  Both
complex square-root branches in F11 are excluded; their zero-cubic
intersection lies on an earlier reversible boundary.

\paragraph{F12: the unified mixed family.}
The first three pairs leave two algebraic closures, according to which
quadratic cross coefficient vanishes.  In either closure all mixed and
self-cubic coefficients appearing in F12 reduce to zero; the remaining
auxiliary equations force the point onto a reversible, separated, or
one-sided boundary.  Direct substitution identifies the survivors with
I01--I04.  Saturation by $a_{20}b_{20}r(7\tau-2)$ leaves no new open
component.

\paragraph{F13: the sextic inverse-integrating-factor family.}
The two components of $\mathfrak R_{13}^{(3)}$ force the cubic coefficients
to vanish and impose, respectively,
\[
e=1,\quad a_{20}^2c=b_{20}b_{02},
\qquad\text{or}\qquad
ce=1,\quad b_{20}^2e=a_{20}a_{02},\quad
a_{20}a_{02}c=b_{20}^2.
\]
Together with the F13 equations these are lower reversible or one-sided
components already contained in I01--I04.  There is no new component on the
open set $a_{20}b_{20}ce(ce^2-1)\ne0$.

\paragraph{F14: the first equal-cubic family.}
The first three pairs force all cubic and mixed coefficients to zero and
then give either
\[
a_{20}a_{02}=b_{20}^2
\qquad\text{or}\qquad
a_{20}^2=b_{20}b_{02}.
\]
Both alternatives are contained in previously obtained reversible or
one-sided I families.  Thus F14 has no new open isochronous component.

\paragraph{F15: the second equal-cubic family.}
The corresponding reduced alternatives are
\[
a_{20}a_{02}=-2b_{20}^2
\qquad\text{or}\qquad
2a_{20}^2=-b_{20}b_{02}.
\]
They are again contained in earlier reversible or one-sided I families, so
F15 produces no new open component.

\paragraph{F16: the reciprocal-parameter mixed family.}
After saturation by $a_{20}b_{20}$ the first three pairs generate the unit
ideal:
\[
\mathfrak R_{16}^{(3)}=\id1.
\]
Its zero boundaries occur in F12--F15 or in a one-sided family and have
already been retained there.

\paragraph{F17: the first zero-self-cubic center family.}
The first three pairs leave two zero-cubic boundaries, selected by
\[
2r+1=0,\quad 3a_{20}^2+4b_{20}b_{02}=0,
\]
or by
\[
r+2=0,\quad 4a_{20}a_{02}+3b_{20}^2=0.
\]
After the F17 equations are used, both loci are contained in I01/I01b or
I02/I02b.  Later pairs introduce no new component.

\paragraph{F18: the second zero-self-cubic center family.}
The corresponding alternatives are
\[
6r-1=0,\quad 5a_{20}^2-36b_{20}b_{02}=0,
\]
or
\[
r-6=0,\quad 36a_{20}a_{02}-5b_{20}^2=0.
\]
After the F18 equations are imposed, these loci are contained in I01/I01b
or I02/I02b.  Later pairs introduce no new component.

\paragraph{F19 and F19b: the one-sided ratio $-1/2$ families.}
On F19 the first three pairs reduce to
\[
a_{20}=a_{02}=a_{12}=a_{03}=a_{21}=0,
\quad b_{11}=b_{30}=b_{03}=b_{21}=0,
\quad 2a_{30}+b_{12}=0,\quad 2a_{11}+b_{20}=0,
\]
after redundant equations are removed.  This is I06; exchange gives I06b.
All later pairs vanish.

\paragraph{F20 and F20b: the one-sided ratio $-1/3$ families.}
The reduced radical adds
\[
a_{20}=a_{02}=a_{30}=a_{12}=a_{03}=a_{21}=0,
\quad b_{11}=b_{02}=b_{30}=b_{12}=b_{21}=0,
\quad 3a_{11}+b_{20}=0.
\]
This is I07, and exchange gives I07b.

\paragraph{F21 and F21b: the one-sided ratio $1/3$ families.}
The first three pairs give
\[
a_{20}=a_{02}=a_{30}=a_{03}=a_{21}=0,
\quad b_{11}=b_{02}=b_{12}=b_{21}=0,
\quad 3a_{11}-b_{20}=0,\quad 3a_{12}+b_{30}=0.
\]
These are the conditions I08; exchange gives I08b.  The seventh-pair check
vanishes identically on both final ideals.

\paragraph{F22 and F22b: the triple-coupling families.}
The first three pairs split F22 into an entirely degenerate triangular
component and a component satisfying
\[
a_{20}=a_{11}=a_{02}=a_{12}=a_{03}=a_{21}=0,
\quad b_{11}=b_{02}=b_{30}=b_{03}=b_{21}=0,
\quad 3a_{30}+b_{12}=0.
\]
The later pairs send both components into I01 or I02 boundaries.  The
exchanged statement holds for F22b.  No new open I family results.

\paragraph{F23 and F23b: the one-sided quadratic cross-coupling families.}
The radical of the first three pairs imposes
\begin{gather*}
a_{20}=a_{11}=a_{02}=a_{03}=a_{21}=0,\qquad
b_{11}=b_{03}=b_{21}=0,\\
a_{12}+b_{30}=a_{30}+b_{12}=0,\qquad
b_{02}b_{30}=b_{20}b_{12}.
\end{gather*}
Using the F23 equations reduces this to I04; exchange gives I04b.

\paragraph{F24 and F24b: the independent-cubic one-sided families.}
On the stated open set the first three pairs give
$\mathfrak R_{24}^{(3)}=\id1$; exchange gives the same result for F24b.
The zero boundary of $a_{02}b_{20}$ belongs to an earlier one-sided family.

\paragraph{F25 and F25b: the one-sided sextic-factor families.}
Saturation by $a_{02}b_{20}r(r-1)$ gives
$\mathfrak R_{25}^{(3)}=\id1$, and the exchanged calculation gives the same
result for F25b.  All zero boundaries of the saturation factor were already
separated in F22--F24.

\paragraph{F26 and F26b: the one-sided mixed rank-boundary families.}
The first three pairs give $\mathfrak R_{26}^{(3)}=\id1$ on the stated open
set, and exchange gives the same result for F26b.  In particular, the
retained value $r=6/7$ does not create an isochronous component.  Every zero
boundary of the saturation factor lies in an earlier one-sided family.

\paragraph{F27 and F27b: the doubly degenerate half-ratio families.}
For F27 the first three pairs add
\[
a_{20}=a_{11}=a_{02}=a_{30}=a_{03}=a_{21}=0,
\quad b_{20}=b_{11}=b_{12}=b_{03}=b_{21}=0,
\quad 2a_{12}+b_{30}=0.
\]
This is I09.  Exchange gives I09b, and all later pairs vanish.

\paragraph{F28: the exceptional homogeneous cubic family.}
Saturation by $a_{30}b_{30}$ gives
\[
\mathfrak R_{28}^{(3)}=\id1.
\]
Thus F28 has no isochronous point on its declared open chart.  Its cubic-axis
boundaries were treated in F21/F21b and F27/F27b.

\subsubsection{Higher-pair exclusions and final coverage}
Most surviving radicals above are final after the first three pairs.  The
few Hamiltonian and reversible cubic closures require later pairs.  The first
new nonzero reductions include
\[
\frac{105}{32}a_{03}^{2}b_{03}^{2},\qquad
-\frac5{64}b_{03}^{4},\qquad
140b_{02}^{2}b_{30}^{3},\qquad
\frac{385}{72}a_{03}^{3}b_{02}^{4},
\]
appearing respectively in $\ell_4^+$, $\ell_4^+$, $\ell_4^+$, and
$\ell_5^+$.  Both zero factors of each expression, and their exchanged
forms, are followed.  They lead to I01 or I02 and their exchanged forms;
their open complements are empty.  Direct reduction of all seventy
coefficient polynomials through pair seven on I01--I13 gives zero.

The preceding paragraphs treat every center family F01--F28 and every
explicit exchanged form.  The exact inclusion certificate contains 109
normal-form checks: each candidate component lies in one of I01--I13, and
each final I ideal contains the restricted linearization ideal.  Therefore
\[
V(\mathcal L_7)=V(\mathcal L_\infty)
=\bigcup_{\mathrm I\mu}V(K_{\mathrm I\mu}),
\]
with 21 explicitly written original-coordinate alternatives.  This proves
that the isochronous decomposition has no omitted center branch.

\section{How the coverage certificate proves that no branch is omitted}
For each chart let $J$ be the ideal generated by the already imposed equations and by the reduced obstruction coefficients, and let $\delta$ be the product of all quantities declared nonzero on that chart.  The exact calculation establishes
\begin{equation}\label{eq:chartcert}
\sqrt{J:\delta^\infty}=\bigcap_{q=1}^{s}\sqrt{Q_q:\delta^\infty}.
\end{equation}
The right-hand ideals are the child charts or final families.  Both inclusions are checked by exact normal-form computations; an empty open chart is certified by $1\in J:\delta^\infty$.  The complementary hypersurface $\delta=0$ is represented by earlier or explicitly generated child charts.  Repeating \eqref{eq:chartcert} down the finite tree proves coverage of the parent chart.

The root identities are elementary and disjoint at the level of constructible sets:
\[
\{a_{20}b_{20}\ne0\}\ \dot\cup\
\{a_{20}=0,b_{20}\ne0\}\ \dot\cup\
\{a_{20}\ne0,b_{20}=0\}\ \dot\cup\
\{a_{20}=b_{20}=0\}.
\]
The first center equations are imposed before this split.  Every subsequent normalization is made only on a chart on which its denominator is in $\delta$.  Exchange of the two equations is an invertible parameter substitution, so it neither creates nor loses points.  Original-coordinate restoration uses the diagonal covariance
\[
p_{ij}=a_{ij}\alpha^{1-i}\beta^{-j},\qquad
q_{ij}=b_{ij}\beta^{1-i}\alpha^{-j},\qquad \alpha\beta\ne0,
\]
under which $g_k$ is multiplied by $(\alpha\beta)^{-k}$.

For centers, the certificate tree contains the nondegenerate, one-sided, doubly degenerate, and pure-cubic leaves followed successively in Section~\ref{sec:center-branches-prose}; a seventh-order obstruction removes the last finite-order residual component.  For isochronicity, every point of $V(\mathcal L_7)$ already belongs to the complete center locus because the truncated eigenfunctions imply vanishing of the first seven center quantities.  Intersecting each center family with the first three resonant pairs produces 23 candidate charts; the fourth through seventh pairs remove the remaining false candidates or send their zero boundaries to the listed I families.  The resulting finite trees prove the two radical identities in Theorems~\ref{thm:center} and~\ref{thm:iso}.  The closure orders are verified sufficient truncations; no claim of minimality is made.

\paragraph{Reproducibility.}
The electronic calculation bundle contains the exact coefficients of $g_1,\ldots,g_6$, the seventh-order closing obstructions, the Singular scripts for every saturation and radical comparison, the normal-form logs, the original-coordinate restoration checks, and the isochronous decomposition through pair seven.  The formulas printed here are the reduced factors used to navigate that certificate tree; the large unreduced polynomials are not substituted into the prose because several occupy many pages individually.